\documentclass[11pt, a4paper]{amsart}
\usepackage{amssymb}
\usepackage{bbm}
\usepackage{a4wide}
\usepackage{enumitem}
\usepackage{xcolor} 
\usepackage[pdftex,
pdftitle={An MIMC method for semilinear SPDEs},
bookmarksopen,
colorlinks,
linkcolor=black,
urlcolor=black,
citecolor=black
]{hyperref}

\usepackage{graphicx}
\usepackage{upgreek}

\usepackage{thm-restate}

\newcommand{\dd}{\,\mathrm{d}}
\newcommand{\E}{\mathbb{E}}
\newcommand{\cC}{\mathcal{C}}
\newcommand{\cD}{\mathcal{D}}
\newcommand{\cE}{\mathcal{E}}
\newcommand{\cF}{\mathcal{F}}
\newcommand{\cG}{\mathcal{G}}
\newcommand{\cH}{\mathcal{H}}
\newcommand{\cI}{\mathcal{I}}
\newcommand{\cJ}{\mathcal{J}}
\newcommand{\cL}{\mathcal{L}}
\newcommand{\cO}{\mathcal{O}}
\newcommand{\N}{\mathbb{N}}
\newcommand{\R}{\mathbb{R}}

\newcommand{\cost}{\mathrm{Cost}}

\newcommand{\boldEll}{{\boldsymbol{\ell}}}
\newcommand{\MIMCest}{\mu_{\textnormal{MI}}}
\newcommand{\MLMCest}{\mu_{\textnormal{ML}}}

\makeatletter
\let\oldHyPsd@CatcodeWarning\HyPsd@CatcodeWarning%
\renewcommand{\HyPsd@CatcodeWarning}[1]{%
  \ifnum\pdfstrcmp{#1}{math shift}=0    %
  \else                                 %
    \oldHyPsd@CatcodeWarning{#1}%
  \fi
}
\makeatother

\makeatletter
\@namedef{subjclassname@1991}{{\normalfont 2020} Mathematics Subject Classification}
\makeatother

\newcommand\Item[1][]{%
	\ifx\relax#1\relax  \item \else \item[#1] \fi
	\abovedisplayskip=0pt\abovedisplayshortskip=0pt~\vspace*{-\baselineskip}}

\newtheorem{theorem}{Theorem}[section]
\newtheorem{lemma}[theorem]{Lemma}
\newtheorem{proposition}[theorem]{Proposition}
\newtheorem{corollary}[theorem]{Corollary}
\theoremstyle{remark}
\newtheorem{remark}[theorem]{Remark}
\theoremstyle{definition}
\newtheorem{example}[theorem]{Example}
\newtheorem{assumption}[theorem]{Assumption}

\usepackage[normalem]{ulem}

\begin{document}

\title[The MIMC method for semilinear SPDEs]{The multi-index Monte Carlo method for semilinear stochastic partial differential equations}

\author[A-L.~Haji-Ali]{Abdul-Lateef Haji-Ali}
\address[A-L.~Haji-Ali]{Maxwell Institute for Mathematical Sciences, Heriot-Watt University, Edinburgh EH14 4AS, UK} \email[]{a.hajiali@hw.ac.uk}

\author[H.~Hoel]{H\aa kon Hoel}
\address[H.~Hoel]{Department of Mathematics, Faculty of Mathematics and Natural Sciences, University of Oslo, 0316 Oslo, Norway} \email[]{haakonah@math.uio.no}

\author[A.~Petersson]{Andreas Petersson}
\address[A.~Petersson]{Department of Mathematics, Faculty of Technology, Linnaeus University, 351 95 V\"axj\"o, Sweden \newline and \newline Department of Mathematics, Faculty of Mathematics and Natural Sciences, University of Oslo, 0316 Oslo, Norway} \email[]{andreas.petersson@lnu.se}

\thanks{The work of A. Petersson was supported in part by the Research Council of Norway (RCN) through project no.\ 274410.}



\begin{abstract}
  Stochastic partial differential equations (SPDEs) are often difficult to
  solve numerically due to their low regularity and high dimensionality. These
  challenges limit the practical use of computer-aided studies and pose
  significant barriers to statistical analysis of SPDEs.
  In this work, we introduce a highly efficient multi-index Monte Carlo method
  (MIMC) designed to approximate statistics of mild solutions to semilinear
  parabolic SPDEs. Key to our approach is the proof of a multiplicative
  convergence property for coupled solutions generated by an exponential
  integrator numerical solver, which we incorporate with MIMC. We further
  describe theoretically how the asymptotic computational cost of MIMC can be
  bounded in terms of the input accuracy tolerance, as the tolerance goes to
  zero. Notably, our methodology illustrates that for an SPDE with low
  regularity, MIMC offers substantial performance improvements over other
  viable methods. Numerical experiments comparing the performance of MIMC with
  the multilevel Monte Carlo method on relevant test problems validate our
  theoretical findings. These results also demonstrate that MIMC significantly
  outperforms state-of-the-art multilevel Monte Carlo, thereby underscoring
  its potential as a robust and tractable tool for solving semilinear
  parabolic SPDEs.
\end{abstract}

\maketitle

\section{Introduction and preliminaries}

Stochastic partial differential equations (SPDEs) are notoriously difficult to
solve numerically due to their low regularity and high dimensionality. These
challenges limit the utility of computer-aided studies and act as a
significant roadblock to statistical analysis of SPDEs. Multilevel
Monte Carlo (MLMC) and its recent extension, the multi-index Monte Carlo
method (MIMC), have demonstrated remarkable success in improving the
tractability of computing statistics of problems involving computationally
expensive approximations. While these methodologies have shown promise when
applied to certain SPDEs, a comprehensive framework and robust theoretical
foundation for MIMC, particularly in the context of semilinear parabolic SPDEs,
remains to be fully established. In this paper, we construct and develop
theory for an MIMC method for weak approximations of SPDEs on bounded domains
with additive nonlinearities and multiplicative affine linear Gaussian noise.
These are formulated as It\^o stochastic differential equations (SDE) on a
Hilbert space $H$:
\begin{equation}
	\label{eq:spde}
	\begin{split}
		\dd X(t) + A X(t) \dd t &= F(X(t)) \dd t + (I + G X(t)) \dd W(t), \quad t \in (0,T], T < \infty, \\
		X(0) &= X_0.
	\end{split}
\end{equation}
Here $W$ is a $Q$-Wiener process in $H$, $A$ is a densely defined,
closed and positive definite operator with a compact inverse and $X_0
\in H$ is a suitable initial condition. We specify our assumptions on
the mappings $F$ and $G$ in Section~\ref{sec:notation} below.

The MLMC method achieves variance reduction and impressive performance by
sampling pairwise coupled realizations on a single-index hierarchy of
numerical resolutions. It was first developed for integral approximations
in~\cite{H01} and subsequently extended to weak approximation of
SDEs~\cite{K05,G08}. MLMC was first applied to SPDEs~\cite{BL12, BLS13} and has
since seen a surge of interest~\cite{LP18,LW19,P20,HKS20,HA23}. The numerical
solver we herein combine with MIMC -- the exponential integrator -- was first
applied to semilinear parabolic SPDEs in~\cite{JK09}, and later combined with
MLMC~\cite{CHJZ22} and with ensemble-based filtering
methods~\cite{JKLZ17,CHLNT21}.

The Multi-Index Monte Carlo method (MIMC) is a recent extension of MLMC that
improves efficiency even further through sampling coupled realizations on a
multi-index hierarchy of resolutions~\cite{HNT16}. The MIMC method was first
applied to high-dimensional PDE with random input and the McKean Vlasov
SDE~\cite{HT18}. It has also been studied for a linear parabolic SPDE with
additive noise for filtering problems~\cite{CDHJ18,JLX21,JLW23}, and
log-factor performance gains have achieved for weak approximations of the
Zakai SPDE~\cite{RW18} (which unlike the herein considered equation is an SPDE
with a 1D driving noise).

The main contribution of this work is the establishment of the multiplicative
convergence property for multi-index-coupled exponential integrator numerical
solutions of the SPDE, as detailed in Theorem~\ref{thm:second-order-diff}.
This essential convergence result demonstrates that the exponential integrator
method is a suitable SPDE solver to combine with MIMC. In
Section~\ref{sec:mimc_method}, we delve into the algorithmic specifics of our
MIMC method, and Theorem~\ref{thm:mimc-cost-error} describes its asymptotic
performance in terms of computational cost relative to the accuracy tolerance
as it approaches zero. Theorems~\ref{thm:mimc-cost-error}
and~\ref{thm:mlmc-cost-error} establish that, for low-regularity SPDEs, our
MIMC method significantly outperforms the state-of-the-art MLMC methods. These
theoretical assertions are substantiated by numerical experiments presented in
Section~\ref{sec:numerical_examples}.

\subsection{Notation and mathematical preliminaries}
\label{sec:notation}
For Banach spaces $U$
and $V$ (all over $\R$), we denote by $\cL(U,V)$ the linear and bounded
operators from $U$ to $V$ and set $\cL(U) := \cL(U,U)$. We write $U \hookrightarrow V$ if $U
\subset V$ and the embedding operator $I_{U \hookrightarrow V}$ fulfills $ \| I_{U \hookrightarrow V}
\|_{\cL(U,V)} < \infty$.

For separable Hilbert spaces $U$ and $V$, $\cL_2(U,V)$ is the space of
Hilbert--Schmidt operators. It is a separable
Hilbert space with inner product defined in terms of the inner
product $\langle \cdot, \cdot \rangle_V$ of $V$,
\begin{equation*}
	\langle \Gamma_1, \Gamma_2\rangle_{\cL_2(U,V)} := \sum_{j=1}^\infty \langle \Gamma_1 e_j, \Gamma_2 e_j \rangle_{V}, \quad \Gamma_1, \Gamma_2 \in \cL_2(U,V),
\end{equation*}
with $(e_j)_{j=1}^\infty$ an arbitrary orthonormal basis (ONB) of
$U$. Moreover,  it fulfills the ideal property
\begin{equation*}
	\|\Gamma_1 \Gamma_2\|_{\cL_2(U,V)} \le \|\Gamma_1\|_{\cL(H,V)} \|\Gamma_2\|_{\cL_2(U,H)}, \quad \Gamma_1 \in \cL(H,V), \Gamma_2 \in \cL_2(U,H),
\end{equation*}
where $H$ is an additional Hilbert space.  We employ the notation $\cL_2(U):=\cL_2(U,U)$. For a more detailed introduction to this class of operators, see \cite[Appendix~B]{PR07}.

We say that a
continuous (typically non-linear) mapping $F \colon U \to V$ belongs
to the class $\cG^1(U,V)$ of Gateaux differentiable mappings if, for
all $u,w \in U$, the Gateaux derivative $F'(u)w := \lim_{\epsilon
  \to 0} \epsilon^{-1} (F(u+\epsilon w)-F(u))$ exists in $V$, $F'(u)
\in \cL(U,V)$ and $F'(\cdot)w \colon U \to V$ is continuous. If, for
all $u,w_1,w_2 \in U$, it also holds that $F''(u)(w_1,w_2) :=\lim_{\epsilon \to 0} \epsilon^{-1} (F'(u+\epsilon w_2)-F'(u))w_1$
exists in $V$, $F''(u) \colon U \times U \to V$ is a symmetric and
bounded bilinear mapping and $F''(\cdot)(w_1,w_2) \colon U \to V$ is
continuous, then $F$ is said to belong to the class $\cG^2(U,V)$. Note
that if there is a Hilbert space $\tilde V$ such that $V
\hookrightarrow \tilde V$, then $F \in \cG^k(U,V)$ implies $F \in
\cG^k(U,\tilde V)$, $k \in \{1,2\}$. For $F \in \cG^1(U,V)$ and
$u,w\in U$, the map $\phi \colon \R \to V$ given by $\phi(t) = F(w +
t(u-w))$ is easily seen to be in $\cG^1(\R,V)$. We can
identify its derivative at $t$, $\phi'(t)$, with an element in $V$,
namely $F'(w + t (u-w))(u-w)$. Since $\phi'$ is continuous and
separably valued, it is Bochner integrable on $[0,1]$. By the fact
that bounded linear operators and Bochner integrals commute along with
the definition of the Gateaux derivative, it follows that
\begin{equation*}
	\left\langle \int^1_0 \phi'(t) \dd t, v \right\rangle_V = \int^1_0 \langle \phi'(t), v \rangle_V \dd t = \int^1_0 (\langle \phi(\cdot), v \rangle_V)'(t) \dd t = \langle \phi(1)-\phi(0),v \rangle_V, \quad v \in V,
\end{equation*}
From this we obtain the mean value theorem
\begin{equation*}
	F(u) - F(w) =  \int^1_0 F'(w + t (u-w))(u-w) \dd t.
\end{equation*}
If  $F\in \cG^2(U,V)$, then, for all $u \in U$,  the mapping $w \mapsto F'(w)u$ is by definition in $\cG^1(U,V)$ with derivative given by $F''(w)(u,\cdot) = F''(w)(\cdot,u)$. The mean value theorem then yields
\begin{equation*}
	\big(F'(w_1) - F'(w_2)\big)u = \int^1_0 F''(w_2 + t (w_1-w_2))(w_1-w_2,u) \dd t. %
\end{equation*}
We refer to \cite{AP95} and~\cite{AKL16} for further details on Gateaux differentiability.

We employ generic constants $C$. These vary between occurrences and
are independent of discretization parameters. For a pair of
non-negative sequences $(a_k)$ and $(b_k)$, we write
$a_k \lesssim b_k$ to signify that there exists a constant $C< \infty$
such that $a_k \le C b_k$ holds for all $k$, and $a_k \eqsim b_k$
signifies that the sequences are equivalent in the sense that both
$a_k \lesssim b_k$ and $b_k \lesssim a_k$.

From here on, we fix a real separable Hilbert space $H = (H,\langle
\cdot, \cdot \rangle, \| \cdot\|)$ and let $((e_j, \lambda_j))_{j=1}^\infty$ denote the sequence of eigenpairs
associated to the operator $A$. We let
$$
\dot{H}^r := \{v \in H : \sum_{j=1}^\infty \lambda^r_j |\langle v, e_j \rangle|^2  < \infty \}, \quad r \ge 0,
$$
and let $\langle \cdot, \cdot \rangle_{\dot{H}^r} := \langle A^{r/2}
\cdot, A^{r/2} \cdot \rangle$ denote the standard graph inner
product. Here, $A^{r/2} \colon \dot H^{r} \to H$ is defined by
\begin{equation}
		A^{r/2} v = \sum_{j=1}^\infty \lambda^{r/2}_j \langle v, e_j \rangle, \quad v \in \dot{H}^r.
	\end{equation}
For $r<0$ this expression defines a bounded operator on $H$ and $\dot{H}^r$ then denotes the completion of $H$ under the aforementioned graph norm. For all
$r,s \in \R$, we may extend $A^{r/2}$ to an operator in
$\cL(\dot{H}^s, \dot{H}^{s-r})$ and we do so without changing
notation. For further details on these types of spaces, see~\cite[Appendix~B]{K14}.

The remainder of this paper is structured as follows:
Section~\ref{sec:mimc_method} outlines the problem assumptions,
introduces the exponential integrator numerical method, and provides
both the algorithmic framework and theoretical results for our MIMC
method. In Section~\ref{sec:convergence_analysis}, we examine the
convergence properties of multi-index-coupled solutions using the
exponential integrator method. Section~\ref{sec:numerical_examples} is
dedicated to a numerical performance comparison between MLMC and
MIMC. We conclude with a summary of our findings and a discussion of
open problems in Section~\ref{sec:conclusion}. Finally, Appendix~\ref{sec:appendix}
presents norm estimates for Gateaux derivatives of a class of
nonlinear composition operators that are relevant for the
scope of our MIMC method.

\section{The MIMCEI method for semilinear SPDEs}
\label{sec:mimc_method}

Let $(\Omega, \cF, P)$ be a probability space with filtration $(\cF_t)_{t \ge 0}$,  satisfying the usual conditions. Let $W$ be a cylindrical $Q$-Wiener \((\cF_t)_{t \ge 0}\)-adapted process in $H$. We assume that its covariance operator $Q$ commutes with $A$, so that in terms of the eigenfunctions $(e_j)_{j=1}^\infty$ of $A$,
\begin{equation*}
	W(t) = \sum_{j = 1}^{\infty} \mu_j^{\frac 1 2} B_j(t) e_j, \quad t \ge 0,
\end{equation*}
for an iid sequence $(B_j)_{j=1}^\infty$ of Brownian motions, and a non-negative sequence $(\mu_j)_{j=1}^\infty \in \ell^1$. This allows for defining fractional powers of \( Q \) analogously to those of \( A \). The expansion is formal and converges in a suitably chosen space.

We let the operator $G$ be of the form
\begin{equation}
	\label{eq:G-form}
	(Gu)v = \sum_{j = 1}^\infty \zeta_j \langle u, e_{j+m}\rangle \langle v, e_j \rangle e_j, \text{ for some sequence } (\zeta_j)_{j=1}^\infty \subset \R \text{ and } m \in \N_0.
\end{equation}
In the following assumption, we specify the regularity conditions that this operator and the other parts of~\eqref{eq:spde} should satisfy.
\begin{restatable}{assumption}{mainassump}
  \label{ass:reg} For some $\eta \in [0,2)$, let $F \in \cG^1(H,H) \cap \cG^2(H,\dot{H}^{-\eta})$ and let $G$ be given by~\eqref{eq:G-form}. Assume further that for some $\nu > 0, \kappa \in (0,2), \delta \in (0,1)$ and $C<\infty$,
	\begin{enumerate}[label=(\roman*)]
		\item \label{ass:eigenvalue-growth} $\lambda_{k} \eqsim k^{\nu}$,
		\item \label{ass:X0} $X_0 \in L^{10}(\Omega,\dot{H}^\kappa)$ is $\cF_0$-measurable,
		\item \label{ass:Q} $|\zeta_j| \mu_j^{1/2} \le C \lambda_j^{\frac{1-\kappa-\delta}{2}}$ and $\sum^\infty_{j=1} \lambda_j^{\kappa-1} \mu_j \le C$,
		\item \label{ass:derivative-bounded} $\|F'(u)v\|\le C\|v\|$ for all $u,v \in H$,
		\item \label{ass:difference-regularity-transfer}
		$\|F(u)-F(v)\|_{\dot{H}^{\kappa}} \le C (1 + \|u\|^2_{\dot{H}^\kappa} + \|v\|^2_{\dot{H}^\kappa})$
		for all $u,v \in \dot{H}^\kappa$,
		\item \label{ass:derivative-negnorm} $\|F'(u)v\|_{\dot{H}^{-\eta}}\le C(1 + \|u\|^{\lceil \kappa \rceil}_{\dot{H}^{\kappa}}) \| v \|_{\dot{H}^{-\kappa}}$ for  all $u \in \dot{H}^{\kappa}, v \in H$ and
		\item \label{ass:second-derivative-bounded} $\|F''(u)(v_1,v_2)\|_{\dot{H}^{-\eta}} \le C \| v_1 \| \|v_2\|$ for all $u, v_1, v_2 \in H$.
                \end{enumerate}
\end{restatable}

We note some important facts related to this assumption. First, Assumption~\ref{ass:reg}\ref{ass:derivative-bounded} implies that $F$ is Lipschitz and of linear growth on $H$, i.e.,
\begin{equation}
	\label{eq:F-lip-lin-growth}
	\|F(u)-F(v)\| \le C \| u - v \| \quad \text{and} \quad \|F(u)\| \le C (\|u\| + 1)
\end{equation}
for some constant $C < 0$ and all $u,v \in H$. Next, the first condition on $\mu_j$ in Assumption~\ref{ass:reg}\ref{ass:Q} ensures that $G \in \cL(H,\cL_2(Q^{1/2}(H),\dot{H}^{\kappa+\delta-1}))$ while the second is equivalent to the more familiar expression $\| I_{Q^{1/2}(H) \hookrightarrow \dot{H}^{\kappa-1}} \|_{\cL_2(Q^{1/2}(H),\dot{H}^{\kappa-1})} = \| A^{\frac{\kappa-1}{2}}Q^{\frac 1 2}\|_{\cL_2(H)} < \infty$. Moreover, if the spectral shift parameter $m$ is non-zero, $G$ does not satisfy the classic symmetry condition (see~\cite[Assumption~3, Remark~1]{JR15}) required for SPDE Milstein schemes without requiring expensive simulations of iterated integrals. Finally, as a consequence of Assumption~\ref{ass:reg}\ref{ass:eigenvalue-growth}, we have $\lambda_{j+m} \simeq \lambda_j$ for fixed $m$.

\begin{example}
	\label{ex:heat}
	Let $H=L^2(\cD)$ for a domain $\cD \subset \R^d$, $d=1,2,3$, which is either convex or has $\cC^2$ boundary $\partial \cD$. Let $A = - \Delta$ be the negative Laplacian with homogeneous zero Dirichlet boundary conditions. The eigenvalues then satisfy \(\lambda_k \eqsim k^{2/d}\). Let $F$ be a composition operator, i.e., $F(u)(x) = f(u(x))$ for a.e.\ $x \in \cD$. Provided that $f$ is three times continuously differentiable with bounded derivatives (note that $f$ itself need not be bounded), then~\eqref{eq:F-lip-lin-growth}, Assumptions~\ref{ass:reg}\ref{ass:derivative-negnorm}-\ref{ass:second-derivative-bounded} are satisfied with $\eta \in (d/2,2)$ and $\kappa \in [0,\eta] \setminus \{1/2,3/2\}$. This is shown in Appendix~\ref{sec:appendix}. We note in particular that the subtraction in the left hand side of Assumption~\ref{ass:reg}\ref{ass:difference-regularity-transfer} is needed for the Dirichlet boundary conditions to be included.

	Results on fractional differential operators may be deduced from those of Appendix~\ref{sec:appendix}. Let $A_r := A^{r}$ for some $r \in (0,1)$, and write $\dot{H}^s_r, s \in \R,$ for the spaces introduced in Section~\ref{sec:notation} with $A$ replaced by $A_r$. Then clearly $\dot{H}^s_r = \dot{H}^{rs}$ and Assumptions~\ref{ass:reg}\ref{ass:derivative-bounded}-\ref{ass:second-derivative-bounded} may be derived from this fact. The condition on $\eta$ implies the restriction $r > d/4$. Similarly, if we define an operator $A_\eta$ by specifying its eigenvalues on the eigenbasis $(e_j)_{j=1}^\infty$ of $A$ to be $(j^\nu)_{j=1}^\infty$, then we must take $\nu \in (d/2,2]$.

	In principle, we can consider much more general self-adjoint elliptic operators and fractional powers thereof, as long as the coefficients are sufficiently smooth. We only require that a norm equivalence~\eqref{eq:sobolev_dot_equivalence} in terms of fractional Sobolev spaces and the eigenvalue decay condition above are fulfilled. Other types of boundary conditions are also possible. We refer to \cite[Chapter~16]{Y10} for further details.
\end{example}

\begin{remark}
	\label{rem:remark-G}
	The multiplicative noise operator $G$ is here assumed to take on a particular form. It is possible to generalize the analysis to finite linear combinations of operators of the type we consider. It is even possible to treat nonlinear operators of the form
	\begin{equation*}
		G(u)v = \sum_{j = 1}^\infty  \zeta_j(\langle u, e_{j+m}\rangle) \langle v, e_j \rangle e_j,
	\end{equation*}
	provided that the functions $(\zeta_j)_{j=1}^\infty $ are sufficiently smooth and map $0$ to $0$. However, it is key that $G$ acts on the spectrum of $A$. With regards to the previous example, we cannot treat general composition operators of the form
	$(G(u)v)(x) = g(u(x)) v(x)$ for some function $g \colon \R \to \R$. We return to this point in Remark~\ref{rem:remark-G-nemytskij}.
\end{remark}

Note that $-A$ generates an analytic semigroup $S=(S(t))_{t\ge 0}$  of linear operators on $H$, such that $S(t) e_j = e^{-\lambda_j t} e_j$ for all $t \ge 0$ and $j \in \N$. This is used to define the mild solution of~\eqref{eq:spde} by
\begin{align*}
	X(t) &:= S(t) X_0 + \int^t_0 S(t-s) F(X(s)) \dd s + \int^t_0 S(t-s) \big(I + G X(s) \big) \dd W(s), \quad t \in [0,T].
\end{align*}
Under Assumption~\ref{ass:reg}, existence and uniqueness of $X$ follows from a fixed point argument as in \cite[Theorem~7.2]{DPZ14}.

\subsection{The exponential integrator method}

Consider a uniform time step $\tau \in (0,1]$ satisfying $T/\tau := M \in \N$ and let $t_j := j \tau$ for $j \in \N$. Let $P_N \in \cL(H)$ denote the projection onto the span of $e_1, \ldots, e_N$, $N \in \N$. It is uniformly bounded in $\cL(H)$ and commutes with $A$ and $G$; we use these facts without making explicit reference to them.

The exponential integrator approximation $X^M_N$ of~\eqref{eq:spde} is now given by $X^M_N(0) := P_N X_0$ and for $j =0, \ldots, M-1$ by
\begin{equation}
	\label{eq:exp-int-def}
	\begin{split}
	X^M_N(t_{j+1}) &:= S(\tau) X^M_N(t_{j}) + \int^{t_{j+1}}_{t_j} S(t_{j+1}-s) P_N F(X^M_N(t_{j})) \dd s\\
	&\quad + \int^{t_{j+1}}_{t_j} S(t_{j+1}-s)\big(P_N + G X^M_N(t_{j})\big)\dd W(s).
\end{split}
\end{equation}
Note that
\begin{align*}
	&\Big\langle \int^{t_{j+1}}_{t_j} S(t_{j+1}-s)\big(P_N + G X^M_N(t_{j})\big)\dd W(s), e_k \Big\rangle \\
	&\quad= \mu_k^{\frac 1 2}(1 + \zeta_k \langle X^M_N(t_{j}), e_{k+m} \rangle) \int^{t_{j+1}}_{t_j} e^{-\lambda_k(t_{j+1}-s)} \dd B_k(s)
\end{align*}
if $k \le N$ and $0$ otherwise, meaning the stochastic term in \eqref{eq:exp-int-def} can be sampled exactly. With $\lfloor t \rfloor_{\tau} = t_j$ for $t \in [t_j, t_{j+1})$, we extend the approximation to a continuous process by
\begin{align*}
	X_N^M(t) &:= S(t) P_N X_0 + \int^t_0 S(t-s) P_N F(X^M_N(\lfloor s \rfloor_{\tau})) \dd s \\
	&\quad+ \int^t_0 S(t-s) \big(P_N + G X^M_N(\lfloor s \rfloor_{\tau}) \big) \dd W(s), \quad t \in [0,T].
\end{align*}
Note the dependence of $\tau$ on $M$ in $\lfloor s \rfloor_{\tau} = \lfloor s \rfloor_{M^{-1}}$.

\subsection{The multi-index Monte Carlo method}\label{subsec:mimc-description}
In this section, we present a multi-index Monte Carlo method for
approximations of $\E[\Psi(X)]$, where $X$ is a mild solution of the
SPDE~\eqref{eq:spde} and $\Psi(X)$ denotes a quantity of interest (QoI). For a
real separable Hilbert space $U$ and a mapping $\psi: H \to U$ that is further
described in Assumption~\ref{ass:qoi} below, we consider two kinds of QoIs, the first are mappings $\Psi:L^{2}([0,T], L^{2}(\Omega, H)) \to L^{2}(\Omega, U)$ defined by
\begin{equation}\label{eq:QoI_def1}
  \Psi(X) = \int_0^T \psi(X(s)) \dd s,
\end{equation}
and $\Psi: L^2(\Omega,H) \to L^{2}(\Omega, U)$ defined by
\begin{equation}\label{eq:QoI_def2}
  \Psi(X) = \psi(X(T))\,.
\end{equation}
Although $\Psi(X(T))$ would have been a more appropriate notation
for the latter mapping, we will for the sake of a streamlined
presentation employ the same notation $\Psi(X)$ for both types of
QoIs.

Since the main motivation for hierarchical Monte Carlo methods is
computational efficiency, a cost and error analysis is also included,
where we define the computational cost of an action as the number of
arithmetic operations (additions, subtractions, multiplications and
divisions) and draws of random variables the action requires. We
impose the following constraints on the regularity of the QoI and the
computational cost of QoIs and numerical solutions of the SPDE:
\begin{assumption}\label{ass:qoi}
  Let $\psi:H \to U$ be a twice Gateaux differentiable mapping
  and let there exists a constant $C < \infty$ such that
  \begin{enumerate}[label=(\roman*)]
  	\item \label{ass:qoi:derivative-bounded} $\|\psi'(u)v\|_U\le C\|v\|$ for all $u,v \in H$ and
  	\item \label{ass:qoi:second-derivative-bounded} $\|\psi''(u)(v_1,v_2)\|_U \le C \| v_1 \| \|v_2\|$ for all $u, v_1, v_2 \in H$.
          \item The QoI $\Psi$ is either defined by~\eqref{eq:QoI_def1} or by~\eqref{eq:QoI_def2}.

  \end{enumerate}

  For any integers $M,N \ge 0$ some \(\mathfrak{l} \in \lbrace 0,1 \rbrace\), the computational cost of
  sampling the SPDE solution, \(X_{M}^{N}\) and evaluating the QoI is bounded
  by
  \begin{enumerate}[label=(\roman*),resume]
  \item \label{eq:cost-spde} $\cost(\Psi(X^M_N)) \le C M N (\log_2(N))^{\mathfrak{l}}$.
  \end{enumerate}
\end{assumption}

\begin{example}
	Clearly any linear mapping $\psi \in \cL(H)$, such as $\psi = I$, satisfies Assumption~\ref{ass:qoi} for \(U = H\). Moreover, if $U = \R$ in the context of Example~\ref{ex:heat}, we may define a mapping $\psi \colon H \to \R$ by
	\begin{equation*}
		\psi(u) = \int_{\cD} \varphi(u(x)) \dd x.
	\end{equation*}
	If $\varphi \colon \R \to \R$ is twice differentiable function with bounded first and second derivatives, its first and second derivative are given by
	\begin{align*}
		\psi'(u) v_1 = \int_{\cD} \varphi'(u(x))v_1(x) \dd x  \text{ and } \psi''(u) (v_1,v_2) = \int_{\cD} \varphi''(u(x))v_1(x) v_2(x) \dd x, \hspace{0.5em} u,v_1,v_2 \in H,
	\end{align*}
        and the resulting QoI, $\Psi$, satisfies Assumption~\ref{ass:qoi}.
\end{example}

Assuming the cost of sampling \(X_{N}^{M}\) is the dominant one,
sufficient conditions for the exponential integrator method satisfying
Assumption~\ref{ass:qoi}\ref{eq:cost-spde} is described
in~\cite[Assumption 1]{CHJZ22}. We allow \(\mathfrak{l}=1\) or \(0\)
when the approximation of \(X^M_N\) requires an FFT computation or
not, respectively.

For a given $\varepsilon \in (0,1)$ we seek to approximate $\E[\Psi(X)]$ with $\cO(\varepsilon^2)$ mean square error.
For an index set $\cI \subset \N_0^2$ and set of integers
$(m_{\boldEll})_{\boldEll \in \cI} \subset \N$,
and for two exponentially increasing sequences of integers $(M_k)_{k\ge0}$ and $(N_k)_{k\ge 0}$,
the MIMCEI method is defined as
\begin{equation}\label{eq:mimcei}
\MIMCest := \sum_{\boldEll \in \cI} \sum_{i=1}^{m_\boldEll} \frac{\Delta_{\boldEll} \Psi(X)^{(\boldEll, i)}}{m_\boldEll}.
\end{equation}
Here,
\begin{equation}\label{eq:mimc-diff-def}
\Delta_\boldEll \Psi(X) = \begin{cases}
  \Psi(X_{N_{\ell_2}}^{M_{\ell_1}})  - \Psi(X_{N_{\ell_2}}^{M_{\ell_1-1}}) - \Psi(X_{N_{\ell_2-1}}^{M_{\ell_1}})
  + \Psi(X_{N_{\ell_2-1}}^{M_{\ell_1-1}}) & \text{if} \quad \boldEll \in \N^2\\
  \Psi(X_{N_{\ell_2}}^{M_{\ell_1}})  - \Psi(X_{N_{\ell_2-1}}^{M_{\ell_1}})  & \text{if} \quad \ell_1 =0, \ell_2>0 \\
  \Psi(X_{N_{\ell_2}}^{M_{\ell_1}})  - \Psi(X_{N_{\ell_2}}^{M_{\ell_1-1}})  & \text{if} \quad \ell_2 =0, \ell_1>0 \\
  \Psi(X_{N_{\ell_2}}^{M_{\ell_1}})     & \text{if} \quad \ell_1=\ell_2 =0
  \end{cases}
\end{equation}
with $\boldEll= (\ell_1,\ell_2)$ and $\Delta_\boldEll \Psi(X)^{(\boldEll,i)}$ are
independent $U$-valued random variables for all $i=1,2, \ldots, m_\boldEll$ and
$\boldEll \in \cI$. The index set will take the shape of right-angled triangle
that is described in the proof of Lemma~\ref{lem:general-mimc-cost-error}. It
must in particular satisfy the telescoping-sum constraint that if $(\ell_1, \ell_2)\in \cI$, then the nearest indices to the west, south and southwest also be
contained in $\cI$, i.e.,
\begin{equation}\label{eq:telescoping-constraint}
(\ell_1, \ell_2) \in \cI \implies \{(\max(\ell_1-1,0), \ell_2), (\ell_1, \max(\ell_2-1,0)), (\max(\ell_1-1,0), \max(\ell_2-1,0))\} \subset \cI.
\end{equation}
This is crucial for the weak convergence of MIMCEI, cf.~\eqref{eq:exp-bound-mimc}.

The next lemma and corollary show that double differences of numerical
solutions of the exponential integrator method in many settings satisfy a
multiplicative convergence property.

\begin{lemma}\label{lem:strong-error-coupling}
  If Assumption~\ref{ass:reg}, for some $\kappa \in (0, 2), \nu>0$, and
  Assumption~\ref{ass:qoi} hold, then there exists a constant $C< \infty$ such that
  the exponential integrator method satisfies the following inequality for all
  integers $\overline{M} \ge M$ and $K \ge N$, where $K = DN$ for some $D \in \N$:
 \begin{equation}
 	\label{eq:strong-error-coupling}
   \begin{split}
    \|\Psi(X_K^{\overline{M}})  - \Psi(X_N^{\overline{M}})  - &\Psi(X_K^M) + \Psi(X_N^M) \|_{L^2(\Omega, U)}^2\\
    & \quad  \le \begin{cases}
      {C \min(M^{-2\kappa},\lambda_N^{-\kappa}M^{-\kappa})}  & \text{if} \quad \kappa \in (0,1/2)\\
      C \min\Big( \lambda_{N}^{-(2 \kappa -1)} M^{-1},\, {\lambda_N^{-\kappa}M^{-\kappa}} \Big)  & \text{if} \quad \kappa \in [1/2,1)\\
      C \lambda_{N}^{-\kappa} M^{-1} & \text{if} \quad \kappa \in [1,2)\,.
    \end{cases}
    \end{split}
  \end{equation}
  and for some constant \(C_{N} < \infty\) depending on \(N\),
  \begin{equation}\label{eq:strong-error-coupling-single-lvl}
   \begin{split}
     \|\Psi(X_K^{M})  - \Psi(X_N^{M}) \|_{L^2(\Omega, U)}^2  &\le C \lambda_{N}^{-\kappa},\\
     \|\Psi(X_N^{M})  - \Psi(X_N^{\overline{M}}) \|_{L^2(\Omega, U)}^2  &\le \min( C_{N} {M^{-1}}, \, C M^{-\min(\kappa,1)} ) \,  .\\
    \end{split}
  \end{equation}
  If additionally $\psi \in \cL(H,U)$, then it also holds that
  \begin{equation}\label{eq:strong-error-coupling2}
    \|\Psi(X_K^{\overline{M}})  - \Psi(X_N^{\overline{M}})  - \Psi(X_K^M) + \Psi(X_N^M) \|_{L^2(\Omega, U)}^2
    \le C \lambda_N^{-2\kappa}.
\end{equation}
\end{lemma}
\begin{proof}
See Appendix~\ref{sec:mimc-proofs}.
\end{proof}

The following corollary, which is a direct consequence of
Lemma~\ref{lem:strong-error-coupling} and Jensen's inequality, bounds the
\(L^{2}(\Omega, U)\) norm and the expectation of \(\Delta_{\boldEll} \Psi(X)\). These
bounds will be used in the following theorem.
\begin{corollary}[Mixed difference convergence]\label{coro:mix-diff}
  Let Assumption~\ref{ass:reg}, for some $\kappa \in (0, 2), \nu>0$, and
  Assumption~\ref{ass:qoi} hold and let $M_k \eqsim N_k \eqsim 2^k$. Then
  there exist a constants $C< \infty$ such that for all \(\boldEll \in \N_{0}^{2}\),
  \begin{equation}\label{eq:strong-error-coupling-delta}
    \E[\|\Delta_{\boldEll} \Psi(X) \|_{U}^{2}]  \le
    C 2^{-\min(\kappa, 1)\,\ell_{1}- \kappa \nu\, \ell_{2}} \begin{cases}
      \min(2^{-\kappa (\ell_{1}-\nu \ell_{2})}, %
      1) & \kappa \in (0, 1/2)\\
      \min(2^{-(1-\kappa) (\ell_{1}-\nu \ell_{2})}, %
      1) & \kappa \in [1/2, 1) \\
      1 & \kappa \in [1, 2), \\%
    \end{cases}
  \end{equation}
  and, for \(\alpha_{1} \geq \min(\kappa,1)/2, \alpha_{2} \geq \kappa\nu/2\),
  \begin{equation}\label{eq:weak-error-coupling-delta}
    \|\E[\Delta_{\boldEll} \Psi (X)] \|_{U}  \le
    C 2^{-\alpha_{1}\,\ell_{1}- \alpha_{2} \nu\, \ell_{2}}
    \begin{cases}
      \min(2^{-\kappa (\ell_{1}-\nu \ell_{2})/2}, %
      1) & \kappa \in (0, 1/2)\\
      \min(2^{-(1-\kappa) (\ell_{1}-\nu \ell_{2})/2}, %
      1) & \kappa \in [1/2, 1) \\
      1 & \kappa \in [1, 2), \\%
    \end{cases}
  \end{equation}
\end{corollary}

We now present the main results on the performance of the MIMCEI method.
\begin{theorem}\label{thm:mimc-cost-error}
  Let Assumption~\ref{ass:reg}, for some $\kappa \in (0, 2), \nu>0$, and
  Assumption~\ref{ass:qoi} hold, let $M_k \eqsim N_k \eqsim 2^k$ and recall
  the parameters $\alpha_1, \alpha_{2} \ge 0$ from \eqref{eq:weak-error-coupling-delta}. Then,
  for any $\varepsilon \in (0,1)$, there exist MIMCEI parameters $\cI$ and $(m_{\ell})_{\ell \in
    \cI}$ that depend on $\varepsilon$, %
  such that \( \|\MIMCest - \E[\Psi(X)]\|_{L^2(\Omega, U)}^2 = \cO(\varepsilon^2) \) and
  when \(\kappa \in [1,2)\), we have
  \[
    \cost(\MIMCest) %
    = \begin{cases}
      \cO(\varepsilon^{-2} \, \lvert \log(\varepsilon^{-1}) \rvert^{2}) & \kappa \nu > 1, \\
      \cO(\varepsilon^{-2} \, \lvert \log(\varepsilon^{-1}) \rvert^{4 + \mathfrak l}) & \kappa \nu = 1, \\
      \cO(\varepsilon^{-2-\frac{1- \kappa\nu}{\alpha_{2}}} \, \lvert \log(\varepsilon^{-1}) \rvert^{\mathfrak l}) & \kappa \nu < 1. \\
    \end{cases}
  \]
  Moreover, when \(\kappa \in (0,1)\), then \(\cost(\MIMCest) = \cO(\varepsilon^{-2-2u} \,
  \lvert \log(\varepsilon^{-1}) \rvert^{4 + \mathfrak l + 2u})\), where
  \[
    2 u = \max\left (0, \frac{1-\kappa}{\alpha_{1}}, \frac{1- \kappa\nu}{\alpha_{2}}
    \right).
  \]
\end{theorem}
\begin{proof}
  Given the bounds in Corollary~\ref{coro:mix-diff}, the result follows
  directly from Lemma~\ref{lem:general-mimc-cost-error}, see also \cite[Theorem 2.2]{HNT16}.
\end{proof}

\begin{theorem}\label{thm:mimc-cost-error-extra}
  Under the same setup as Theorem~\ref{thm:mimc-cost-error}, when
  \(\kappa \in (0,1)\), we have the additional bound,
  \(\cost(\MIMCest) = \cO(\varepsilon^{-2 - 2u}\, |\log_2 \varepsilon|^{4+ \mathfrak l+2u }) \),
  where
  \[
    2u = \max\left(0, \frac{1-\min(1, 2\kappa)}{\alpha_{1}+ \kappa/2}, \frac{1-\kappa\nu}{\alpha_{2}},
      \frac{1+\nu (1-2\kappa)} {\alpha_{2} + \nu \alpha_{1}} \right) \leq \frac{1-\kappa}{\alpha_{1}}.
  \]
\end{theorem}
\begin{proof}
  Given the bounds in Corollary~\ref{coro:mix-diff}, with the improved rates
  with respect to time-discretization, the result follows directly from
  Lemma~\ref{lem:general-mimc-cost-error-min}.
\end{proof}

The MIMCEI method is an extension of the multilevel Monte Carlo method exponential integrator
(MLMCEI) studied in~\cite{CHJZ22,CHLNT21,JKLZ17}
\[
\MLMCest :=  \sum_{\ell =1}^L \sum_{{i_\ell} =1}^{m_\ell} \frac{\Psi(X^{M_\ell}_{N_\ell})^{i_\ell} - \Psi(X^{M_{\ell-1}}_{N_{\ell-1}})^{i_\ell} }{m_\ell} +
\sum_{{i_0} =1}^{m_0} \frac{\Psi(X^{M_0}_{N_0})^{i_0}}{m_0},
\]
where $L\ge0$, $(m_\ell)$ is a positive sequence, and $(M_\ell)$ and $(N_\ell)$ are
exponentially increasing sequences, e.g., $M_k \eqsim N_k \eqsim 2^k$. For the
sake of performance comparison, we next present a cost versus error result for
MLMCEI. The result is in the spirit of~\cite[Theorem 2]{CHJZ22}, but for the
regularity setting considered in this work. 
\begin{theorem}\label{thm:mlmc-cost-error}
  Let Assumption~\ref{ass:reg}, for $\kappa \in (0, 2), \nu>0$, and
  Assumption~\ref{ass:qoi} hold. Then there exist $\alpha_1 \ge \min(1,\kappa)/2$ and $\alpha_2 \ge
  \kappa \nu /2$ and a constant $C< \infty$ such that
  \begin{equation}\label{eq:weak-assumption-mlmc}
    \|\E[\Psi(X_K^{\overline{M}})  - \Psi(X_N^M)] \|_U \le C  (M^{-\alpha_1} + N^{-\alpha_2}),
  \end{equation}
  for all integers $K \ge N$ and $ \overline M \ge M$,
  and for any $\varepsilon \in (0,1)$, there exist MLMCEI parameters
  $L$, $(M_k)$, $(N_k)$ and $(m_{\ell})_{\ell=0}^L$ that depend on $\varepsilon$ such that
  \[
    \|\MLMCest - \E[\Psi(X)]\|_{L^2(\Omega, U)}^2 = \mathcal O(\varepsilon^2)
  \]
  and 
  \[
  \cost(\MLMCest) = \begin{cases}
      \cO\left(\varepsilon^{-2 - \big(1 + \nu(1-\kappa) \big)/\min(\nu \alpha_1, \,\alpha_2)} |\log_2 \varepsilon | \right)
    & \text{if} \quad \kappa \in (0,1)\\
      \cO\left(\varepsilon^{-2 - 1/\min(\kappa \nu \alpha_1, \,\alpha_2)} |\log_2 \varepsilon | \right) & \text{if} \quad \kappa \in [1,2)\, .
  \end{cases}
  \]

\end{theorem}

\begin{proof}[Sketch of proof]
  Lemma~\ref{lem:strong-error-coupling} implies that
  \[
    \|\Psi(X_K^{\overline{M}}) - \Psi(X_N^M) \|_{L^2(\Omega,H)} \le C (M^{-\min(1,\kappa)/2} + \lambda_{N+1}^{-\kappa/2}),
  \]
  which by Jensen's
  inequality %
  implies that inequality~\eqref{eq:weak-assumption-mlmc} is satisfied for
  $\alpha_1= \min(1,\kappa)/2$ and $\alpha_2 = \kappa \nu/2$ at least. We set
  \[
  M_{k} \eqsim 2^{k/\min(1,\kappa)} \quad \text{and} \quad N_{k} \eqsim 2^{k/(\kappa \nu)}
  \]
  and
  \[
  L = \max(\lceil \log_2(\varepsilon^{-1})/\bar \alpha \rceil, 1),
  \]
  with $\bar \alpha := \min(\alpha_1/\min(1,\kappa) \,, \, \alpha_2/(\kappa \nu))$.
  Defining $m_\ell$ as in~\cite[(2.14)]{CHJZ22},
  here with variance sequence $V_\ell \lesssim 2^{-\ell}$ and cost sequence
  $C_\ell \eqsim (\ell+1) 2^{(1/\min(1,\kappa) \, + \, 1/(\kappa \nu)) \ell}$,
  the result follows by a similar argument as in the proof of~\cite[Theorem 1]{CHJZ22}
  with the MLMC triplet of rate parameters $\bar \alpha$ as defined above, $\bar \beta =1$ and
  $\bar \gamma = 1/\min(1,\kappa)+ 1/(\kappa \nu)$, and using that $\bar \alpha \ge 1/2$.
\end{proof}

For performance comparison, we restrict ourselves to the
high-regularity setting favoring MLMC the most, when $\kappa \in [1,2)$.
Then Theorems~\ref{thm:mimc-cost-error} and~\ref{thm:mlmc-cost-error} assert that
MIMC asymptotically outperforms MLMC when
\[
  \frac{1-\kappa \nu}{\alpha_2} \le \max\left( \frac{1}{\kappa \nu \alpha_1}, \frac{1}{
      \alpha_2} \right),
\]
which we note always holds when $\kappa \nu\ge 1$.

\begin{remark}
By considering the SPDE~\eqref{eq:spde} with $G=0$ and
  imposing the stronger regularity assumptions~\cite[Appendix A]{CHJZ22}
  for the reaction term $F$, one can show that MLMCEI with optimally
  chosen parameters will achieve $\cO(\varepsilon^2)$ mean square error at
  \[
  \cost(\MLMCest) =
  \begin{cases}
    \cO(\varepsilon^{-2} ) & \text{if} \quad \kappa \nu >2 \\
     \cO(\varepsilon^{-2} |\log_2 \varepsilon |^5) & \text{if} \quad \kappa \nu = 2 \\
     \cO\left(\varepsilon^{-(1+2/(\kappa \nu))} |\log_2 \varepsilon |^{2/(\kappa \nu) +2}  \right) & \text{if} \quad \kappa \nu < 2.
  \end{cases}
  \]
  So even in this more restricted problem setting, MIMCEI will asymptotically be more efficient than MLMCEI whenever $\kappa \nu < 2$ and $\kappa \in [1,2)$. However, numerical experiments and known results for smooth \(F\) suggest that the techniques of Theorem~\ref{thm:second-order-diff} lead to non-sharp convergence rates in this setting. We therefore postpone a complete analysis in the case $G=0$ to future work.
\end{remark}

\begin{remark}\label{rem:reduced-samples}
  Note that in Corollary~\ref{coro:mix-diff}, we use sub-optimal bounds on
  \(\E[\|\Delta_{\boldEll} \Psi(X) \|_{U}^{2}]\) whenever \(\psi \in \mathcal L(H, U)\) for
  \(\boldEll \in \N^{2}\), cf., Lemma~\ref{lem:strong-error-coupling}. %
  For example, when \(\kappa \in (1,2)\), by combining
  \eqref{eq:strong-error-coupling}--\eqref{eq:strong-error-coupling2}, a more
  optimal bound can be written as
  \begin{equation}\label{eq:l2-mix-diff-sharp-bounds}
    \E[\|\Delta_{\boldEll} \Psi(X) \|_{U}^{2}] \lesssim
    \begin{cases}
      2^{-\kappa \nu \ell_2} & \ell_{1} = 0\\
      2^{-\kappa \nu \ell_{2} -\ell_{1}} \min \big(1 , 2^{- \kappa\nu \ell_{2} + \ell_{1}} \big) & \ell_{1} > 0
    \end{cases}
  \end{equation}
  In other words, unless \(\ell_{1}=0\) where we have a single difference along
  the spatial direction, we obtain an improved convergence rate with respect
  to \(\ell_{2}\).
  The sub-optimal bounds in Corollary~\ref{coro:mix-diff} are sufficient for
  the purposes of Theorems~\ref{thm:mimc-cost-error} and
  \ref{thm:mimc-cost-error-extra}, since the computational complexity of MIMC
  is determined by the convergence rates of \(\|\Delta_{\boldEll} \Psi(X) \|_{L^{2}{(\Omega,
      U)}}\) when \(\min(\ell_{1}, \ell_{2}) = 0\) and the improved bounds in
  Lemma~\ref{lem:strong-error-coupling} for \(\boldEll \in \N^{2}\) do not lead
  to an improved computational complexity.
  Nevertheless, recalling that the number of samples for a multi-index
  \(\boldEll\) can be made a function of the bound on \(\|\Delta_{\boldEll} \Psi(X)
  \|_{L^{2}(\Omega, U)}\), cf., \eqref{eq:ml-general}, we can use the improved bound
  \eqref{eq:l2-mix-diff-sharp-bounds} in place of \eqref{eq:strong-error-coupling} to reduce the number of samples for a
  multi-index \(\boldEll \in \N^{2}\). We explore this in the numerical
  experiment in Section~\ref{subsec:linDriftMultNoise}.

  Additionally, the choice of the index set in
  Theorem~\ref{thm:mimc-cost-error} depends on the ``weak'' convergence rates
  with respect to \(\ell_{1}\) and \(\ell_{2}\) in the bound on \(\|\E[\Delta_{\boldEll} \Psi
  (X)] \|_{U}\), cf. \eqref{eq:weak-error-coupling-delta}. When \(\Psi\) is
  linear, we may be able to obtain improved weak convergence rates that
  reflect the faster convergence of \(\|\Delta_{\boldEll} \Psi(X) \|_{L^{2}{(\Omega, U)}}\)
  in \eqref{eq:l2-mix-diff-sharp-bounds}. However, we are again restricted by
  the slow weak convergence rate along the axis \(\ell_{1}=0\). Methods that can
  improve the convergence rate along this axis, such as multilevel
  Richardson-Romberg extrapolation \cite{LP17}, could provide a way to allow
  us to utilize the improved convergence rate when \(\ell_{1} > 0\). Such a study
  is outside the scope of the current work.
\end{remark}

\section{Convergence analysis} \label{sec:convergence_analysis}

The goal of this section is to obtain a bound on the $L^2(\Omega,H)$-norm of the second order difference
\begin{equation}
	\label{eq:second-order}
	\cE^{L,M}_{K,N}(t) :=
	X_K^L(\lfloor t\rfloor_{M^{-1}}) - X_N^L(\lfloor t\rfloor_{M^{-1}}) - X_K^M(\lfloor t\rfloor_{M^{-1}}) + X_N^M(\lfloor t\rfloor_{M^{-1}})
\end{equation}
for integers $L\ge M$, $K \ge N$.
First, we need a preliminary lemma involving bounds on $P_N$ and $S$. For brevity, we use the notation $E(t) := S(t)-I$.
\begin{restatable}{lemma}{ineqlemma}
	\label{lem:SP-estimates}
	There is a constant $C < \infty$ such that
        \begin{enumerate}[label=(\roman*)]
		\item \label{lem:SP-estimates:proj-bound} $\| (I-P_N) \|_{\cL(\dot{H}^r,H)} = \| (I-P_N) A^{-\frac{r}{2}} \|_{\cL(H)} \le \lambda_{N+1}^{-\frac{r}{2}}$, for $N \in \N, r \ge 0$
		\item \label{lem:SP-estimates:semigroup-time-bound}
		$\| E(t) \|_{\cL(\dot{H}^{r},H)} = \| E(t) A^{-\frac{r}{2}} \|_{\cL(H)} \le C t^{\frac{r}{2}}$, for $t \ge 0, r \in [0,2], $
		\item \label{lem:SP-estimates:semigroup-analyticity}
		$\| S(t) \|_{\cL(H,\dot{H}^{r})} = \| A^\frac{r}{2} S(t) \|_{\cL(H)} \le C t^{-\frac{r}{2}}$ for $ t> 0, r \ge 0$, $\|S(t)\|_{\cL(H)} \le C$ for $t \ge 0$ and
		\item \label{lem:SP-estimates:semigroup-integral-bound-2}
		$\int^{t_2}_{t_1} \| S(t) v \|_{\dot{H}^r}^2 \dd t \le C (t_2-t_1)^{1-r} \| v \|^2$, for $t_2 \ge t_1 \ge 0, v \in H, r \in [0,1]$.
	\end{enumerate}
\end{restatable}
The first estimate follows directly by definition of $P_N$, we refer to
\cite[Lemma~B.9]{K14} for the rest.
Lemma~\ref{lem:SP-estimates}\ref{lem:SP-estimates:semigroup-analyticity}
  and Assumption~\ref{ass:reg}\ref{ass:eigenvalue-growth} yield a bound on
  $S(t)G$ that we will make ample use of, namely,
  \begin{equation}
    \label{eq:theta-G-bound}
    \begin{split}
      \| S(t) G(u) \|^2_{\cL_2(Q^{\frac{1}{2}}(H),\dot{H}^\theta)} &= \sum_{j=1}^\infty \lambda_j^{\theta} \zeta^2_j \mu_j |\langle u , e_{j+m} \rangle|^2 \| S(t) e_j \|^2 \\
                                                            &\le C \sum_{j=1}^\infty \lambda_j^{1-\kappa-\delta+\theta} \lambda^{\chi}_j \lambda^{-\chi}_j |\langle u , e_{j+m} \rangle|^2 \| S(t) e_j \|^2  \\
                                                            &\le C \sum_{j=1}^\infty |\langle u , \lambda_{j+m}^{\frac{\theta-\kappa+\chi}{2}} e_{j+m} \rangle|^2 \| A^{\frac{1-\chi-\delta}{2}} S(t) e_j \|^2 \\
                                                            &\le C t^{\delta + \chi-1} \| u \|^2_{\dot{H}^{\theta+\chi-\kappa}}
    \end{split}
  \end{equation}
  for $\theta \in \R$, $\chi \in (-\infty,1-\delta]$, $t > 0$ and $u \in \dot{H}^{\theta+\chi-\kappa}$. This and the other
  estimates of Lemma \ref{lem:SP-estimates} are used to derive a lemma involving
regularity and convergence of the exponential integrator approximation
$X_N^M$. With the exception of~\ref{lem:appr:neg-space-error}, the results are
standard but we know no estimates with these exact assumptions in the
literature to cite, wherefore we include a proof.
\begin{restatable}{lemma}{boundslemma}
	\label{lem:appr}
		Under Assumption~\ref{ass:reg}, there is for all $T < \infty$ a constant $C < \infty$ such that
		\begin{enumerate}[label=(\roman*)]
			\item \label{lem:appr:space-reg} $\| X_N^M(t)\|_{L^{10}(\Omega,\dot{H}^{\kappa})} \le C, \quad t \in [0,T], M,N \in \N, $
			\item \label{lem:appr:time-reg} $\| X_N^M(t) - X_N^M(s) \|_{L^{10}(\Omega,H)} \le C \lambda^{\frac{\min(2\kappa,1)-\min(\kappa,1)}{2}}_N |t-s|^{\min(2\kappa,1)/2}, \hspace{0.5em} s,t \in [0,T], M,N \in \N$ and
$\| X_N^M(t) - X_N^M(s) \|_{L^{10}(\Omega,H)} \le C |t-s|^{\min(\kappa,1)/2}, \hspace{0.5em} s,t \in [0,T], M,N \in \N$
			\item \label{lem:appr:space-error}$\| X_N^M(t) - X_K^M(t) \|_{L^{10}(\Omega,H)} \le C \lambda_{N+1}^{-\kappa/2}, \quad t \in [0,T], K \ge N \in \N, $
			\item \label{lem:appr:neg-space-error} $\| X_N^M(t) - X_K^M(t) \|_{L^{\frac{10}{3}}(\Omega,\dot{H}^{-\kappa})} \le C \lambda_{N+1}^{-\kappa}, t \in [0,T], K \ge N \in \N$ and
			\item \label{lem:appr:time-error}$\| X_N^M(t) - X_N^L(t) \|_{L^{10}(\Omega,H)} \le C \lambda^{\frac{\min(2\kappa,1)-\min(\kappa,1)}{2}}_N M^{-\min(2\kappa,1)/2}, \quad t \in [0,T], L \ge M \in \N$ and
$\| X_N^M(t) - X_N^L(t) \|_{L^{10}(\Omega,H)} \le C M^{-\min(\kappa,1)/2}, \quad t \in [0,T], L \ge M \in \N$.
		\end{enumerate}
\end{restatable}
\begin{proof}
	Throughout the proof, we let $C< \infty$ denote a generic constant which may change from line to line but does not depend on $M, N, L, s$ or $t$. For brevity, we write $\tau := M^{-1}$.

	To derive \ref{lem:appr:space-reg},  we first use the triangle inequality to see that
	\begin{equation}
		\label{eq:lem:appr:pfeq2}
		\begin{split}
		\qquad\| X_N^M(t)\|_{L^{10}(\Omega,\dot{H}^{\kappa})} &\le \|S(t) P_N X_0\|_{L^{10}(\Omega,\dot{H}^{\kappa})} + \Big\|\int^t_0 S(t-s) P_N F(X^M_N(\lfloor s \rfloor_{\tau})) \dd s \Big\|_{L^{10}(\Omega,\dot{H}^{\kappa})}\\
		&\quad+ \Big\|\int^t_0 S(t-s) \big(P_N + G X^M_N(\lfloor s \rfloor_{\tau}) \big) \dd W(s)\Big\|_{L^{10}(\Omega,\dot{H}^{\kappa})}.
	\end{split}
	\end{equation}
	Using Lemma \ref{lem:SP-estimates}\ref{lem:SP-estimates:semigroup-analyticity} and Assumption~\ref{ass:reg}\ref{ass:X0}, the first term is bounded by a constant, whereas for the second term we also use~\eqref{eq:F-lip-lin-growth} to see that
	\begin{equation*}
		\Big\|\int^t_0 S(t-s) P_N F(X^M_N(\lfloor s \rfloor_{\tau})) \dd s \Big\|_{L^{10}(\Omega,\dot{H}^{\kappa})} \le C \int^t_0 (t-s)^{-\kappa/2} \big(1+\| X^M_N(\lfloor s \rfloor_{\tau})\|_{L^{10}(\Omega,H)}\big) \dd s.
	\end{equation*}
	For the third term, the Burkholder--Davis--Gundy (BDG) inequality \cite[Theorem~4.36]{DPZ14}, the triangle inequality, Lemma~\ref{lem:SP-estimates}\ref{lem:SP-estimates:semigroup-integral-bound-2}, Assumption~\ref{ass:reg}\ref{ass:Q} and \eqref{eq:theta-G-bound} with $\theta = \kappa$ yield that the stochastic integral term is bounded by
	\begingroup
	\allowdisplaybreaks
	\begin{align*}
		& \E\Big[\Big(\int^t_0 \| S(t-s) \big(P_N + G X^M_N(\lfloor s \rfloor_{\tau}) \big) \|_{\cL_2(Q^{1/2}(H),\dot{H}^\kappa)}^2 \dd s\Big)^{5}\Big]^{\frac{1}{10}} \\
		&\quad\le \left(\int^t_0 \| S(t-s) \|_{\cL_2(Q^{1/2}(H),\dot{H}^\kappa)}^2 \dd s\right)^{\frac{1}{2}} \\
		&\qquad+ \E\Big[\Big(\int^t_0 \| S(t-s)G X^M_N(\lfloor s \rfloor_{\tau}) \|_{\cL_2(Q^{1/2}(H),\dot{H}^\kappa)}^2 \dd s\Big)^{5}\Big]^{\frac{1}{10}}\\
		&\quad\le C\left(\int^t_0 \| S(s) A^{\frac{\kappa-1}{2}} I_{Q^{1/2}(H) \hookrightarrow \dot{H}^{\kappa-1}} \|_{\cL_2(Q^{1/2}(H),\dot{H}^1)}^2 \dd s\right)^{\frac{1}{2}} \\
		&\qquad+ C\E\Big[\Big(\int^t_0 (t-s)^{\delta-1}\|X^M_N(\lfloor s \rfloor_{\tau})\|^2 \dd s \Big)^{5}\Big]^{\frac{1}{10}} \\
		&\quad\le C \Bigg(\| I_{Q^{1/2}(H) \hookrightarrow \dot{H}^{\kappa-1}} \|_{\cL_2(Q^{1/2}(H),\dot{H}^{\kappa-1})} + \Big(\int^t_0 (t-s)^{\delta-1}\|X^M_N(\lfloor s \rfloor_{\tau})\|_{L^{10}(\Omega,\dot{H}^\kappa)}^2 \dd s\Big)^{1/2}\Bigg).
	\end{align*}%
	\endgroup
	In the last step we made use of the fact that $\dot{H}^r \hookrightarrow \dot{H}^s$ for $r \ge s$, a fact that we will use several times below without further mention. Squaring both sides of \eqref{eq:lem:appr:pfeq2} and appealing to H\"older's inequality for the term involving $F$, combined with the discrete Gronwall inequality completes the proof of \ref{lem:appr:space-reg} in the case that $t = \lfloor t \rfloor_{\tau}$. Combining this estimate with the (continuous) Gronwall inequality proves the claim for general $t$.

	We only prove the first of the two bounds in \ref{lem:appr:time-reg}, since the second is of a more standard nature and can be proven by similar techniques. To see that the first bound holds true, we use the semigroup property and the BDG inequality
	to see that
	\begin{align*}
		\| X_N^M(t) - X_N^M(s)\|_{L^{10}(\Omega,H)} &\le C \Big(\|E(t-s) X_N^M(s)\|_{L^{10}(\Omega,H)} \\&\quad+ \Big\|\int^t_s S(t-r) P_N F(X^M_N(\lfloor r \rfloor_{\tau})) \dd r \Big\|_{L^{10}(\Omega,H)} \\
		&\quad+ \E\Big[\Big(\int^t_s \| S(t-r) \big(P_N + G X^M_N(\lfloor r \rfloor_{\tau}) \big) \|_{\cL_2(Q^{1/2}(H),H)}^2 \dd r\Big)^{5}\Big]^{\frac{1}{10}}\Big).
	\end{align*}
	Using the fact that
	\begin{equation}\label{eq:lem:appr:pfeq3}\| P_N v \| \le C  \lambda^{r/2}_N \| P_N v \|_{\dot{H}^{-r}} \quad \text{for } v \in H, r \ge 0,
	\end{equation}
	the first term can be bounded by
\begin{align*}
	\|E(t-s) X_N^M(s)\|_{L^{10}(\Omega,H)} &\le \lambda^{\frac{\min(2\kappa,1)-\min(\kappa,1)}{2}}_N \| E(t-s) A^{\frac{\min(\kappa,1)-\min(2\kappa,1)-\kappa}{2}}\|_{\cL(H)} \\
	&\quad \times \|X_N^M(s)\|_{L^{10}(\Omega,\dot{H}^{\kappa})}
\end{align*}
	which in turn may be bounded by a constant times  $$\lambda^{\frac{\min(2\kappa,1)-\min(\kappa,1)}{2}}_N |t-s|^{(\min(2\kappa,1)-\min(0,1-\kappa))/2}\le C\lambda^{\frac{\min(2\kappa,1)-\min(\kappa,1)}{2}}_N |t-s|^{\min(2\kappa,1)/2}$$ using Lemma~\ref{lem:SP-estimates}\ref{lem:SP-estimates:semigroup-time-bound} and Lemma~\ref{lem:appr}\ref{lem:appr:space-reg}. For the term involving $F$ we can use Lemma~\ref{lem:appr}\ref{lem:appr:space-reg} along with Lemma~\ref{lem:SP-estimates}\ref{lem:SP-estimates:semigroup-analyticity} and~\eqref{eq:F-lip-lin-growth} to obtain a similar bound, and for the last term we also use Assumption~\ref{ass:reg}\ref{ass:Q}, Lemma~\ref{lem:SP-estimates}\ref{lem:SP-estimates:semigroup-integral-bound-2} and~\eqref{eq:theta-G-bound}. The last term, in particular, can be bounded by a constant multiplied by $\lambda^{\frac{\min(2\kappa,1)-\min(\kappa,1)}{2}}_N |t-s|^{\min(2\kappa,1)/2}$ and no smaller bound is in general possible.

	We next derive \ref{lem:appr:space-error}, observing that
	\begin{align*}
		&\| X_N^M(t) - X_K^M(t) \|_{L^{10}(\Omega,H)} \\
		&\quad\le C \Bigg(\|(P_N-P_K)S(t)X_0\|_{L^{10}(\Omega,H)} \\
		&\qquad + \Big\|\int^t_0 S(t-s) (P_N - P_K) F(X^M_N(\lfloor s \rfloor_{\tau})) \dd s \Big\|_{L^{10}(\Omega,H)} \\
		&\qquad + \Big\|\int^t_0 S(t-s) P_K \big(F(X^M_N(\lfloor s \rfloor_{\tau})) -F(X^M_K(\lfloor s \rfloor_{\tau}))\big) \dd s \Big\|_{L^{10}(\Omega,H)} \\
		&\qquad+ \E\Big[\Big(\int^t_0 \| S(t-s) G \big(X^M_N(\lfloor s \rfloor_{\tau}) - X^M_K(\lfloor s \rfloor_{\tau}) \big) \|_{\cL_2(Q^{1/2}(H),H)}^2 \dd s\Big)^{5}\Big]^{\frac{1}{10}} \\
		&\qquad+ \Big(\int^t_0 \| S(t-s) (P_N - P_K ) \|_{\cL_2(Q^{1/2}(H),H)}^2 \dd s\Big)^{\frac 1 2}\Bigg).
	\end{align*}
	By Assumption~\ref{ass:reg}\ref{ass:X0} along with Lemma \ref{lem:SP-estimates}\ref{lem:SP-estimates:proj-bound} and~\ref{lem:SP-estimates:semigroup-analyticity},
	\begin{equation*}
		\|(P_N-P_K)S(t)X_0\|_{L^{10}(\Omega,H)} = \|P_K(P_N-I)A^{-\frac{\kappa}{2}}S(t)A^{\frac{\kappa}{2}}X_0\|_{L^{10}(\Omega,H)} \le C \lambda_{N+1}^{-\frac{\kappa}{2}} \|X_0\|_{L^{10}(\Omega,\dot{H}^\kappa)}.
	\end{equation*}
	Similarly, using also Lemma~\ref{lem:appr}\ref{lem:appr:space-reg} and~\eqref{eq:F-lip-lin-growth}, the first term involving $F$ can be bounded by a constant times $\lambda_{N+1}^{-\frac{\kappa}{2}}$  and (using also Assumption~\ref{ass:reg}\ref{ass:derivative-bounded}) the second by a constant times $\int^t_0 \| X_N^M(\lfloor s \rfloor_{\tau}) - X_K^M(\lfloor s \rfloor_{\tau}) \|_{L^{10}(\Omega,H)} \dd s$. Next, we note that by~\eqref{eq:theta-G-bound},
	\begin{align*}
		&\E\Big[\Big(\int^t_0 \| S(t-s) G \big(X^M_N(\lfloor s \rfloor_{\tau} - X^M_K(\lfloor s \rfloor_{\tau}) \big) \|_{\cL_2(Q^{1/2}(H),H)}^2 \dd s\Big)^{5}\Big]^{\frac{1}{10}} \\
		&\quad \le C \Big(\int^t_0 (t-s)^{\delta-1}\| X_N^M(\lfloor s \rfloor_{\tau}) - X_K^M(\lfloor s \rfloor_{\tau}) \|_{L^{10}(\Omega,H)}^2 \dd s\Big)^{\frac 1 2},
	\end{align*}
	while the last term is treated by Assumption~\ref{ass:reg}\ref{ass:Q} and Lemma~\ref{lem:SP-estimates}\ref{lem:SP-estimates:semigroup-integral-bound-2}. As in the proof of \ref{lem:appr:space-reg}, an appeal to H\"older's inequality combined with Gronwall inequality finishes the proof.

	To show~\ref{lem:appr:neg-space-error}, we split the error in two parts,
	\begin{align*}
		\| X_N^M(t) - X_K^M(t) \|_{L^\frac{10}{3}(\Omega,\dot{H}^{-\kappa})} \le &\| (I-P_N) X_K^M(t) \|_{L^\frac{10}{3}(\Omega,\dot{H}^{-\kappa})} \\
		&\quad+ \| X_N^M(t) - P_N X_K^M(t) \|_{L^\frac{10}{3}(\Omega,\dot{H}^{-\kappa})}.
	\end{align*}
	By Lemmas~\ref{lem:SP-estimates}\ref{lem:SP-estimates:proj-bound} and~\ref{lem:appr}\ref{lem:appr:space-reg}, the first term is bounded by $C \lambda_{N+1}^{-\kappa}$. For the second term, we note that $Z := X^M_N - P_N X^M_K$ can be expressed by
	\begin{align*}
		Z(t) &= \int^t_0 S(t-s) P_N \big(F(X^M_N(\lfloor s \rfloor_{\tau})) - F(P_N X^M_K(\lfloor s \rfloor_{\tau}))\big) \dd s  \\
		&\quad+ \int^t_0 S(t-s) P_N \big(F(P_N X^M_K(\lfloor s \rfloor_{\tau})) - F(X^M_K(\lfloor s \rfloor_{\tau}))\big) \dd s  \\
		&\quad+ \int^t_0 S(t-s) G Z(\lfloor s \rfloor_{\tau}) \dd W(s).
	\end{align*}
	By Lemma~\ref{lem:SP-estimates}\ref{lem:SP-estimates:semigroup-analyticity}, it follows that
	\begin{align*}
		&\Big\|\int^t_0 S(t-s) P_N \big(F(P_N X^M_K(\lfloor s \rfloor_{\tau})) - F(X^M_K(\lfloor s \rfloor_{\tau}))\big) \dd s\Big\|_{L^{\frac{10}{3}}(\Omega,H)} \\
		&\quad\le \int^t_0 \| S(t-s) A^{\frac \eta 2}\|_{\cL(H)} \| A^{-\frac \eta 2} F(P_N X^M_K(\lfloor s \rfloor_{\tau})) - F(X^M_K(\lfloor s \rfloor_{\tau})) \|_{L^{\frac{10}{3}}(\Omega,H)} \dd s \\
		&\quad\le C  \int^t_0 (t-s)^{-\eta/2} \| F(P_N X^M_K(\lfloor s \rfloor_{\tau})) - F(X^M_K(\lfloor s \rfloor_{\tau})) \|_{L^{\frac{10}{3}}(\Omega,\dot{H}^{-\eta})} \dd s.
	\end{align*}
	Combining this with the mean value theorem and the BDG inequality along with Assumptions~\ref{ass:derivative-bounded} and \ref{ass:derivative-negnorm} and~\eqref{eq:theta-G-bound}, it follows that
	\begingroup
	\allowdisplaybreaks
	\begin{align*}
		&\|Z(t)\|_{L^\frac{10}{3}(\Omega,H)} \\
		&\hspace{0.4em}\le C \Bigg(\int^t_0 \Big\|\int^1_0 F'\big(X^M_K(\lfloor s \rfloor_{\tau}) + \rho Z(\lfloor s \rfloor_{\tau}) \big) Z(\lfloor s \rfloor_{\tau}) \dd \rho \Big\|_{L^\frac{10}{3}(\Omega,H)} \dd s  \\
		&\hspace{1.0em}+ \int^t_0 (t-s)^{-\eta/2} \Big\|\int^1_0 F'\big((\rho I+(1-\rho)P_N)X^M_K(\lfloor s \rfloor_{\tau})\big) (I-P_N)X^M_K(\lfloor s \rfloor_{\tau}) \dd \rho \Big\|_{L^\frac{10}{3}(\Omega,\dot{H}^{-\eta})} \dd s  \\
		&\hspace{1.0em}+ \E\Big[\Big(\int^t_0 \| S(t-s) G Z(\lfloor s \rfloor_{\tau}) \|_{\cL_2(Q^{1/2}(H),H)}^2 \dd s\Big)^{\frac{5}{3}}\Big]^{\frac{3}{10}} \Bigg)\\
		&\hspace{0.4em}\le C \Bigg(\int^t_0 \|Z(\lfloor s \rfloor_{\tau})\|_{L^\frac{10}{3}(\Omega,H)} \dd s \\
		&\hspace{1.0em}+ \lambda^{-\kappa}_{N+1} \int^t_0 (t-s)^{-\eta/2} \E\Big[\big|(1+\|X^M_K(\lfloor s \rfloor_{\tau})\|^{2}_{\dot{H}^\kappa}) \|X^M_K(\lfloor s \rfloor_{\tau})\|_{\dot{H}^\kappa} \big|^\frac{10}{3}\Big]^{\frac{3}{10}} \dd s \\
		&\hspace{1.0em}+ \Big(\int^t_0 (t-s)^{\delta-1}\|Z(\lfloor s \rfloor_{\tau})\|^2_{L^\frac{10}{3}(\Omega,H)} \dd s\Big)^{\frac{1}{2}} \Bigg).
	\end{align*}%
	\endgroup
	In the last inequality, we used the estimate $\|(I-P_N)X^M_K(\lfloor s \rfloor_{\tau})\|_{\dot{H}^{-\kappa}} \le \lambda^{-\kappa}_{N+1} \|X^M_K(\lfloor s \rfloor_{\tau})\|_{\dot{H}^\kappa}$.
	Now H\"older's inequality (with conjugate exponents $3$ and $3/2$ in the expectation of the middle term and $2$ for the first term), Lemma~\ref{lem:appr}\ref{lem:appr:space-reg}
	and the Gronwall inequality (first the discrete, then the continuous) proves the claim.

	For~\ref{lem:appr:time-error}, we again only derive the first bound. To do this, we bound the term $\| X_N(t) - X^M_N(t) \|_{L^{10}(\Omega,H)}$, where the semidiscrete approximation $X_N$ is defined by
	\begin{align*}
		X_N(t) &:= S(t) P_N X_0 + \int^t_0 S(t-s) P_N F(X_N(s)) \dd s \\
		&\quad+ \int^t_0 S(t-s) \big(P_N + G X_N(s) \big) \dd W(s), \quad t \in [0,T].
	\end{align*}
	The claim then follows from the triangle inequality and the fact that $L \ge M$. By adding and subtracting terms involving $X^M _N(s)$ and using the BDG inequality, we arrive at
	\begin{align*}
		&\| X_N(t) - X_N^M(t) \|_{L^{10}(\Omega,H)} \\
		&\quad= \Big\| \int^t_0 S(t-s) P_N \big(F(X_N(s)) - F(X^M_N(s))\big) \dd s \Big\|_{L^{10}(\Omega,H)} \\
		&\qquad+ \Big\| \int^t_0 S(t-s) P_N \big(F(X^M _N(s))-F(X^M_N(\lfloor s \rfloor_\tau))\big) \dd s \Big\|_{L^{10}(\Omega,H)} \\
		&\qquad+ \E\Big[\Big(\int^t_0 \| S(t-s) G \big(X_N(s) - X^M_N(s) \big) \|_{\cL_2(Q^{1/2}(H),H)}^2 \dd s\Big)^{5}\Big]^{\frac{1}{10}}\\
		&\qquad+ \E\Big[\Big(\int^t_0 \| S(t-s) G \big(X^M _N(s) - X^M_N(\lfloor s \rfloor_\tau) \big) \|_{\cL_2(Q^{1/2}(H),H)}^2 \dd s\Big)^{5}\Big]^{\frac{1}{10}}.
	\end{align*}
	The second and fourth term are bounded by a constant times  $\lambda^{\frac{\min(2\kappa,1)-\min(\kappa,1)}{2}}_N M^{-\min(2\kappa,1)/2}$, which follows from Lemma~\ref{lem:SP-estimates}\ref{lem:SP-estimates:semigroup-analyticity} combined with  Assumption~\ref{ass:reg}\ref{ass:derivative-bounded}, Lemma~\ref{lem:appr}\ref{lem:appr:time-reg} and~\eqref{eq:theta-G-bound}. Combining these same arguments with H\"older's inequality on the terms stemming from $F$, we arrive at
	\begin{align*}
		\| X_N(t) - X_N^M(t) \|^2_{L^{10}(\Omega,H)} &\le C \Bigg( \lambda^{\frac{\min(2\kappa,1)-\min(\kappa,1)}{2}}_N M^{-\min(2\kappa,1)} \\
		&\hspace{3em}+ \int^t_0 (t-s)^{\min(\kappa+\delta-1,0)} \| X_N(s) - X_N^M(s) \|^2_{L^{10}(\Omega,H)} \dd s \Bigg)
	\end{align*}
	and the Gronwall inequality finishes the proof.
\end{proof}

Using Lemma~\ref{lem:appr}\ref{lem:appr:space-error} and (a variant of) Lemma~\ref{lem:appr}\ref{lem:appr:time-error} one immediately obtains a bound on the second order difference $\cE^{L,M}_{K,N}(t)$ of \eqref{eq:second-order} by a constant times $\lambda_{N+1}^{-\kappa/2} + M^{-\min(\kappa,1)/2}$. As the following theorem shows, however, this can be significantly improved. We obtain three different regimes of convergence, depending on whether the smoothness parameter $\kappa$ fulfills $\kappa \in (0,1/2)$, $\kappa \in [1/2,1)$ or $\kappa \in [1,2)$. %

\begin{theorem}
	\label{thm:second-order-diff}
	Let $K\ge N$ and $L \ge M \in \N$, where $K = D N$ for some $D \in \N$. Under Assumption~\ref{ass:reg}, there is a constant $C < \infty$ that does not depend on  $K,L,M, N$ or $t \in [0,T]$, such that\begin{equation*}
		\|\cE^{L,M}_{K,N}(t)
		\|_{L^2(\Omega,H)} \le \begin{cases}
			C \min( M^{-\kappa},  \, \lambda_{N+1}^{-\kappa}) \quad &\text{ if } \kappa \in (0,1/2), \\
			C \min\Big( \lambda_{N+1}^{-(2\kappa-1)/2} M^{-1/2}, \,  \lambda_{N+1}^{-\kappa}\Big) \quad &\text{ if } \kappa \in [1/2,1), \\
			C \min( \lambda_{N+1}^{-\kappa/2} M^{-1/2}, \, \lambda_{N+1}^{-\kappa}) \quad &\text{ if } \kappa \in [1,2).
		\end{cases}
	\end{equation*}
\end{theorem}

\begin{remark}
	Our assumptions on $F$ can be weakened if we content ourselves with deriving an upper bound of the form $ \lambda_{N+1}^{-(\min(2\kappa,\kappa+1)-\min(2\kappa,1))/2} M^{-\min(2\kappa,1)/2}$. Then, the negative norm bound of Lemma~\ref{lem:appr}\ref{lem:appr:neg-space-error} is not needed. As can be seen from the proof below, we do not make use of \ref{ass:reg}\ref{ass:difference-regularity-transfer} and need not apply Assumption~\ref{ass:reg}\ref{ass:derivative-negnorm} for $\kappa > 1$. When $F$ is a composition type mapping we therefore only need to assume twice differentiability of the underlying function $f$. %
\end{remark}

To show this claim, we need another lemma involving the second order difference
\begin{equation*}
	\cH^{L,M}_{K,N}(t) := X_K^L(\lfloor t\rfloor_{L^{-1}}) - X_N^L(\lfloor t\rfloor_{L^{-1}}) - X_K^L(\lfloor t\rfloor_{M^{-1}}) + X_N^L(\lfloor t\rfloor_{M^{-1}}).
\end{equation*}

\begin{lemma}
	\label{lem:second-order-diff}
	Under the same conditions as in Theorem~\ref{thm:second-order-diff}, there is a constant \( C < \infty \) that does not depend on \( K, L, M, N \), or \( t \in [0,T] \), such that
	\begin{equation*}
		\| \cH^{L,M}_{K,N}(t) \|_{L^4(\Omega, \dot{H}^{-\min(\kappa,1)})} \le
		\begin{cases}
			C M^{-\kappa}
			& \text{if } \kappa \in (0,1/2), \\
			C \lambda_{N+1}^{-(2\kappa - 1)/2} M^{-1/2}
			& \text{if } \kappa \in [1/2,1), \\
			C \lambda_{N+1}^{-\kappa/2} M^{-1/2}
			& \text{if } \kappa \in [1,2).
		\end{cases}
	\end{equation*}
	Moreover,
	\begin{equation*}
		\| \cH^{L,M}_{K,N}(t) \|_{L^2(\Omega, \dot{H}^{-\kappa})} \le C \lambda_{N+1}^{-\kappa}.
	\end{equation*}
\end{lemma}

\begin{proof}
	Throughout the proof, we let $C< \infty$ denote a generic constant which may change from line to line but does not depend on $K,L,M$ or $N$ or $t$.
	From the semigroup property, it follows that
	\begingroup
	\allowdisplaybreaks
	\begin{align*}
		&X_K^L(\lfloor t\rfloor_{L^{-1}}) - X_N^L(\lfloor t\rfloor_{L^{-1}}) - X_K^L(\lfloor t\rfloor_{M^{-1}}) + X_N^L(\lfloor t\rfloor_{M^{-1}}) \\
		&\quad= \int^{\lfloor t\rfloor_{L^{-1}}}_{\lfloor t\rfloor_{M^{-1}}}  P_K S(\lfloor t\rfloor_{L^{-1}}-s) (F(X_K^L(\lfloor s\rfloor_{L^{-1}})) - F(X_N^L(\lfloor s\rfloor_{L^{-1}}))) \dd s \\
		&\qquad+ \int^{\lfloor t\rfloor_{L^{-1}}}_{\lfloor t\rfloor_{M^{-1}}}  S(\lfloor t\rfloor_{L^{-1}}-s) (P_K-P_N) F(X_N^L(\lfloor s\rfloor_{L^{-1}})) \dd s \\
		&\qquad+ \int^{\lfloor t\rfloor_{M^{-1}}}_{0}  P_K S(\lfloor t\rfloor_{M^{-1}} - s) E(\lfloor t\rfloor_{L^{-1}}-\lfloor t\rfloor_{M^{-1}}) \\
		&\hspace{8em}\times(F(X_K^L(\lfloor s\rfloor_{L^{-1}})) - F(X_N^L(\lfloor s\rfloor_{L^{-1}}))) \dd s \\
		&\qquad+ \int^{\lfloor t\rfloor_{M^{-1}}}_{0}  S(\lfloor t\rfloor_{M^{-1}} - s) (P_K-P_N) E(\lfloor t\rfloor_{L^{-1}}-\lfloor t\rfloor_{M^{-1}}) F(X_N^L(\lfloor s\rfloor_{L^{-1}})) \dd s \\
		&\qquad+\int^{\lfloor t\rfloor_{L^{-1}}}_{\lfloor t\rfloor_{M^{-1}}} S(\lfloor t\rfloor_{L^{-1}}-s) \big(P_K-P_N + G\big(X_K^L(\lfloor s\rfloor_{L^{-1}})-X_N^L(\lfloor s\rfloor_{L^{-1}})\big) \big) \dd W(s) \\
		&\qquad+ \int^{\lfloor t\rfloor_{M^{-1}}}_{0}  S(\lfloor t\rfloor_{M^{-1}} - s) E(\lfloor t\rfloor_{L^{-1}}-\lfloor t\rfloor_{M^{-1}}) \\
		&\hspace{8em}\times \big(P_K-P_N + G\big(X_K^L(\lfloor s\rfloor_{L^{-1}})-X_N^L(\lfloor s\rfloor_{L^{-1}})\big) \big)  \dd W(s) \\
		&\quad=: \mathrm{I}^1_F + \mathrm{I}^2_F + \mathrm{I}^3_F + \mathrm{I}^4_F + \mathrm{I}^1_W + \mathrm{I}^2_W,
	\end{align*}%
	\endgroup
	We go through the terms one by one and show that each can be bounded in $L^4(\Omega,\dot{H}^{-\min(\kappa,1)})$ and $L^2(\Omega,\dot{H}^{-\kappa})$ by $C \lambda_{N+1}^{-(\min(2\kappa,\kappa+1)-\min(2\kappa,1))/2} M^{-\min(2\kappa,1)/2}$ and $C \lambda_{N+1}^{-\kappa}$, respectively. Note that $|\lfloor t \rfloor_{L^{-1}} - \lfloor t \rfloor_{M^{-1}}| \le M^{-1}$, a fact that will be used without explicit mention. This we also do for the inequalities corresponding to $\dot{H}^r \hookrightarrow \dot{H}^s, r > s$.

	On the one hand, by~\eqref{eq:F-lip-lin-growth}, Lemma~\ref{lem:SP-estimates}\ref{lem:SP-estimates:semigroup-analyticity}
	and Lemma~\ref{lem:appr}\ref{lem:appr:space-error},
	\begin{align*}
		\|\mathrm{I}^1_F\|_{L^4(\Omega,\dot{H}^{-\min(\kappa,1)})} &\le C M^{-1} \sup_{s \in [0,T]}\| X_K^L(\lfloor s\rfloor_{L^{-1}}) - X_N^L(\lfloor s\rfloor_{L^{-1}}) \|_{L^4(\Omega,H)} \\ &\le C M^{-1} \lambda_{N+1}^{-\kappa/2}.
	\end{align*}
	On the other hand, by Assumption~\ref{ass:reg}\ref{ass:derivative-negnorm}, the mean value theorem, H\"older's inequality (with conjugate exponents $5/2$ and $5/3$),
	Lemma~\ref{lem:SP-estimates}\ref{lem:SP-estimates:semigroup-analyticity} and Lemma~\ref{lem:appr}\ref{lem:appr:neg-space-error}, $\|\mathrm{I}^1_F\|_{L^2(\Omega,\dot{H}^{-\kappa})}$ may be bounded by
	\begin{align*}
		&\int^{\lfloor t\rfloor_{L^{-1}}}_{\lfloor t\rfloor_{M^{-1}}}  (\lfloor t\rfloor_{L^{-1}}-s)^{-\eta/2} \Big\|\int^1_0 F'\big((1-\rho)X_N^L(\lfloor s\rfloor_{L^{-1}})+\rho X_K^L(\lfloor s\rfloor_{L^{-1}})\big) \\
		&\hspace{16em}\times \big(X_K^L(\lfloor s\rfloor_{L^{-1}}) - X_N^L(\lfloor s\rfloor_{L^{-1}})\big) \dd \rho\|_{L^2(\Omega,\dot{H}^{-\eta})} \dd s \\
		&\quad \le C \sup_{s \in [0,T]}\| X_K^L(\lfloor s\rfloor_{L^{-1}}) - X_N^L(\lfloor s\rfloor_{L^{-1}}) \|_{L^{\frac{10}{3}}(\Omega,\dot{H}^{-\kappa})} \le C\lambda_{N+1}^{-\kappa}.
	\end{align*}

	Next, we note that
	\begin{align*}
		&\|\mathrm{I}^2_F\|_{L^4(\Omega,\dot{H}^{-s})} \\
		&\quad\le C \int^{\lfloor t\rfloor_{L^{-1}}}_{\lfloor t\rfloor_{M^{-1}}}  \|A^{\frac{r}{2}}S(\lfloor t\rfloor_{L^{-1}}-s) A^{-\frac{r+s}{2}}(P_K-P_N) F(X_N^L(\lfloor \rho\rfloor_{L^{-1}}))\|_{L^4(\Omega,H)} \dd \rho.
	\end{align*}
	Using~\eqref{eq:F-lip-lin-growth}, Lemma~\ref{lem:appr}\ref{lem:appr:space-reg} combined with Lemma~\ref{lem:SP-estimates}\ref{lem:SP-estimates:proj-bound} and \ref{lem:SP-estimates:semigroup-analyticity} yields a bound in terms of $M^{-1/2} \lambda_{N+1}^{-(\min(\kappa,1)+1)/2}$ when $s=\min(\kappa,1), r=1$ and $M^{-(1-\kappa/2)}\lambda_{N+1}^{-\kappa} \le \lambda_{N+1}^{-\kappa}$ when $s=r=\kappa$.

	Then, by Assumption~\ref{ass:reg}\ref{ass:derivative-bounded}, Lemma~\ref{lem:appr}\ref{lem:appr:space-error} and Lemma~\ref{lem:SP-estimates}\ref{lem:SP-estimates:semigroup-time-bound}-\ref{lem:SP-estimates:semigroup-analyticity}, we can on the one hand bound $\|\mathrm{I}^3_F\|_{L^4(\Omega,\dot{H}^{-\min(\kappa,1)})}$ by $\|\mathrm{I}^3_F\|_{L^4(\Omega,H)}$ and then via
	\begin{align*}
		&\|\mathrm{I}^3_F\|_{L^4(\Omega,H)} \le
		\int^{\lfloor t\rfloor_{M^{-1}}}_{0} (\lfloor t\rfloor_{M^{-1}} - s)^{-\frac 1 2}  \|A^{-\frac 1 2}E(\lfloor t\rfloor_{L^{-1}}-\lfloor t\rfloor_{M^{-1}}) \|_{\cL(H)} \\
		&\hspace{14em}\times\| X_K^L(\lfloor s\rfloor_{L^{-1}}) - X_N^L(\lfloor s\rfloor_{L^{-1}}) \|_{L^4(\Omega,H)} \dd s \le C M^{-1/2} \lambda_{N+1}^{-\kappa/2}.
	\end{align*}
	Repeating the argument from $\mathrm{I}^1_F$, we instead find that  $\|\mathrm{I}^{3}_F\|_{L^2(\Omega,\dot{H}^{-\kappa})} \le C \lambda_{N+1}^{-\kappa}$.

	We next use~\eqref{eq:F-lip-lin-growth}, Lemma~\ref{lem:appr}\ref{lem:appr:space-reg} and  Lemma~\ref{lem:SP-estimates}\ref{lem:SP-estimates:proj-bound}-\ref{lem:SP-estimates:semigroup-analyticity} to obtain a bound on the term $\|\mathrm{I}^4_F\|_{L^4(\Omega,\dot{H}^{-\min(\kappa,1)})}$ via
	\begin{align*}
		 &\int^{\lfloor t\rfloor_{M^{-1}}}_{0} (\lfloor t\rfloor_{M^{-1}} - s)^{-\frac{\kappa+1-\min(\kappa,1)}{2}} \|A^{- \frac \kappa 2}(P_K-P_N)\|_{\cL(H)} \\
		&\hspace{1em}\times\|A^{-\frac 1 2}E(\lfloor t\rfloor_{L^{-1}}-\lfloor t\rfloor_{M^{-1}}) \|_{\cL(H)} \|F(X_N^L(\lfloor s\rfloor_{L^{-1}}))\|_{L^4(\Omega,H)} \dd s \le C M^{-  1/2} \lambda_{N+1}^{- \kappa/ 2},
	\end{align*}
	whereas the bound  $\|\mathrm{I}^4_F\|_{L^2(\Omega,\dot{H}^{-\kappa})} \le C \lambda_{N+1}^{-\kappa}$ is obtained in the same way as for the term $\mathrm{I}^2_F$.

	The term $\|\mathrm{I}^1_W\|_{L^q(\Omega,\dot{H}^{-s})}$, $q \in \{2,4\}, s \in \{\kappa,1\},$ is split in two parts, using the BDG inequality, Assumption~\ref{ass:reg}\ref{ass:Q}, \eqref{eq:theta-G-bound} with $\theta = -s$ and $\chi=\min(2\kappa,1)-\delta$ and all parts of Lemma~\ref{lem:SP-estimates} to obtain a bound by
	\begin{align*}
		&\Big(\int^{\lfloor t\rfloor_{L^{-1}}}_{\lfloor t\rfloor_{M^{-1}}} \|A^{\frac{r}{2}} S(\lfloor t\rfloor_{L^{-1}}-\rho) A^{-\frac{\kappa+r+s-1}{2}} (P_K - P_N) A^{\frac{\kappa-1}{2}}\|_{\cL_2(Q^{1/2}(H),H)}^2 \dd \rho\Big)^{\frac 1 2} \\
		&\quad +  \E\Big[\Big(\int^{\lfloor t\rfloor_{L^{-1}}}_{\lfloor t\rfloor_{M^{-1}}}  \|S(\lfloor t\rfloor_{M^{-1}} - \rho)G\big(X_K^L(\lfloor \rho\rfloor_{L^{-1}})-X_N^L(\lfloor \rho\rfloor_{L^{-1}})\big)\|_{\cL_2(Q^{1/2}(H),\dot{H}^{-s})}^2 \dd \rho \Big)^{\frac q 2}\Big]^{\frac 1 q} \\
		&\quad\le C \Bigg( M^{- \frac{1-r}{2}} \| I_{Q^{1/2}(H) \hookrightarrow \dot{H}^{\kappa-1}} \|_{\cL_2(H,\dot{H}^{\kappa-1})} \lambda_{N+1}^{- \frac{ \kappa+r+s-1}{ 2}} \\
		&\hspace{3em}+ \Big(\int^{\lfloor t\rfloor_{L^{-1}}}_{\lfloor t\rfloor_{M^{-1}}} (\lfloor t\rfloor_{L^{-1}}-\rho)^{\min(2\kappa,1)-1} \| X_K^L(\lfloor \rho\rfloor_{L^{-1}}) - X_N^L(\lfloor \rho\rfloor_{L^{-1}}) \|_{L^q(\Omega,\dot{H}^{-\kappa-s+\chi})}^2 \dd \rho\Big)^{\frac 1 2}\Bigg).%
	\end{align*}
	For $q=4$ and $s=\min(\kappa,1)$ we use Lemma~\ref{lem:appr}\ref{lem:appr:space-error} and set $r=\max(0,1-2\kappa)$ so that
	$$\|\mathrm{I}^1_W\|_{L^4(\Omega,\dot{H}^{-\min(\kappa,1)})} \le C \lambda_{N+1}^{-(\min(2\kappa,\kappa+1)-\min(2\kappa,1))/2} M^{-\min(2\kappa,1)/2}.$$
	For $q=2$ and $s=\kappa$ we set $r=1$ in the first term and use Lemma~\ref{lem:appr}\ref{lem:appr:neg-space-error} to bound the second term by $\lambda_{N}^{-\kappa}$ (with $\chi=\kappa$) so that $\|\mathrm{I}^1_W\|_{L^2(\Omega,\dot{H}^{-\kappa})} \le C \lambda_{N+1}^{-\kappa}$. Similarly, applying~\eqref{eq:theta-G-bound} with \( \theta = 1 - r - s \) and \(\chi = 0\) yields that $\|\mathrm{I}^2_W\|_{L^q(\Omega,\dot{H}^{-s})}$ is bounded by a constant times
		\begin{align*}
			& \Big(\int^{\lfloor t\rfloor_{M^{-1}}}_0 \| A^{\frac 1 2} S(\lfloor t\rfloor_{M^{-1}} - \rho) A^{\frac{r-1}{ 2}}E(\lfloor t\rfloor_{L^{-1}}-\lfloor t\rfloor_{M^{-1}}) \\
			&\hspace{13em}\times A^{-\frac{\kappa+s+r-1}{2}} (P_K - P_N) A^{\frac{\kappa-1}{2}}\|_{\cL_2(Q^{1/2}(H),H)}^2 \dd \rho\Big)^{\frac 1 2} \\
			 &\, + \E\Big[\Big(\int^{\lfloor t\rfloor_{M^{-1}}}_0 \hspace{-0.8em}\rho^{\delta - 1} \|A^{\frac{ r-1 }{2} }E(\lfloor t\rfloor_{L^{-1}}-\lfloor t\rfloor_{M^{-1}})\|^2 \|X_K^L(\lfloor \rho\rfloor_{L^{-1}})-X_N^L(\lfloor \rho\rfloor_{L^{-1}})\|_{1-r-s-\kappa}^2 \dd \rho \Big)^{\frac q 2}\Big]^{\frac 1 q}.
		\end{align*}
	In the case $q=4$ and $s=\min(\kappa,1)$ we set \(r = \max(1-2\kappa,0)\) to obtain $\|\mathrm{I}^2_W\|_{L^4(\Omega,\dot{H}^{-\min(\kappa,1)})} \le C \lambda_{N+1}^{-(\min(2\kappa,\kappa+1)-\min(2\kappa,1))/2} M^{-\min(2\kappa,1)/2}$ while for $q=2$ and $s=\kappa$ we set $r=1$ to obtain $\|\mathrm{I}^2_W\|_{L^2(\Omega,\dot{H}^{-\kappa})} \le C \lambda_{N+1}^{- \kappa}$.
\end{proof}

\begin{proof}[Proof of Theorem~\ref{thm:second-order-diff}]
Throughout the proof, we let $C< \infty$ denote a generic constant which may change from line to line but does not depend on $K,L,M$ or $N$ or $t$. We recall that we want to derive derive the inequality
	\begin{equation*}
				\|\cE^{L,M}_{K,N}(t)
				\|_{L^2(\Omega,H)} \le
					C \min( \lambda_{N+1}^{-(\min(2\kappa,\kappa+1)-\min(2\kappa,1))/2} M^{-\min(2\kappa,1)/2},\lambda_{N+1}^{-\kappa}).
	\end{equation*}
	First, note that
	\begingroup
	\allowdisplaybreaks
	\begin{align*}
		\cE^{L,M}_{K,N}(t) &= \int^{\lfloor t \rfloor_{M^{-1}}}_0 P_K S(\lfloor t \rfloor_{M^{-1}}-s)
		F(X^L_K(\lfloor s \rfloor_{L^{-1}})) \dd s \\
		&\quad - \int^{\lfloor t \rfloor_{M^{-1}}}_0 P_N S(\lfloor t \rfloor_{M^{-1}}-s) F(X^L_N(\lfloor s \rfloor_{L^{-1}})) \dd s \\
		&\quad - \int^{\lfloor t \rfloor_{M^{-1}}}_0 P_K S(\lfloor t \rfloor_{M^{-1}}-s) F(X^M_K(\lfloor s \rfloor_{M^{-1}})) \dd s \\
		&\quad + \int^{\lfloor t \rfloor_{M^{-1}}}_0 P_N S(\lfloor t \rfloor_{M^{-1}}-s) F(X^M_N(\lfloor s \rfloor_{M^{-1}})) \dd s \\
		&\quad + \int^{\lfloor t \rfloor_{M^{-1}}}_0 S(\lfloor t \rfloor_{M^{-1}}-s) G\Big( X_K^L(\lfloor s\rfloor_{L^{-1}}) - X_N^L(\lfloor s\rfloor_{L^{-1}}) \\
		&\hspace{14em}- X_K^M(\lfloor s\rfloor_{M^{-1}}) + X_N^M(\lfloor s\rfloor_{M^{-1}}) \Big) \dd W(s).
	\end{align*}%
	\endgroup
	We rewrite this as
	\begingroup
	\allowdisplaybreaks
	\begin{align*}
		\cE^{L,M}_{K,N}(t) &= \int^{\lfloor t \rfloor_{M^{-1}}}_0 P_K (I-P_N) S(\lfloor t \rfloor_{M^{-1}}-s) \\
		&\hspace{5em}\times\big( F(X^L_K(\lfloor s \rfloor_{L^{-1}})) - F(X^M_K(\lfloor s \rfloor_{M^{-1}})) \big) \dd s \\
		&\quad+\int^{\lfloor t \rfloor_{M^{-1}}}_0 P_N S(\lfloor t \rfloor_{M^{-1}}-s) \big( F(X^L_K(\lfloor s \rfloor_{M^{-1}})) - F(X^L_N(\lfloor s \rfloor_{M^{-1}})) \\
		&\hspace{10em}- F(X^M_K(\lfloor s \rfloor_{M^{-1}})) + F(X^M_N(\lfloor s \rfloor_{M^{-1}})) \big) \dd s \\
		&\quad+\int^{\lfloor t \rfloor_{M^{-1}}}_0 P_N S(\lfloor t \rfloor_{M^{-1}}-s) \big( F(X^L_K(\lfloor s \rfloor_{L^{-1}})) - F(X^L_N(\lfloor s \rfloor_{L^{-1}})) \\
		&\hspace{10em}- F(X^L_K(\lfloor s \rfloor_{M^{-1}})) + F(X^L_N(\lfloor s \rfloor_{M^{-1}})) \big) \dd s \\
		&\quad + \int^{\lfloor t \rfloor_{M^{-1}}}_0 S(\lfloor t \rfloor_{M^{-1}}-s) G \cE^{L,M}_{K,N}(s) \dd W(s) \\
		&\quad + \int^{\lfloor t \rfloor_{M^{-1}}}_0 S(\lfloor t \rfloor_{M^{-1}}-s) G\cH^{L,M}_{K,N}(s) \dd W(s) \\
		&=: \mathrm{I}^1_F + \mathrm{I}^2_F + \mathrm{I}^3_F + \mathrm{I}^1_W + \mathrm{I}^2_W
	\end{align*}%
	\endgroup
	and proceed similarly as in Lemma~\ref{lem:second-order-diff}. Except for the last two, we derive two different bounds for each term, corresponding to the bounds $ C \lambda_{N+1}^{-(\min(2\kappa,\kappa+1)-\min(2\kappa,1))/2} M^{-\min(2\kappa,1)/2}$ and $C \lambda_{N+1}^{-\kappa}$ for $\|\cE^{L,M}_{K,N}(t)
	\|_{L^2(\Omega,H)}$.

	For the term $\|\mathrm{I}^1_F\|_{L^2(\Omega,H)}$, we first note that
	\begin{align*}
		\|\mathrm{I}^1_F\|_{L^2(\Omega,H)} &\le C \int^{\lfloor t \rfloor_{M^{-1}}}_0 \| A^{- \frac{\kappa+r}{2}} (I-P_N) \|_{\cL(H)} \|A^{\frac{\kappa}{2}} S(\lfloor t \rfloor_{M^{-1}}-s)\|_{\cL(H)} \\
		&\hspace{5em}\times\big\| F(X^L_K(\lfloor s \rfloor_{L^{-1}})) - F(X^M_K(\lfloor s \rfloor_{M^{-1}})) \big\|_{L^2(\Omega,\dot{H}^r)} \dd s.
	\end{align*}
	On the one hand, we set $r=0$ and treat the factor involving $F$ using the mean value theorem in combination with Assumption~\ref{ass:reg}\ref{ass:derivative-bounded}, Lemma~\ref{lem:appr}\ref{lem:appr:time-reg} and~\ref{lem:appr:time-error} to obtain a bound in terms of $$ \lambda_{N+1}^{-\kappa/2} \lambda_{K}^{(\min(2\kappa,1)-\min(\kappa,1))/2} M^{-\min(2\kappa,1)/2} \le C \lambda_{N+1}^{-(\min(2\kappa,\kappa+1)-\min(2\kappa,1))/2} M^{-\min(2\kappa,1)/2}. $$ Here we also made use of the coupling assumption $K = DN$. On the other hand, we set $r=\kappa$ and treat this factor using Assumption~\ref{ass:reg}\ref{ass:difference-regularity-transfer} to obtain a bound in terms of $ \lambda_{N+1}^{-\kappa}$.
	The two other factors are in both cases handled by Lemma~\ref{lem:SP-estimates}\ref{lem:SP-estimates:proj-bound} and~\ref{lem:SP-estimates:semigroup-analyticity}.

	To handle $\|\mathrm{I}^2_F\|_{L^2(\Omega,H)}$, we use, on the one hand, the mean value theorem a total of four times to obtain
	\begingroup
	\allowdisplaybreaks
	\begin{align*}
		&F(X^L_K(\lfloor s \rfloor_{M^{-1}})) - F(X^L_N(\lfloor s \rfloor_{M^{-1}})) - F(X^M_K(\lfloor s \rfloor_{M^{-1}})) + F(X^M_N(\lfloor s \rfloor_{M^{-1}})) \\
		&\hspace{0.5em}= F(X^L_K(\lfloor s \rfloor_{M^{-1}})) - F(X^L_N(\lfloor s \rfloor_{M^{-1}}) + X^M_K(\lfloor s \rfloor_{M^{-1}}) - X^M_N(\lfloor s \rfloor_{M^{-1}})) \\
		&\qquad+ \Big(\big(F(X^L_N(\lfloor s \rfloor_{M^{-1}}) + X^M_K(\lfloor s \rfloor_{M^{-1}}) - X^M_N(\lfloor s \rfloor_{M^{-1}})) - F(X^L_N(\lfloor s \rfloor_{M^{-1}}))\big) \\
		&\hspace{4em}- \big(F(X^M_K(\lfloor s \rfloor_{M^{-1}})) - F(X^M_N(\lfloor s \rfloor_{M^{-1}}))\big)\Big) \\
		&\hspace{0.5em}= \int^1_0 U_s(\lambda)  \cE^{L,M}_{K,N}(s) \dd \lambda\\
		&\quad+ \int^1_0 \int^1_0 \tilde U_s(\lambda, \tilde \lambda)  \big(X^L_N(\lfloor s \rfloor_{M^{-1}}) - X^M_N(\lfloor s \rfloor_{M^{-1}}) \big) \big(X^M_K(\lfloor s \rfloor_{M^{-1}}) - X^M_N(\lfloor s \rfloor_{M^{-1}}) \big) \dd \lambda \dd \tilde \lambda,
	\end{align*}%
	\endgroup
	where
	\begin{align*}
		U_s(\lambda) = F'\Big( (1-\lambda) \big(X^L_N(\lfloor s \rfloor_{M^{-1}}) + X^M_K(\lfloor s \rfloor_{M^{-1}}) - X^M_N(\lfloor s \rfloor_{M^{-1}})\big) + \lambda X^L_K(\lfloor s \rfloor_{M^{-1}})  \Big)
	\end{align*}
	and
	\begin{align*}
		\tilde U_s(\lambda, \tilde \lambda) = F''\Big(X^M_N(\lfloor s \rfloor_{M^{-1}}) &+ \lambda \big(X^L_N(\lfloor s \rfloor_{M^{-1}})-X^M_N(\lfloor s \rfloor_{M^{-1}})\big) \\
		&+ \tilde \lambda \big(X^M_K(\lfloor s \rfloor_{M^{-1}})-X^M_N(\lfloor s \rfloor_{M^{-1}})\big) \Big).
	\end{align*}
	Then, using Assumptions~\ref{ass:reg}\ref{ass:derivative-bounded} and~\ref{ass:second-derivative-bounded} along with  Lemma~\ref{lem:SP-estimates}\ref{lem:SP-estimates:semigroup-analyticity},  Lemma~\ref{lem:appr}\ref{lem:appr:space-error} and Lemma~\ref{lem:appr}\ref{lem:appr:time-error}, we see that
	\begin{align*}
		\|\mathrm{I}^2_F\|_{L^2(\Omega,H)} &\le C \Bigg( \int^{\lfloor t \rfloor_{M^{-1}}}_0 \|\cE^{L,M}_{K,N}(s)\|_{L^2(\Omega,H)} \dd s \\
		&\quad+ \int^{\lfloor t \rfloor_{M^{-1}}}_0 s^{-\eta/2} \|X^L_N(\lfloor s \rfloor_{M^{-1}}) - X^M_N(\lfloor s \rfloor_{M^{-1}})\|_{L^4(\Omega,H)} \\
		&\hspace{8em}\times\|X^M_K(\lfloor s \rfloor_{M^{-1}}) - X^M_N(\lfloor s \rfloor_{M^{-1}})\|_{L^4(\Omega,H)} \dd s \Bigg)\\ &\le C \Bigg(\int^{\lfloor t \rfloor_{M^{-1}}}_0 \|\cE^{L,M}_{K,N}(s)\|_{L^2(\Omega,H)} \dd s \\
		&\quad + \lambda_{N+1}^{-(\min(2\kappa,\kappa+1)-\min(2\kappa,1))/2} M^{-\min(2\kappa,1)/2}\Bigg).
	\end{align*}
	On the other hand, we can use the mean value theorem, H\"older's inequality and Assumption~\ref{ass:reg}\ref{ass:derivative-negnorm} in the split
	\begin{align*}
		&\| F(X^L_K(\lfloor s \rfloor_{M^{-1}})) - F(X^L_N(\lfloor s \rfloor_{M^{-1}})) - F(X^M_K(\lfloor s \rfloor_{M^{-1}})) + F(X^M_N(\lfloor s \rfloor_{M^{-1}})) \|_{L^2(\Omega,\dot{H}^{-\eta})} \\
		&\quad\le \| F(X^L_K(\lfloor s \rfloor_{M^{-1}})) - F(X^L_N(\lfloor s \rfloor_{M^{-1}}))\|_{L^2(\Omega,\dot{H}^{-\eta})} \\
		&\qquad+ \|F(X^M_K(\lfloor s \rfloor_{M^{-1}})) - F(X^M_N(\lfloor s \rfloor_{M^{-1}})) \|_{L^2(\Omega,\dot{H}^{-\eta})} \\
		&\quad\le C  \Big(\| X^L_K(\lfloor s \rfloor_{M^{-1}}) - X^L_N(\lfloor s \rfloor_{M^{-1}})\|_{L^\frac{10}{3}(\Omega,\dot{H}^{-\kappa})} \\
		&\hspace{4em}+ \|X^M_K(\lfloor s \rfloor_{M^{-1}}) - X^M_N(\lfloor s \rfloor_{M^{-1}}) \|_{L^\frac{10}{3}(\Omega,\dot{H}^{-\kappa})} \Big)
	\end{align*}
	so that by  Lemma~\ref{lem:SP-estimates}\ref{lem:SP-estimates:semigroup-analyticity} and Lemma~\ref{lem:appr}\ref{lem:appr:neg-space-error}, $\|\mathrm{I}^2_F\|_{L^2(\Omega,H)} \le C \lambda_{N+1}^{-\kappa}$.

	For the term $\mathrm{I}^3_F$, we argue in the same way as for $\mathrm{I}^2_F$: first, the mean value theorem yields
	\begingroup
	\allowdisplaybreaks
	\begin{align*}
          &F(X^L_K(\lfloor s \rfloor_{L^{-1}})) - F(X^L_N(\lfloor s \rfloor_{L^{-1}}))
		- F(X^L_K(\lfloor s \rfloor_{M^{-1}})) + F(X^L_N(\lfloor s \rfloor_{M^{-1}}))  \\
		&= \int^1_0 V_s(\lambda)  \cH^{L,M}_{K,N}(s) \dd \lambda\\
		&\qquad+ \int^1_0 \int^1_0 \tilde V_s(\lambda, \tilde \lambda)  \big(X^L_N(\lfloor s \rfloor_{L^{-1}}) - X^L_N(\lfloor s \rfloor_{M^{-1}}) \big) \big(X^L_K(\lfloor s \rfloor_{M^{-1}}) - X^L_N(\lfloor s \rfloor_{M^{-1}}) \big) \dd \lambda \dd \tilde \lambda,
	\end{align*}%
	\endgroup
	where
	\begin{align*}
	V_s(\lambda) = F'\Big( (1-\lambda) \big(X^L_N(\lfloor s \rfloor_{L^{-1}}) + X^L_K(\lfloor s \rfloor_{M^{-1}}) - X^L_N(\lfloor s \rfloor_{M^{-1}})\big) + \lambda X^L_K(\lfloor s \rfloor_{L^{-1}})  \Big)
	\end{align*}
	and
	\begin{align*}
	\tilde V_s(\lambda, \tilde \lambda) = F''\Big(X^L_N(\lfloor s \rfloor_{M^{-1}}) &+ \lambda \big(X^L_K(\lfloor s \rfloor_{M^{-1}})-X^L_N(\lfloor s \rfloor_{M^{-1}})\big) \\
	&+ \tilde \lambda \big(X^L_N(\lfloor s \rfloor_{L^{-1}})-X^L_N(\lfloor s \rfloor_{M^{-1}})\big) \Big).
	\end{align*}
	We now repeat the arguments for $\mathrm{I}^2_F$, replacing the use of  Lemma~\ref{lem:appr}\ref{lem:appr:time-error} with Lemma~\ref{lem:appr}\ref{lem:appr:time-reg}. We also apply Lemmas~\ref{lem:SP-estimates}\ref{lem:SP-estimates:semigroup-analyticity}, \ref{lem:appr}\ref{lem:appr:space-reg} and \ref{lem:second-order-diff}, along with Assumption~\ref{ass:reg}\ref{ass:derivative-negnorm}, to arrive at the bound
	\begin{align*}
		&\Big\| \int^{\lfloor t \rfloor_{M^{-1}}}_0 P_{N} S(\lfloor t \rfloor_{M^{-1}}-s) \int^1_0 V_s(\rho) \cH^{L,M}_{K,N}(s) \dd \rho \dd s \Big\|_{L^2(\Omega,H)} \\
		&\le C \int^{\lfloor t \rfloor_{M^{-1}}}_0 (\lfloor t \rfloor_{M^{-1}}-s)^{-\eta/2}\|\cH^{L,M}_{K,N}(s) \|_{L^4(\Omega,\dot{H}^{-\min(\kappa,1)})}  \\ &\hspace{4em}\times \big(1 + \max( \| X^L_N(\lfloor s \rfloor_{L^{-1}}) + X^L_K(\lfloor s \rfloor_{M^{-1}}) - X^L_N(\lfloor s \rfloor_{M^{-1}}) \|_{L^4(\Omega,\dot{H}^{\min(\kappa,1)})},\\
		&\hspace{10em}\| X^L_K(\lfloor s \rfloor_{L^{-1}}) \|_{L^4(\Omega,\dot{H}^{\min(\kappa,1)})}) \big)\dd s \\
		&\le C   \lambda_{N+1}^{-(\min(2\kappa,\kappa+1)-\min(2\kappa,1))/2} M^{-\min(2\kappa,1)/2}.
	\end{align*}
	Arguing in the same way as for $I^2_F$ we can also derive the bound $\|\mathrm{I}^3_F\|_{L^2(\Omega,H)} \le C \lambda_{N+1}^{-\kappa}$.

	Next, It\^o isometry and~\eqref{eq:theta-G-bound} yields an immediate bound
	\begin{align*}
	\| I^1_W \|_{L^2(\Omega,H)}^2 \le C \int^{\lfloor t \rfloor_{M^{-1}}}_0 (t-s)^{\delta-1} \| \cE^{L,M}_{K,N}(s) \|_{L^2(\Omega,H)}^2 \dd s,
	\end{align*}
	and similarly,
	\begin{equation*}
		\| I^2_W \|_{L^2(\Omega,H)}^2 \le C \sup_s \| \cH^{L,M}_{K,N}(s) \|_{L^2(\Omega,\dot{H}^{-\kappa})}^2.
	\end{equation*}
	Summing up, using also Lemma~\ref{lem:second-order-diff} and the H\"older inequality, we have shown on the one hand that
	\begin{align*}
		\| \cE^{L,M}_{K,N}(t) \|_{L^2(\Omega,H)}^2 &\le C \Bigg( \lambda_{N+1}^{-(\min(2\kappa,\kappa+1)-\min(2\kappa,1))/2} M^{-\min(2\kappa,1)/2} \\
		&\quad+ \int^{\lfloor t \rfloor_{M^{-1}}}_0 (t-s)^{\min(\kappa+\delta-1,0)} \| \cE^{L,M}_{K,N}(s) \|_{L^2(\Omega,H)}^2 \dd s\Bigg),
	\end{align*}
	and on the other hand that
	\begin{equation*}
		\| \cE^{L,M}_{K,N}(t) \|_{L^2(\Omega,H)}^2 \le C \Bigg(\lambda_{N+1}^{-2\kappa} + \int^{\lfloor t \rfloor_{M^{-1}}}_0 (t-s)^{\min(\kappa+\delta-1,0)}\| \cE^{L,M}_{K,N}(s) \|_{L^2(\Omega,H)}^2 \dd s\Bigg).
	\end{equation*}
	The proof is now completed by an appeal to the discrete Gronwall inequality.
\end{proof}

\begin{remark}
	\label{rem:remark-G-nemytskij}
	The reason for why we do not consider the composition operators $G$ discussed in Remark~\ref{rem:remark-G} is due to the fact that we would have to introduce a truncation of the noise. After applying the same type of mean value theorem arguments as for $F$, we would end up with terms like
	\begin{equation*}
		\sum_{j=1}^{\infty} \int^{\lfloor t \rfloor_{M^{-1}}}_{0} \| S(\lfloor t \rfloor_{M^{-1}} - s) G'(Y_s)\big(X^L_N(\lfloor s \rfloor_{L^{-1}})-X^L_N(\lfloor s \rfloor_{M^{-1}}),(P_N-P_K)\mu^{1/2}_j e_j\big)\|^2 \dd s,
	\end{equation*}
	where $Y_s$ is a term with the same regularity as $X^L_N$ and $(G'(u)(v,w))(x) = g'(u(x))v(x)w(x)$ for $u,v,w \in H$ and almost every $x \in \cD$. The main tool we have to treat this term is Proposition~\ref{prop:F-Fdoubleprime}. This introduces a restriction to spatial dimension $d=1$ regardless of how smooth the derivative $g'$ of $g$ is assumed to be. Even then, it is not possible to obtain product bounds of the form $\lambda_{N+1}^{-\kappa/2} M^{-1/2}$ or better, due to the fact that a negative norm would have to be used on $(P_N-P_K)\mu^{1/2}_j e_j$. This would necessitate a positive norm bound on $X^L_N(\lfloor s \rfloor_{L^{-1}})-X^L_N(\lfloor s \rfloor_{M^{-1}})$ which would decrease the temporal convergence rate.
\end{remark}

\section{Numerical examples} \label{sec:numerical_examples}
In this section we numerically study the convergence properties,
approximation error and cost for MIMCEI and MLMCEI applied to two
SPDEs in settings where Assumption~\ref{ass:reg} holds with $\kappa=1$
and Assumption~\ref{ass:qoi} holds with $\psi \in \cL(H,U)$ and $U=H$.
The proof of Theorem~\ref{thm:mimc-cost-error} then shows that
the performance of MIMCEI is critically linked to the
multiplicative convergence property
\[
\|\Delta_{\boldEll} X \|_{L^2(\Omega,H)}
\lesssim  N_{\ell_2}^{-\beta_2/2} \min(M_{\ell_1}^{-\beta_1/2}, N_{\ell_2}^{-\beta_2/2})
\eqsim 2^{-\max(\beta_1 \ell_1/2 + \beta_2  \ell_2/2, \,  \beta_2  \ell_2)},
\]
with the double difference $\Delta_{\boldEll} X$
defined in~\eqref{eq:mimcei} (here for the case $\Psi(X) = X(T)$).
We will numerically estimate the rates $(\beta_1, \beta_2)$, and show that the results are consistent with the theoretical rates in
  Lemma~\ref{lem:strong-error-coupling} with $\kappa=1$ and the identity mapping $\psi \in \cL(H)$. For this purpose, we let
\[
  e_F(\ell_1,\ell_2) := \sqrt{\sum_{i=1}^{m} \frac{ \|\Delta_{\boldEll} X^{(i)} \|^2}{m} }
\]
be a Monte Carlo approximation of $\|\Delta_{\boldEll} X\|_{L^2(\Omega,H)}$
using $m=10^4$ iid samples, and numerically
validate~\eqref{eq:strong-error-coupling-assumption}
by finding a function
\[
  p(\ell_1,\ell_2) :=  2^{- \max(\bar \beta_1 \ell_1/2 +  \bar \beta_2 \ell_2/2 + c_1, \, \bar \beta_3 \ell_2 + c_2)}
\]
with parameters $(\bar \beta_1, \bar \beta_2, \bar \beta_3)$ that are greater than or approximately equal to $\beta_1, \beta_2$ and $\beta_2$, respectively, and for which it
additionally holds that $p$ dominates $e_F$, meaning $p \ge e_F$. We interpret
this as numerical validation of $e_F \lesssim 2^{- \max(\beta_1 \ell_1/2 + \beta_2 \ell_2/2, \, \beta_2 \ell_2)}$.

\subsection{Linear drift and multiplicative noise} \label{subsec:linDriftMultNoise}
We consider the linear SPDE
\begin{equation}\label{eq:ex1_linSpde}
  dX(t) = (-A X(t) + X(t)) dt + (I + G X(t)) dW(t) \qquad t \in (0,1]
\end{equation}
where $A$ has eigenpairs $((\lambda_k,e_k))_{k\ge1}$ with $\lambda_k = 0.2 k^{4/3}$.
The $Q$-Wiener process has coefficients
$\mu_{k} = k^{-1.01}$ and the initial condition is
\[
X(0) = \sum_{k=1}^\infty \xi_k k^{-2} e_k
\]
where $\xi_k \stackrel{iid}{\sim}N(0,1)$, and
$(Gu)v = \sum_{j = 1}^\infty \zeta_j \langle u, e_{j+1}\rangle \langle v, e_j \rangle e_j$
 with $\zeta_k = \mu_k^{-1/2}$. This is a problem setting with
$\kappa=1$ and double difference convergence rates $\beta_1=1$ and
$\beta_2 = \kappa \nu =
4/3$. Figure~\ref{fig:linDriftMultNoise_nu43_fit1} shows that
the numerical estimate of $\|\Delta_{\boldEll}
X\|_{L^2(\Omega,H)}$ is consistent with the rates predicted by
Lemma~\ref{lem:strong-error-coupling}.
\begin{figure}
  \center
  \includegraphics[width=0.85\textwidth]{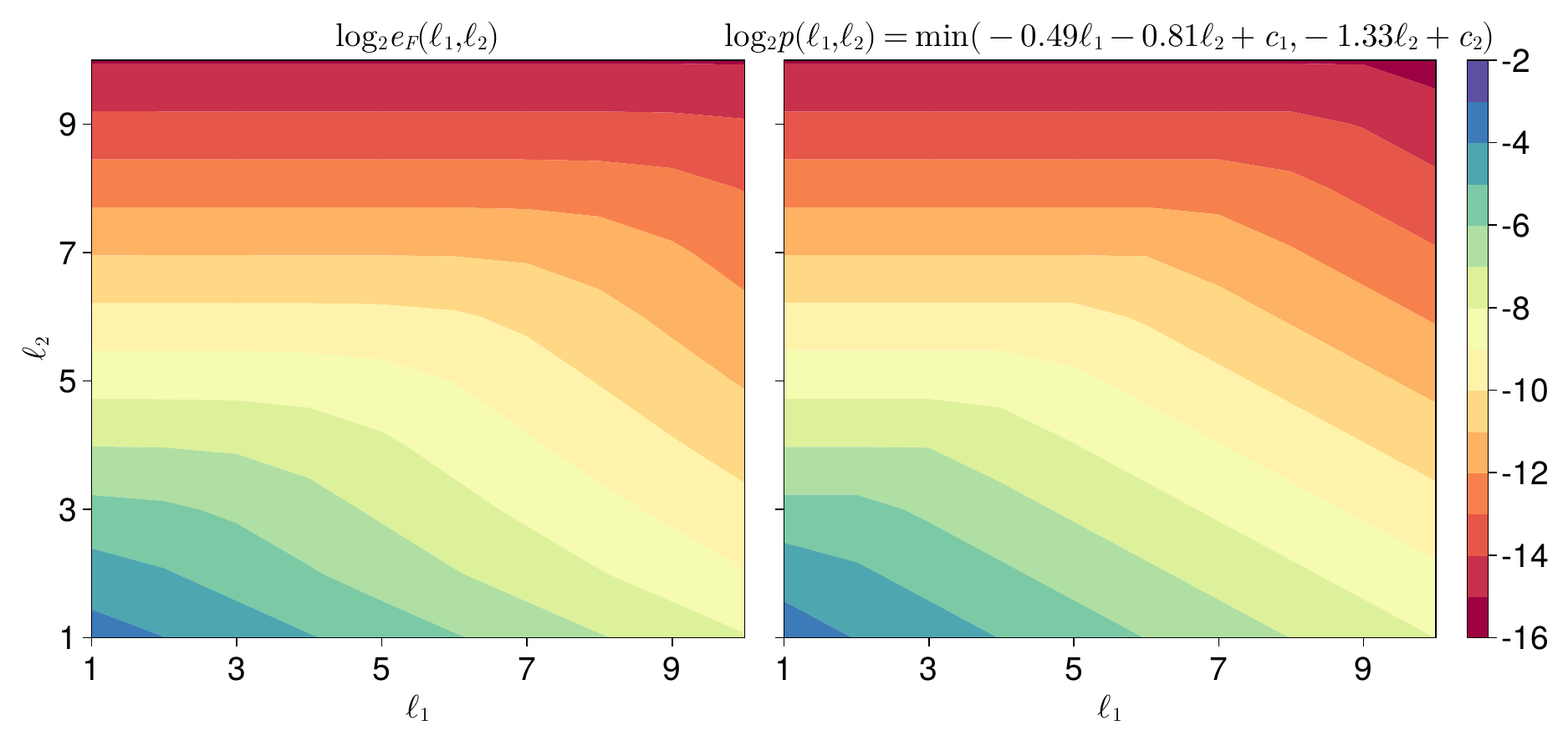}
  \caption{\textbf{Left:} Estimate of the double difference convergence rate for the SPDE~\eqref{eq:ex1_linSpde}.
  \textbf{Right:} A function $p(\ell_1,\ell_2) \approx \min(2^{- \beta_1 \ell_1/2 -  \beta_2 \ell_2/2 +c_1}, 2^{- \beta_2 \ell_2+c_2})$
  that appears to be a good global approximation of $e_F$.}
  \label{fig:linDriftMultNoise_nu43_fit1}
\end{figure}

MIMCEI is implemented with the parameter values in the proof of
Lemma~\ref{lem:general-mimc-cost-error} with $\beta_2>\beta_1=1$ and the possibly
non-sharp weak rates $\alpha_j = \beta_j/2$. More particularly, we choose the index set
of the MIMCEI estimator as
\begin{equation}\label{eq:index-set-shape}
  \cI \equiv \cI_{L} := \{\boldEll \in \N_0^2 \mid \xi_1 \ell_1 + \xi_2 \ell_2 \le L \} \quad
  \text{with} \quad  \xi_j := \alpha_j  + (1-\beta_j)/10\,.
\end{equation}
with \(L = \max(\lceil \log_2(\varepsilon^{-1}) \rceil , 1)\) for a given error tolerance \(\varepsilon\).
We also set the number of samples at level \(\boldEll \in \cI_{L}\),
\[
  m_{\boldEll} \eqsim \left\lceil \frac{L \varepsilon^{-2}}{2^{\ell_{1} (1 + \beta_1)/2}
      2^{\ell_{2}(1+\beta_2)/2}} \right\rceil, \qquad \boldEll \in \cI.
\]
which is the same choice as \eqref{eq:ml-general}, up to a constant, since
\(\beta_{1}=1, \beta_{2}>1\). On the other hand, MLMCEI is implemented with the
parameter values in the proof of Theorem~\ref{thm:mlmc-cost-error}, also with
a possibly non-sharp weak rate, $\overline{\alpha} = 1/2$. Underestimation of
$\overline{\alpha}$ would affect the performance of MLMCEI negatively, but it is a
challenging and intractable problem to estimate weak convergence rates for
parabolic SPDEs. Our crude numerical estimates of the weak rate,
which we will not detail here, indicate that it indeed is close to
$1/2$ for this problem setting. For further details on parameters and the
computer code, see the Julia programming language implementation
in~\url{https://github.com/haakonahMatata/SEMIMCLINEAR_SPDE}.

Figure~\ref{fig:linDriftMultNoise_nu43_costVsError} shows that the two errors $\|\MIMCest - \E[X(1)]\|^2_{L^2(\Omega)}$ and ${\|\MLMCest
  - \E[X(1)]\|^2_{L^2(\Omega)}}$, for one realization of MIMCEI and MLMCEI each,
are both $\cO(\varepsilon^2)$ and that the computational cost of both methods grows with
rates that are consistent with theory. The reference solution is $\E[X(1)] =0$ for this
SPDE. Figure~\ref{fig:linDriftMultNoise_nu2_costVsError} shows a similar
performance comparison for the SPDE~\eqref{eq:ex1_linSpde} in the smoother
setting when $\nu =2$, where we expect a smaller performance discrepancy due to
the improved regularity.
\begin{figure}
  \center
  \includegraphics[width=0.49\textwidth]{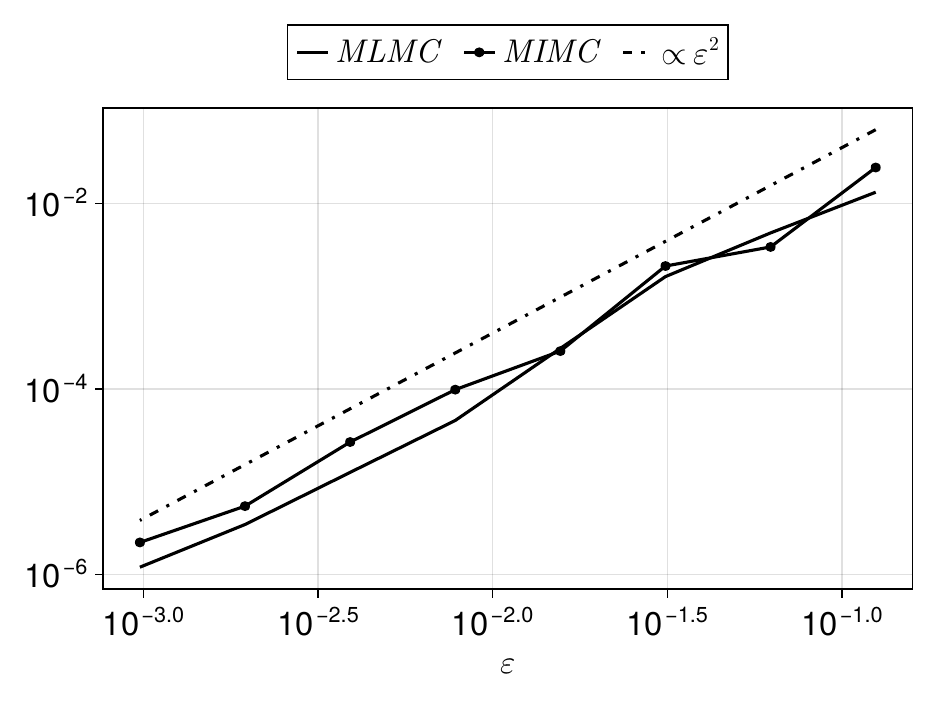}
  \includegraphics[width=0.49\textwidth]{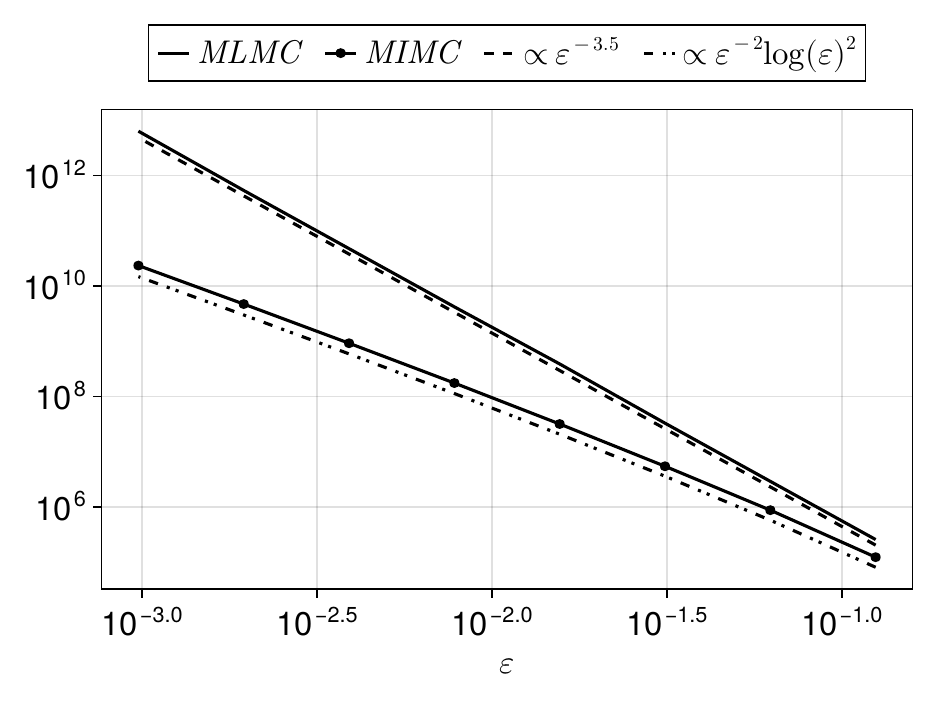}
  \caption{Comparison of performance of MIMCEI and MLMC for the
    SPDE~\eqref{eq:ex1_linSpde} with $\nu =4/3$.  \textbf{Left:}
    Approximation error. \textbf{Right:} Computational cost.}
  \label{fig:linDriftMultNoise_nu43_costVsError}
\end{figure}
\begin{figure}
  \center
  \includegraphics[width=0.49\textwidth]{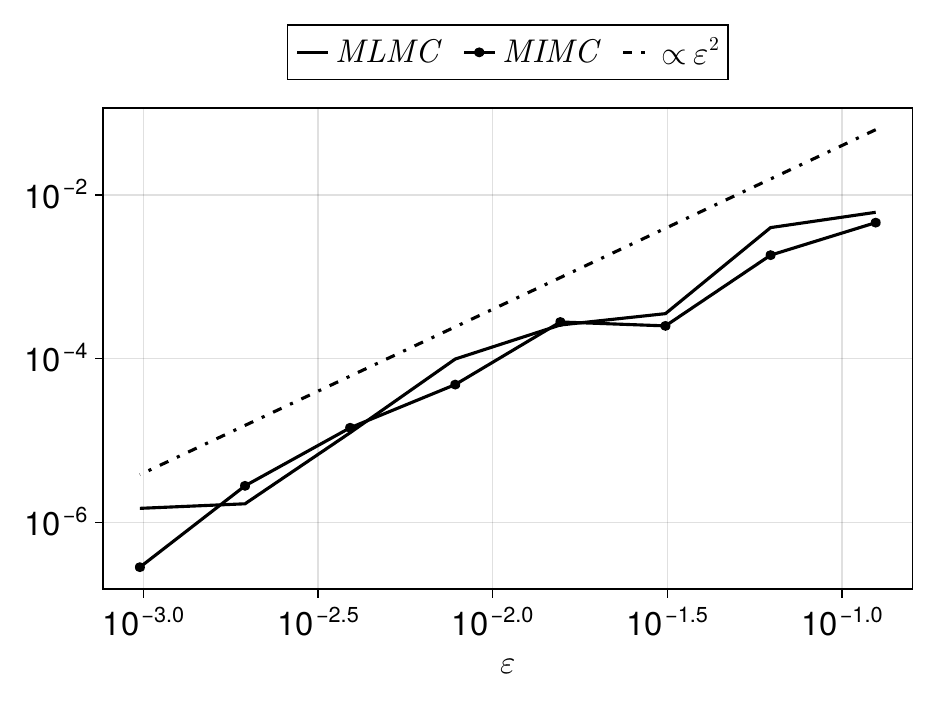}
  \includegraphics[width=0.49\textwidth]{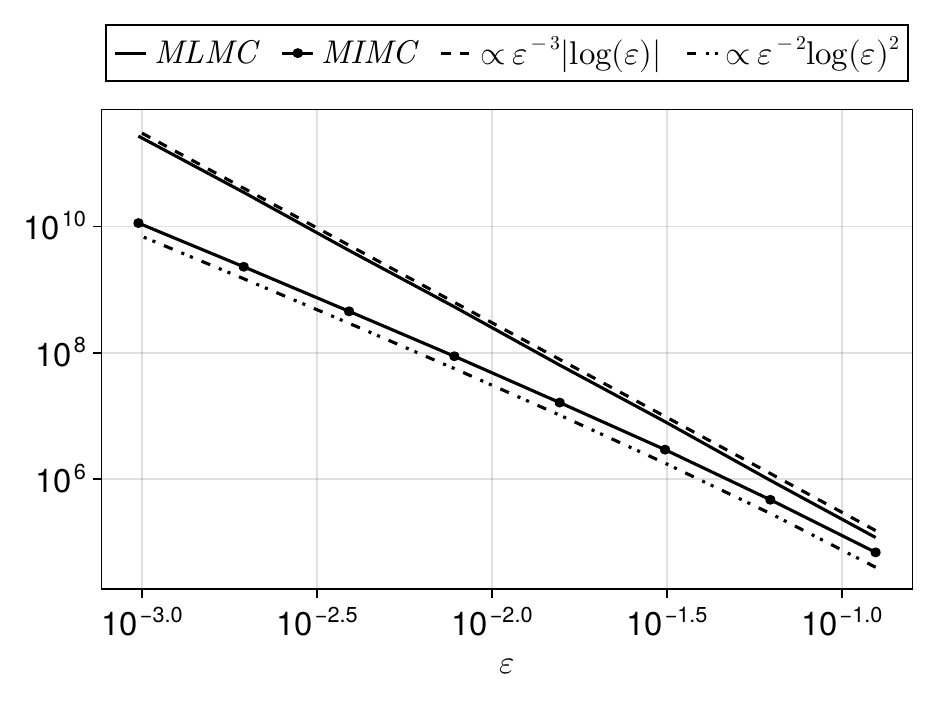}
  \caption{Comparison of performance of MIMCEI and MLMC for the SPDE~\eqref{eq:ex1_linSpde}
  with $\nu =2$.
  \textbf{Left:} Approximation error. \textbf{Right:} Computational cost.}
  \label{fig:linDriftMultNoise_nu2_costVsError}
\end{figure}

Lastly, we compare the performance of MIMCEI in
Theorem~\ref{thm:mimc-cost-error} with the reduced-samples MIMCEI method,
as mentioned in Remark~\ref{rem:reduced-samples}, in the low-regularity setting $\nu=10/9$.
For the given SPDE, the improved spatial convergence rate is
$\beta_3 = 2 \kappa \nu$. Both MIMCEI methods are implemented with the
possibly non-sharp weak rates $\alpha_1 =1/2$, $\alpha_2=\kappa \nu/2$ and $\alpha_3 = \kappa \nu$, and
MLMCEI is implemented with $\overline{\alpha} = 1/2$. In alignment with theory, the
numerical tests in Figure~\ref{fig:linDriftMultNoise_nu10Over9_costVsError}
show that the reduced-samples MIMCEI method asymptotically is roughly $10\%$
faster than the other one, that both MIMCEI methods are considerably faster
than MLMCEI, and that all methods estimate $\E[X(1)]$ with similar accuracy.
\begin{figure}
  \center
  \includegraphics[width=0.49\textwidth]{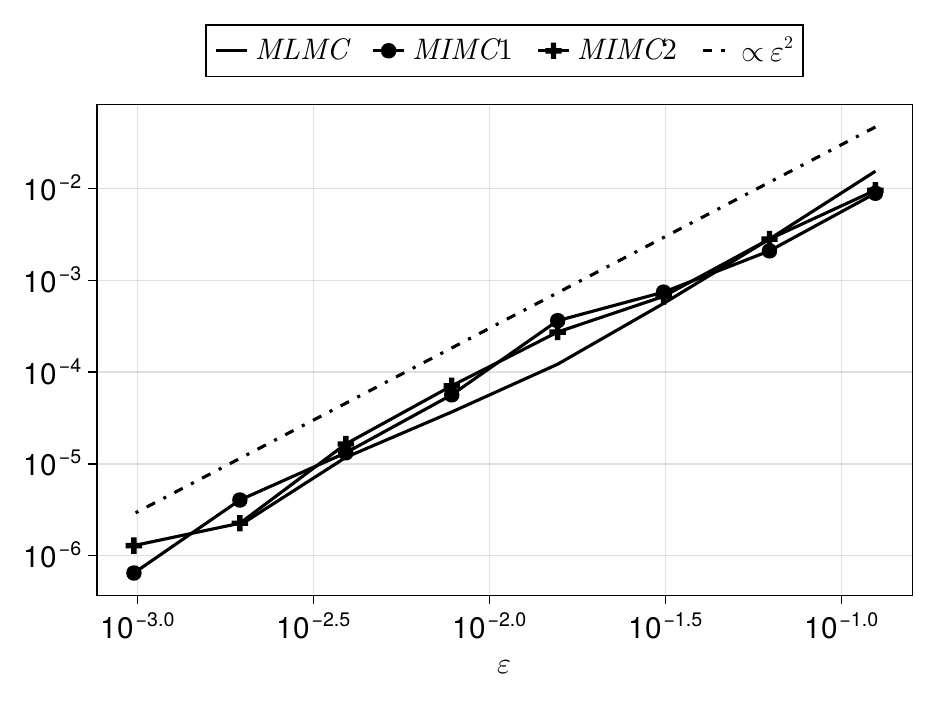}
  \includegraphics[width=0.49\textwidth]{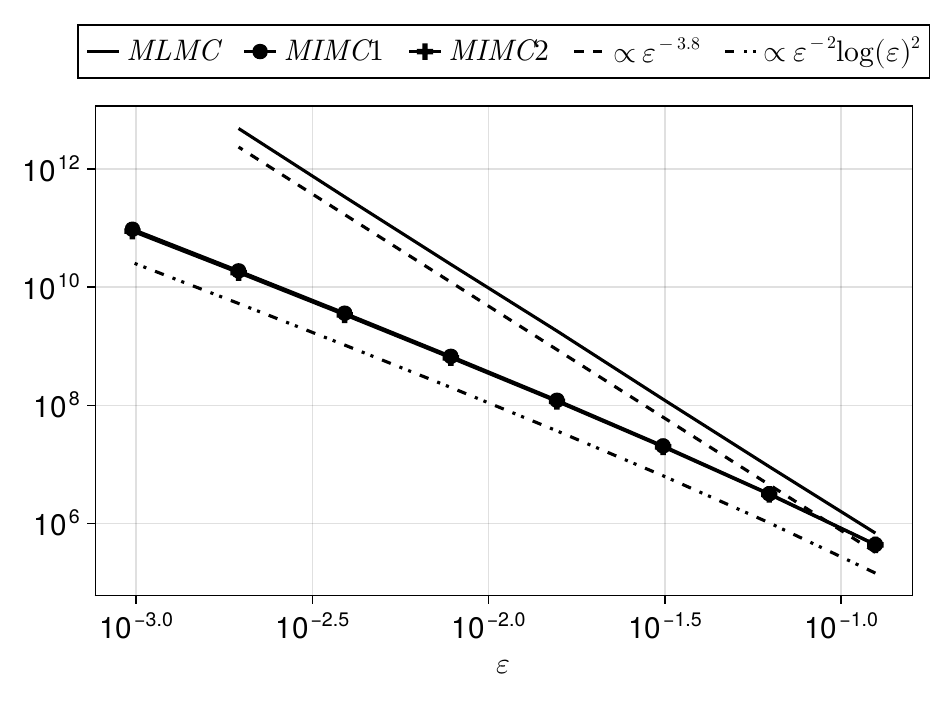}
  \caption{Comparison of performance of standard MIMCEI (labeled MIMC1),
    reduced-samples MIMCEI (labeled MIMC2), and MLMC applied to
    the SPDE~\eqref{eq:ex1_linSpde} with $\nu =10/9$.  \textbf{Left:}
    Approximation error. \textbf{Right:} Computational cost.}
  \label{fig:linDriftMultNoise_nu10Over9_costVsError}
\end{figure}

\subsection{Nonlinear drift and multiplicative noise}
We consider the linear SPDE
\begin{equation}\label{eq:ex2_nonlinearSpde}
dX(t) = (-A X(t) + F(X(t))) dt + (I + G X(t)) dW(t) \qquad t \in (0,1]
\end{equation}
on $H = L^2([0,1])$ where $A$ has eigenpairs $((\lambda_k,e_k))_{k\ge1}$
with $\lambda_k = k^{\nu}$, Fourier eigenbasis $e_k$, Nemytskii
operator $F(X(t,\cdot))(x) = \sin(X(t,x)) + X(t,x)$.  The $Q$-Wiener
process has coefficients $\mu_{k} = k^{-1.0001}$ and $X_0(x) =
2\min(x,1-x) \in \dot{H}^1$ (where this holds true for $\kappa =1$
combined with the two eigenvalue-exponent values $\nu \in \{4/3, 2\}$
that we consider in this example), and $(Gu)v = \sum_{j = 1}^\infty
\zeta_j \langle u, e_{j+2}\rangle \langle v, e_j \rangle e_j$ with
$\zeta_k = \mu_k^{-1/2}$. We use the fast Fourier transform (FFT) and
the inverse FFT to evaluate $F(X(t))$ in the numerical method.
See~\cite{CHLNT21} for details on FFT, and we note that possible error
contributions from approximating $F(X(t))$
by FFT/IFFT is not included in our theory. This is a setting with
$\kappa=1$ and double difference convergence rates $\beta_1=1$ and
$\beta_2 = \kappa \nu = 4/3$.

Figure~\ref{fig:nonlinearExampleFit1} shows that the numerical estimate of
$\|\Delta_{\boldEll} X\|_{L^2(\Omega,H)}$ is
consistent with theoretical convergence rates.
\begin{figure}
  \center
  \includegraphics[width=0.85\textwidth]{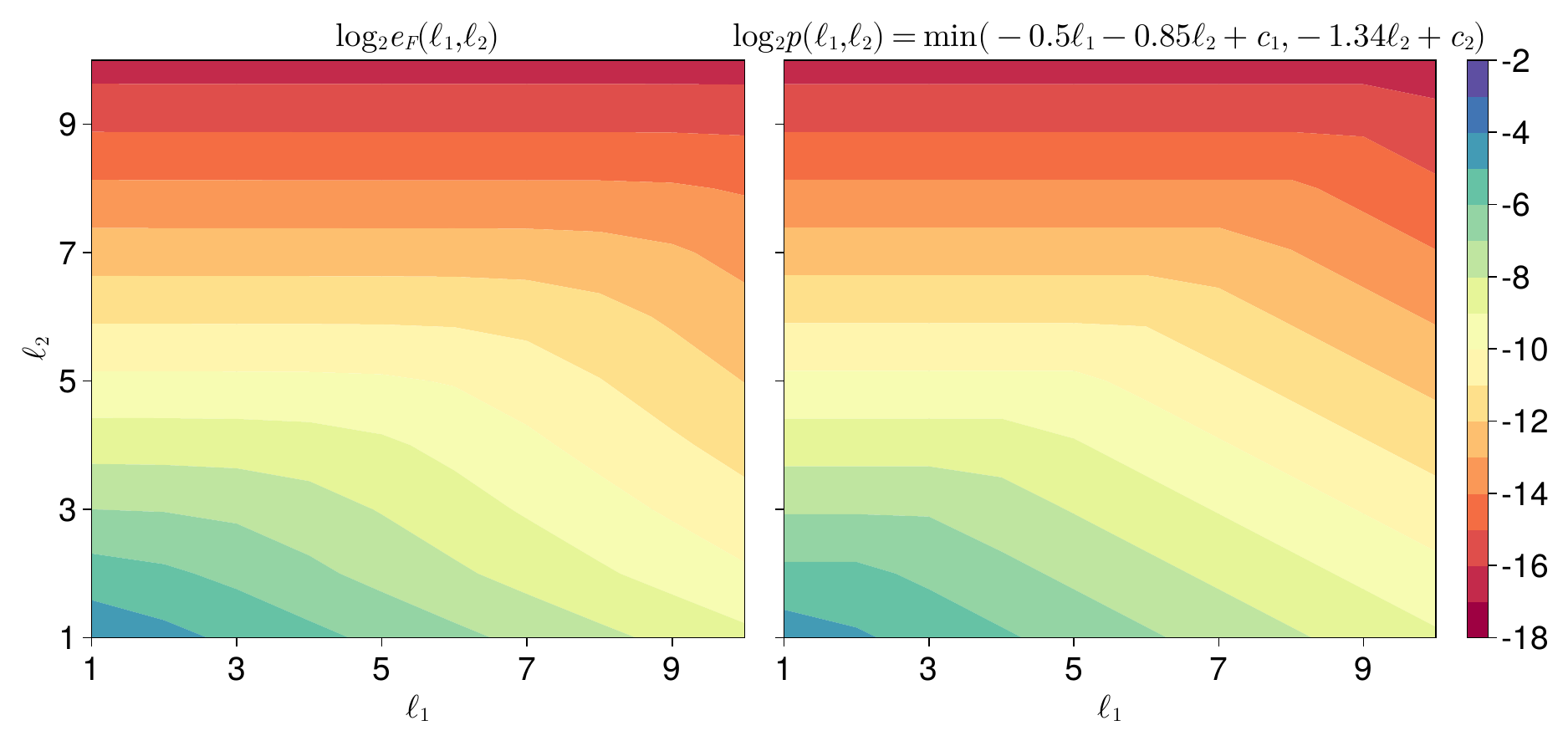}
  \caption{\textbf{Left:} Estimate of the double difference convergence rate for the SPDE~\eqref{eq:ex2_nonlinearSpde}.
  \textbf{Right:} A function $p(\ell_1,\ell_2) \approx \min(2^{- \beta_1 \ell_1/2 -  \beta_2 \ell_2/2 +c_1}, 2^{- \beta_2 \ell_2 +c_2})$
  that is a good global approximation of $e_F$.}
  \label{fig:nonlinearExampleFit1}
\end{figure}
Figure~\ref{fig:nonlinearShifted_nu43_costVsError} shows that the
approximation errors $\|\MIMCest - \E[X(1)]\|^2$ and $\|\MLMCest -
\E[X(1)]\|^2$ for one realization of MIMCEI and MLMCEI each,
both are $\cO(\varepsilon^2)$ and
that the computational cost has the growth rates expected from theory. In all
experiments for this SPDE, a pseudo-reference solution of $\E[X(1)]$ is
computed by MIMCEI with error tolerance $\varepsilon = 2^{-12}$.
Figure~\ref{fig:nonlinearShifted_nu2_costVsError} shows a similar performance
comparison for the SPDE~\eqref{eq:ex2_nonlinearSpde} in the smoother setting
when $\nu =2$, where we observe, as expected, a smaller discrepancy in
performance between the two methods. For these performance comparisons, let us
however note that MIMCEI is implemented with the possibly non-sharp weak rates
$\alpha_i = \beta_i/2$ and MLMCEI with $\overline{\alpha} = 1/2$. For the sake of fair
comparisons, crude numerical estimates indicate that the true weak convergence
rate may be near $1$ for MLMCEI. If one were to use $\overline{\alpha} =1$ instead
of $1/2$, then the computational cost of MLMCEI would reduce from $\cO(\varepsilon^{-2
  -2/(\kappa \nu))})$ to $\cO(\varepsilon^{-2 - 1/(\kappa \nu)})$. This is an improvement one should
exploit if one were able to estimate the weak rate reliably, but MLMCEI would
regardless still be notably less tractable than MIMCEI.

\begin{figure}
  \center
  \includegraphics[width=0.49\textwidth]{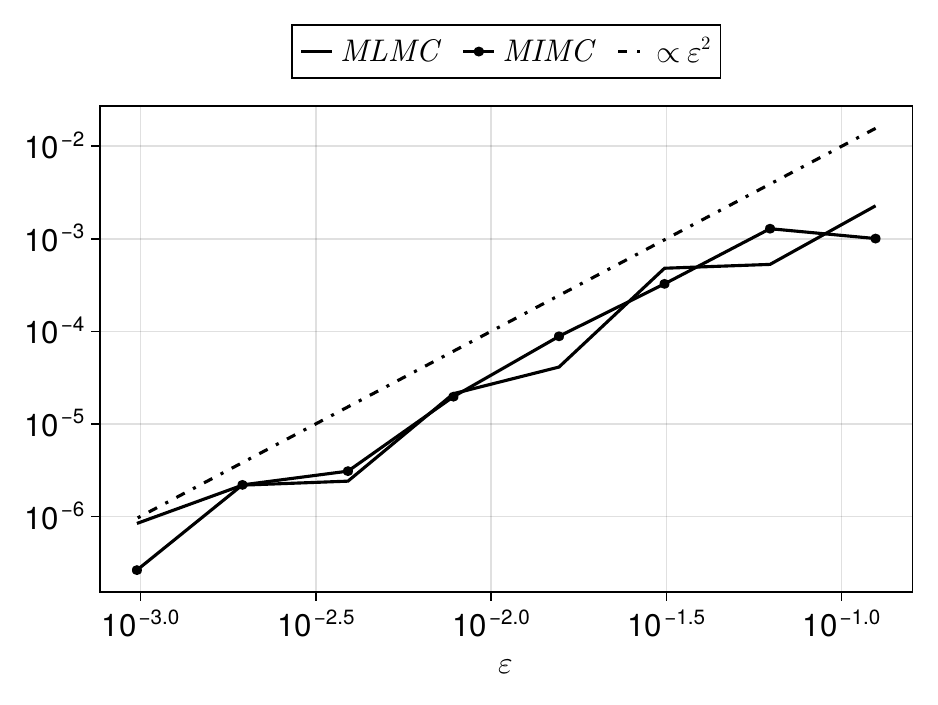}
  \includegraphics[width=0.49\textwidth]{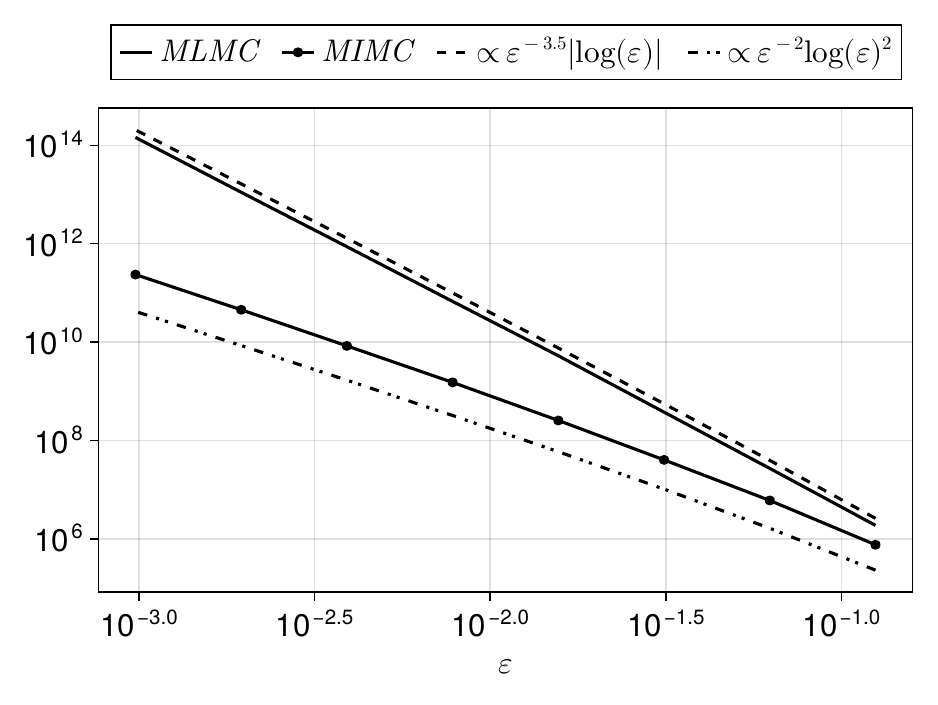}
  \caption{Comparison of performance of MIMCEI and MLMC for the SPDE~\eqref{eq:ex2_nonlinearSpde}
  with $\nu =4/3$.
  \textbf{Left:} Approximation error. \textbf{Right:} Computational cost.}
  \label{fig:nonlinearShifted_nu43_costVsError}
\end{figure}
\begin{figure}
  \center
  \includegraphics[width=0.49\textwidth]{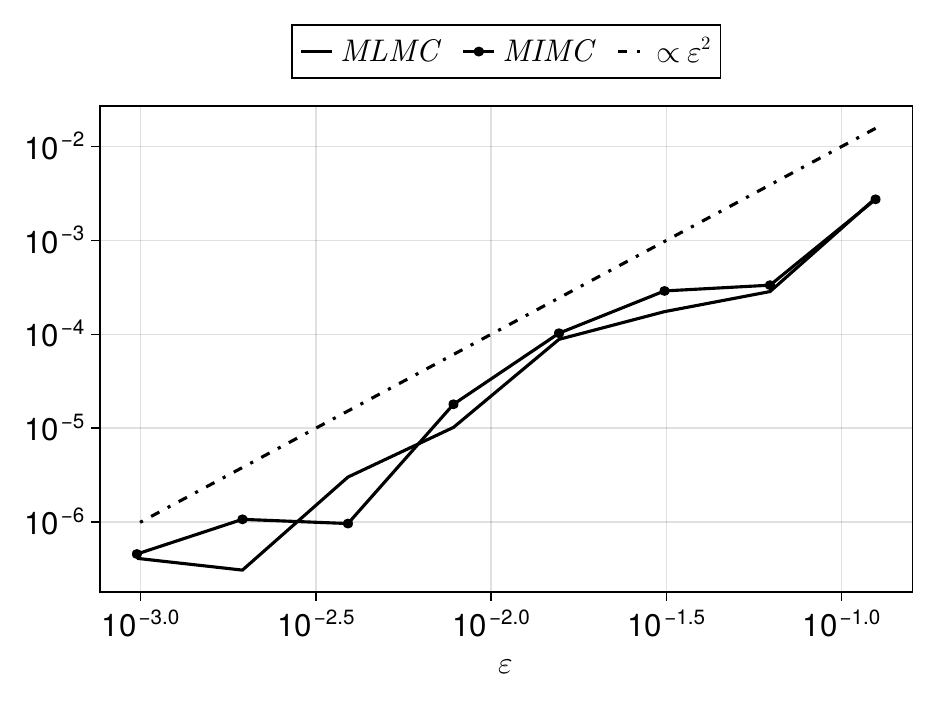}
  \includegraphics[width=0.49\textwidth]{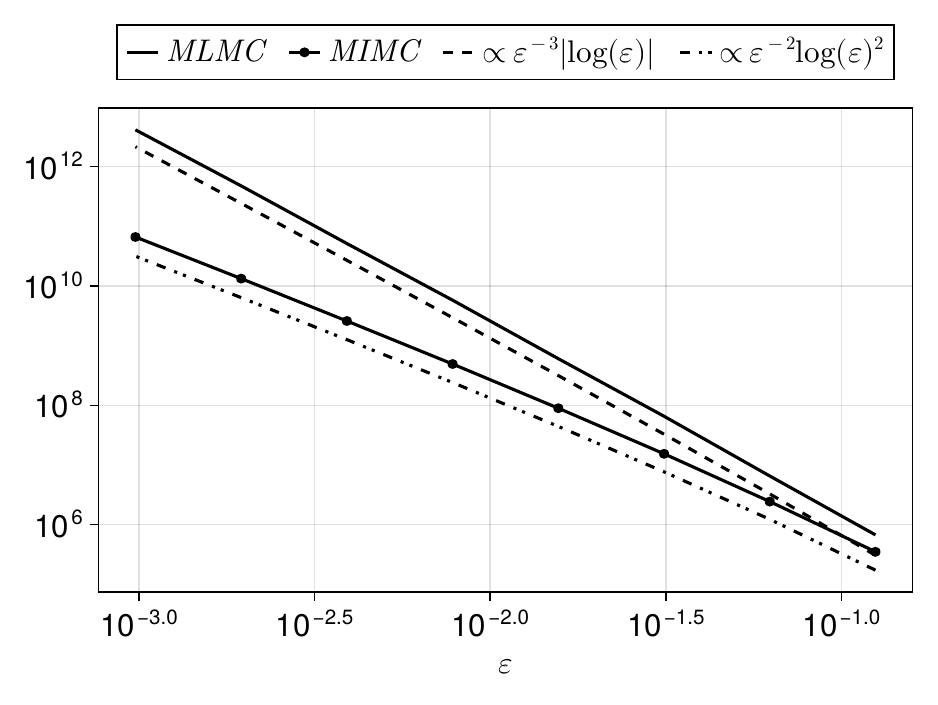}
  \caption{Comparison of performance of MIMCEI and MLMC for the SPDE~\eqref{eq:ex2_nonlinearSpde}
  with $\nu =2$.
  \textbf{Left:} Approximation error. \textbf{Right:} Computational cost.}
  \label{fig:nonlinearShifted_nu2_costVsError}
\end{figure}

\section{Conclusion}\label{sec:conclusion}

We have developed an MIMC method for semilinear SPDEs and conducted a
rigorous theoretical analysis of its convergence properties
and efficiency. Both theoretical findings and numerical experiments
confirm that our method significantly outperforms existing tractable
MLMC methods. These gains in efficiency extend
the frontier of feasible computer-aided studies of statistical
studies of semilinear SPDEs.

MLMC for SPDEs and MIMC for SDEs have been successfully utilized in filtering to
produce efficient ensemble Kalman filtering methods and particle
filters~\cite{HS22,CHLNT21,JLX21}. A potential extension of our work involves
investigating the integration of the MIMC methodology developed herein with
filtering methods for semilinear SPDEs to achieve a broader scope of
applications and enhanced performance. On the other hand, advanced splitting
method extensions of the exponential integrator for SPDEs with additive noise
and highly nonlinear drift coefficients have been developed and studied
theoretically for settings that are outside the scope of this paper,
cf.~\cite{BG20,DGK24}. Another extension of our work is to consider the
application of MIMC to such SPDEs when combined with splitting exponential
integrator methods.

An open question, outlined in Remark~\ref{rem:reduced-samples}, is whether it
is possible to leverage the improved spatial convergence rates in
Lemma~\ref{lem:strong-error-coupling} to obtain an MIMC method with improved
computational complexity. The key challenge lies in efficiently handling the
slowly-converging MIMC samples along the $\ell_2$-axis, associated with the slow
convergence when \(\ell_{1}=0\) in~\eqref{eq:l2-mix-diff-sharp-bounds} and the
corresponding bound on the weak error. It would, for example, be interesting
to explore if multilevel Richardson-Romberg extrapolation~\cite{LP17} is a
suitable tool to overcome this hurdle.

\bibliographystyle{abbrv}

\appendix
\section{Proofs in Section~\ref{sec:mimc_method}}\label{sec:mimc-proofs}

\begin{lemma}\label{lem:sum-bound}
  For \(\eta_{1}, \eta_{2}, \xi_{1}, \xi_{2}, L > 0\), we have, for hidden
  constants independent of \(L\),
  \begin{equation}\label{eq:sum-exp-compliment}
    \begin{split}
      \sum_{\lbrace (\ell_{1}, \ell_{2}) \in \N^{2} \: : \: \xi_{1} \ell_{1} + \xi_{2} \ell_{2} > L \rbrace}
      2^{-\eta_{1}\ell_{1} - \eta_{2}\ell_{2}}
      & \lesssim
        \begin{cases}
          2^{-\min(\eta_{1}/\xi_{1}, \eta_{2}/\xi_{2}) L} &  \eta_{1}/ \xi_{1} \neq \eta_{2}/ \xi_{2},\\
          L \,
          2^{-\min(\eta_{1}/\xi_{1}, \eta_{2}/\xi_{2}) L} & \textnormal{otherwise}.  \\
        \end{cases}
    \end{split}
  \end{equation}
  And, for any \(\eta_{1}, \eta_{2} \in \R\),
  \begin{equation}\label{eq:sum-exp}
    \begin{split}
      \sum_{\lbrace
      (\ell_{1}, \ell_{2}) \in \N^{2} \: : \:
      \xi_{1} \ell_{1} + \xi_{2} \ell_{2} \leq L
      \rbrace}
      2^{\eta_{1}\ell_{1} + \eta_{2}\ell_{2}}
      & \lesssim 2^{L \max(0, \eta_{1}/\xi_{1}, \eta_{2}/\xi_{2})} (L+1)^{v_{\eta}},
    \end{split}
  \end{equation}
  where
  \begin{equation}\label{eq:sum-exp-log-powers}
    v_{\eta} =
    \begin{cases}
      2 & \eta_{1}=\eta_{2}=0, \\
      1 & \eta_{2} < \eta_{1}=0
          \textnormal{ or }
          \eta_{1} < \eta_{2}=0
          \textnormal{ or }
          \frac{\eta_{1}}{\xi_{1}} = \frac{\eta_{2}}{\xi_{2}} > 0,
      \\
      0 & \textnormal{otherwise.}
    \end{cases}
  \end{equation}
\end{lemma}
\begin{proof}
  First
  \[
    \begin{split}
      \sum_{\lbrace
      (\ell_{1}, \ell_{2}) \in \N^{2} \: : \:
      \xi_{1} \ell_{1} + \xi_{2} \ell_{2} > L
      \rbrace}
      2^{-\eta_{1}\ell_{1} - \eta_{2}\ell_{2}}
      & \lesssim \sum_{\ell_{1}=\lfloor L/\xi_{1} \rfloor}^{\infty}
        2^{-\eta_{1} \ell_{1}} \sum_{\ell_{2}= \lfloor L/\xi_{2} - \ell_{1} \xi_{1}/\xi_{2} \rfloor }^{\infty}
        2^{-\eta_{2} \ell_{2}} \\
      & \eqsim
        2^{-\eta_{2}L / \xi_{2}}
        \sum_{\ell_{1}=\lfloor L/\xi_{1} \rfloor}^{\infty}
        2^{( \eta_{2} \xi_{1}/\xi_{2} -\eta_{1}) \ell_{1}}
        \sum_{\ell_{2}=0}^{\infty}
        2^{-\eta_{2} \ell_{2}} \\
      & \lesssim
        2^{-\eta_{2}L / \xi_{2}},
    \end{split}
  \]
  assuming that \(\eta_{2}/\xi_{2} < \eta_{1}/\xi_{1}\). Otherwise, we obtain a similar
  bound with a symmetric calculation, so that, when \(\xi_{2}/\eta_{2} \neq
  \xi_{1}/\eta_{1}\),
  \[
    \sum_{\lbrace (\ell_{1}, \ell_{2}) \in \N^{2} \: : \: \xi_{1} \ell_{1} + \xi_{2} \ell_{2} > L \rbrace} \leq
    \widetilde C \, 2^{-\min(\eta_{1}/ \xi_{1}, \eta_{2}/ \xi_{2}) L}.
  \]
  In general, let \(\theta = \min(\eta_{1}/ \xi_{1}, \eta_{2}/ \xi_{2})\), and note that \(\lbrace
  \xi_{1} \ell_{1} + \xi_{2} \ell_{2} > L \rbrace \subseteq \lbrace \eta_{1} \ell_{1} + \eta_{2} \ell_{2} > \theta L \rbrace \) and
  define \(\cJ_{k} := \lbrace \boldEll \in \N_{0} \: : \: k< \eta_{1} \ell_{1} + \eta_{2} \ell_{2}
  \leq k+1 \rbrace\). A direct calculation yields that
    \[
      \lvert {\cJ_{k}} \rvert \leq
      (k+1) \left(\frac{1}{\eta_{2}} + \frac{1}{\eta_{1}\eta_{2}}\right)
          + 1 + \frac{1}{\eta_{2}} \lesssim k.
    \]
    Then we prove \eqref{eq:sum-exp-compliment} as follows:
  \[
    \begin{split}
      \sum_{\substack{\boldEll \in \N_{0}^{2}\\
      \eta_{1} \ell_{1} + \eta_{2} \ell_{2} > \theta L}} 2^{-\eta_{1} \ell_{1} -\eta_{2} \ell_{2}}
      &\leq \sum_{k=\lfloor  \theta L \rfloor}^{\infty} \sum_{\boldEll \in \cJ_{k}} 2^{-\eta_{1} \ell_{1} -\eta_{2} \ell_{2}} \\
      &\leq \sum_{k=\lfloor\theta L \rfloor}^{\infty} \sum_{\boldEll \in \cJ_{k}} 2^{-k} \\
      &\lesssim \sum_{k=\lfloor\theta L \rfloor}^{\infty} k 2^{-k} \lesssim \theta L 2^{-\theta  L}. \\
    \end{split}
  \]
  To prove \eqref{eq:sum-exp}, note that, for \(\eta_{1}\eta_{2}=0\),
  \[
    \begin{split}
      \sum_{\boldEll \in \mathcal I_{L}} 2^{\eta_{1} \ell_{1} + \eta_{2} \ell_{2}}
      &\leq \sum_{\ell_2=0}^{\lceil L/\xi_2 \rceil}
        \sum_{\ell_{1}=0}^{\lceil(L - \xi_{2} \ell_{2})/\xi_{1}\rceil} 2^{\eta_{1} \ell_{1}} 2^{\eta_{2} \ell_{2}} \\
      &\lesssim
        \begin{cases}
          (L+1)^{2} & \eta_{1} = \eta_{2} = 0, \\
          (L+1) & \min(\eta_{1}, \eta_{2}) < 0, \\
          2^{L \max(\eta_{1}/\xi_{1}, \eta_{2}/\xi_{2})} & \text{otherwise}. \\
        \end{cases}
    \end{split}
  \]
  When \(\eta_{1} < 0\) or \(\eta_{2} <0\) and \(\eta_{1}\eta_{2} \neq 0\) then
  \begin{equation}
    \begin{split}
      \sum_{\boldEll \in \mathcal I_{L}} 2^{\eta_{1} \ell_{1} + \eta_{2} \ell_{2}}
      &\leq \sum_{\ell_2=0}^{\lceil L/\xi_2 \rceil}
        \sum_{\ell_{1}=0}^{\lceil(L - \xi_{2} \ell_{2})/\xi_{1}\rceil} 2^{\eta_{1} \ell_{1}} 2^{\eta_{2} \ell_{2}} \\
      &\lesssim \sum_{\ell_2=0}^{\lceil L/\xi_2 \rceil} 2^{\max(0, \eta_{1}/\xi_{1}) \, (L-\xi_2 \ell_2)+ \eta_{2} \ell_2} \\
      &\lesssim 2^{L \max(0, \eta_{1}/\xi_{1}, \eta_{2}/\xi_{2})}.
    \end{split}
  \end{equation}
  Moreover, when \(\eta_{1} > 0\) and \(\eta_{2} > 0\),
  \[
    \begin{split}
      \sum_{\boldEll \in \mathcal I_{L}} 2^{\eta_{1} \ell_{1} + \eta_{2} \ell_{2}}
      &\leq \sum_{\ell_2=0}^{\lceil L/\xi_2 \rceil}
        \sum_{\ell_{1}=0}^{\lceil(L - \xi_{2} \ell_{2})/\xi_{1}\rceil} 2^{\eta_{1} \ell_{1}} 2^{\eta_{2} \ell_{2}} \\
      &\lesssim \sum_{\ell_2=0}^{\lceil L/\xi_2 \rceil} 2^{\max(0, \eta_{1}/\xi_{1}) \, (L-\xi_2 \ell_2)+ \eta_{2} \ell_2} \\
      &\lesssim
        \begin{cases}
          2^{L \max(0, \eta_{1}/\xi_{1}, \eta_{2}/\xi_{2})} & \eta_{1}/\xi_{1} \neq \eta_{2}/\xi_{2},\\
          2^{L \max(0, \eta_{1}/\xi_{1}, \eta_{2}/\xi_{2})} (L+1) & \eta_{1}/\xi_{1} = \eta_{2}/\xi_{2}.
        \end{cases}
    \end{split}
  \]
  Combining the different cases proves \eqref{eq:sum-exp}.
\end{proof}

\begin{lemma}\label{lem:general-mimc-cost-error}
  {Define the sequences
  $M_k \eqsim N_k \eqsim 2^k$ for \(k \in \N\), }and let \(\Psi(X^{M_{k}}_{N_{k}})\)
  denote a numerical approximation of \(\Psi(X)\) that it is strongly convergent
  and satisfies
  \begin{equation}\label{eq:weak-conv-assumption}
    \lim_{k \to \infty} \|\E[\Psi(X^{M_{k}}_{N_{k}}) - \Psi(X)] \|_{U} = 0,
  \end{equation}
  and, for \(\Delta_{\boldEll} \Psi(X)\) as defined in \eqref{eq:mimc-diff-def},
  \begin{equation}\label{eq:strong-error-coupling-assumption}
    \begin{split}
      \|\Delta_{\boldEll}\Psi(X)\|_{L^2(\Omega, U)}^2
      \le C\,  2^{-\beta_1 \ell_{1} -\beta_{2} \ell_{2}},
    \end{split}
  \end{equation}
  for some $C, \beta_1, \beta_2 > 0$ and all $\boldEll = (\ell_{1}, \ell_{2}) \in \N_0^2$.
  Then there exists \(\alpha_{i} \geq \beta_{i}/2\) for \(i=1,2\), such that
  \begin{equation}\label{eq:weak-conv-mimc-general}
    \|\E[\Delta_{\boldEll}\Psi(X)]\|_{U} \leq C \, 2^{-\alpha_{1} \ell_{1} - \alpha_{2} \ell_{2}}.
  \end{equation}
  Morever, there exist an index set $\cI\subset \N_0^2$, depending on \(\varepsilon, \beta_{1}\)
  and \(\beta_{2}\) that fulfills the telescoping
  constraint~\eqref{eq:telescoping-constraint}, and a sequence
  $(m_\boldEll)_{\boldEll \in \cI} \subset \N$, also depending on $\varepsilon$, such that an
  MIMC estimator \eqref{eq:mimcei} has a mean-square error (MSE) of
  \[
    \|\MIMCest - \E[\Psi(X)]\|_{L^2(\Omega, U)}^2 = \mathcal O(\varepsilon^2)
  \]
  and, given \(\cost(\Delta_{\boldEll} \Psi(X)) \simeq M_{\ell_{1}} N_{\ell_{2}}
  (\log(N_{\ell_{2}}))^{\mathfrak l} \simeq \ell_{2}^{\mathfrak{l}} 2^{\ell_{1}+\ell_{2}}\) for \({\mathfrak{l}} \in
  \lbrace0,1\rbrace\), it holds that \(\cost(\MIMCest) = \mathcal O\left(\varepsilon^{-2-2u} \lvert
    \log(\varepsilon^{-1}) \rvert^{r} \right)\) for
  \[
    2u = \max\left(0, \frac{1-\beta_{1}}{\alpha_{1}}, \frac{1-\beta_{2}}{\alpha_{2}} \right)
  \]
  and
  \[
    r =
      1_{\beta_{2} \leq 1} \, \mathfrak l + \begin{cases}
        4 & \beta_{1}=\beta_{2}=1 \\
        2 & \beta_{1}=1 < \beta_{2} \textnormal{ or } \beta_{2}=1 < \beta_{1} \\
        3 + 2u & \frac{1-\beta_{1}}{\alpha_{1}} = \frac{1-\beta_{2}}{\alpha_{2}} > 0 \\
        0 & \text{otherwise}. \\
    \end{cases}
    \]
\end{lemma}

\begin{proof}[Proof of Lemma~\ref{lem:general-mimc-cost-error}]
  The proof is similar to the proof in \cite[Theorem 2.2]{HNT16}, except for
  the additional log term in the cost. We include it here for completeness.
  We begin by splitting the error into a statistical error and a bias:
  \begin{equation}\label{eq:bias-stat-error-mimc}
    \|\MIMCest - \E[\Psi(X)]\|_{L^2(\Omega, U)}^2 =  \|\MIMCest - \E[\MIMCest]\|_{L^2(\Omega, U)}^2 + \|\E[\MIMCest - \Psi(X)]\|_{U}^2.
  \end{equation}
  For the first term, the statistical error, using the independence of samples,
  \begin{equation}\label{eq:var-bound-mimc-general}
    \begin{split}
      \|\MIMCest - \E[\MIMCest]\|_{L^2(\Omega, U)}^2
      &= \left\|\sum_{\boldEll \in \cI} \frac{1}{m_{\boldEll}}\sum_{i=1}^{m_{\boldEll}}(\Delta_{\boldEll} \Psi(X)^{(\boldEll, i)} - \E[\Delta_{\boldEll} \Psi(X)])\right\|_{L^2(\Omega, U)}^2 \\
      &= \sum_{\boldEll \in \cI}  \frac{1}{m_{\boldEll}}\left\|\Delta_{\boldEll} \Psi(X) - \E[\Delta_{\boldEll} \Psi(X)]\right\|_{L^2(\Omega, U)}^2 \\
      &\le \sum_{\boldEll \in \cI} \frac{\|\Delta_{\boldEll} \Psi(X) \|_{L^2(\Omega, U)}^2}{m_\boldEll}.
    \end{split}
  \end{equation}
  Setting $C_{\boldEll} := 2^{\ell_{1}+\ell_2} \ell_{2}^{\mathfrak{l}}$ to be the computational
  cost, up to a constant, of computing $\Delta_{\boldEll} \Psi(X)$, {defining
    \(\widetilde C_{\boldEll} :=2^{\ell_{1}+\ell_2} \leq C_{\boldEll}\) for every
    \(\boldEll \in \N^{2}\)}, and setting \(\widehat V_{\boldEll} := 2^{-\beta_{1} \ell_{1} -
    \beta_{2}\ell_{2}}\) to be the upper bound in
  \eqref{eq:strong-error-coupling-assumption}, again up to a constant, we set
  \begin{equation}\label{eq:ml-general}
    m_{\boldEll} = \left \lceil %
      \varepsilon^{-2} \sqrt{\frac{\widehat V_{\boldEll}}{\widetilde C_\boldEll}}
      \sum_{{\mathbf{k}} \in \cI} \sqrt{\widehat V_{{\mathbf{k}}} \widetilde C_{{\mathbf{k}}}} \right\rceil.
  \end{equation}
  This is in alignment with MLMC and MIMC theory, cf.~\cite{G08,HNT16}. We
  proceed to bound the statistical error
  \[
    \begin{split}
      \|\MIMCest - \E[\MIMCest]\|_{L^2(\Omega, U)}^2& \stackrel{\eqref{eq:var-bound-mimc-general}}{\leq}
                                               \varepsilon^2 \sum_{\boldEll \in \cI}
                                               \left(\sqrt{\frac{\widehat V_{\boldEll}}{\widetilde C_\boldEll}}
                                               \sum_{{\mathbf{k}} \in \cI} \sqrt{\widehat V_{{\mathbf{k}}} \widetilde C_{{\mathbf{k}}}}\right)^{-1}
                                               \|\Delta_{\boldEll} \Psi(X) \|^{2}_{L^2(\Omega, U)} \\
                                             &\stackrel{\eqref{eq:strong-error-coupling-assumption}}{<}
                                               C \varepsilon^2 \sum_{\boldEll \in \cI}
                                               \left(\sqrt{\frac{\widehat V_{\boldEll}}{\widetilde C_\boldEll}}
                                               \sum_{{\mathbf{k}} \in \cI} \sqrt{\widehat V_{{\mathbf{k}}} \widetilde C_{{\mathbf{k}}}}\right)^{-1}
                                               \widehat V_{\boldEll} \\
                                             &=  \mathcal O(\varepsilon^2).
    \end{split}
  \]

  For the second term in \eqref{eq:bias-stat-error-mimc}, i.e., the bias,
  first note that by~\eqref{eq:strong-error-coupling-assumption} and Jensen's
  inequality, it holds that
  \[
    \begin{split}
      \|\E[ \Delta_{\boldEll}\Psi(X) ]\|_{U} \le
      \E[\|  \Delta_{\boldEll}\Psi(X) \|_{U}]
      \lesssim 2^{-\beta_1 \ell_{1}/2} 2^{-\beta_2 \ell_{2}/2} \,,
    \end{split}
  \]
  which shows that \eqref{eq:weak-conv-mimc-general} is satisfied for some
  $\alpha_1 \ge \beta_1/2$, $\alpha_2 \ge \beta_2/2$.
  By linearity of the expectation operator on $U$, we obtain that for any
  $k_1,k_2 \in \N_0$,
  \[
    \E[\Psi(X_{N_{k_2}}^{M_{k_1}})] = \sum_{\ell_1=0}^{k_1} \sum_{\ell_2=0}^{k_2}
    \E[\Delta_{\boldEll} \Psi(X) ].
  \]
  We then follow \cite[(31)--(33)]{HNT16} and set
  \begin{equation}\label{eq:alternative-index-set-shape}
    \cI := \cI_{L} = \{\boldEll \in \N_0^2 \mid \xi_1 \ell_1 + \xi_2 \ell_2 \le L \},
  \end{equation}
  for \(\xi_j > 0\) which will be specified later depending on the values of \(\lbrace
  \beta_{i}, \alpha_{i} \rbrace_{i=1,2}\). Clearly \(\cI_{L}\) satisfies the telescoping sum
  constraint~\eqref{eq:telescoping-constraint}, so that, given
  \eqref{eq:weak-conv-assumption},
  \begin{equation}\label{eq:exp-bound-mimc}
    \begin{split}
      \|\E[\MIMCest - \Psi(X)]\|_{U} &\leq \lim_{k \to \infty} \|\E[\MIMCest - \Psi(X_{N_{k}}^{M_{k}})]\|_{U}\\
                                     &= \left\| \sum_{\boldEll \in \N_0^2 \setminus \cI_{L} } \E[\Delta_\boldEll \Psi(X)] \right\|_{U}\\
                                &\le \sum_{\boldEll \in \N_0^2 \setminus \cI_{L} } \| \E[\Delta_\boldEll \Psi(X)]\|_{ U} \\
                                &\le 2^{-\alpha_{1} \ell_{1}/2} 2^{-\alpha_{2} \ell_{2}/2} \,.
    \end{split}
  \end{equation}
  Here, the limit in the second equality holds thanks to the %
  summability of $(\|\E[\Delta_\boldEll \Psi(X)]\|_U)_{\boldEll \in \N_0^2}$ implied
  by~\eqref{eq:weak-conv-mimc-general}. We bound \eqref{eq:exp-bound-mimc}
  using \eqref{eq:sum-exp-compliment}, and then conclude that setting
  \begin{equation}\label{eq:mimc-L-choice}
    L = \begin{cases}
      \max(\lceil
      \theta(\log_2(2 \varepsilon^{-1}) +
      \log_{2}({\theta}\log_2( 2 \varepsilon^{-1}))
      )\rceil , 1) & \xi_{1}/\alpha_{1} = \xi_{2}/\alpha_{2},\\
      \max(\lceil
      {\theta} ({\log_2(2 \varepsilon^{-1})
      })\rceil , 1) & \textnormal{otherwise}, \\
    \end{cases}
  \end{equation}
  for \(\theta = \max(\xi_{1}/\alpha_{1},\xi_{2}/\alpha_{2})\), ensures that \(\|\E[\MIMCest -
  \Psi(X)]\|_{U} %
  = \mathcal O(\varepsilon)\). Substituting the bounds on the bias and the statistical
  error proves the bound on the mean square error.
  
  We now bound the
  computational cost:
  \begin{equation}\label{eq:mimc-cost-general}
    \begin{split}
      \cost(\MIMCest) &:= \sum_{\boldEll \in \cI_{L}} m_{\boldEll}\, \cost(\Delta_{\boldEll}\Psi(X))\\
                      & \lesssim
                        \varepsilon^{-2} \sum_{\boldEll \in \cI_{L}}
                        \left(\sqrt{\frac{\widehat V_{\boldEll}}{\widetilde C_\boldEll}}
                        \sum_{{\mathbf{k}} \in \cI_{L}} \sqrt{\widehat V_{{\mathbf{k}}} \widetilde C_{{\mathbf{k}}}}\right)
                        C_{\boldEll}
                        + \sum_{\boldEll \in \cI_{L}} C_{\boldEll}\\
                      & = \varepsilon^{-2}
                        {\left(\sum_{\boldEll \in \cI_{L}} \sqrt{\widehat V_{\boldEll} \widetilde C_\boldEll}
                        \right)
                        \left(\sum_{\boldEll \in \cI_{L}} C_{\boldEll}\sqrt{\widehat V_{\boldEll} / \widetilde C_\boldEll}\right)}
                        + \sum_{\boldEll \in \cI_{L}} C_{\boldEll}. %
    \end{split}
  \end{equation}
  We start by bounding the first term in last line by considering the
  different cases of \(\beta_{2}\) and recalling that
  \begin{equation}\label{eq:mimc-cost-factors}
    \begin{aligned}
      \sqrt{\widehat V_{\boldEll} \widetilde C_\boldEll}
      &\eqsim
        2^{(1-\beta_{1})\ell_{1}/2 + (1-\beta_{2})\ell_{2}/2} \\
      \textnormal{and}\qquad
      C_\boldEll \sqrt{\widehat V_{\boldEll} / \widetilde C_\boldEll} &\eqsim 2^{(1-\beta_{1})\ell_{1}/2 + (1-\beta_{2})\ell_{2}/2} \ell_{2}^{\mathfrak l}
    \end{aligned}
  \end{equation}
  We first make the choice \(\xi_{i} = \alpha_{i} + (1-\beta_{i})/c > 0\), for some \(c \geq
  2\).
  \begin{description}
  \item[The case $\beta_2 \leq 1$] Using \eqref{eq:sum-exp}, for \(\eta_{i}=1-\beta_{i}\),
    and substituting \eqref{eq:mimc-L-choice}, we arrive at
    \[
      \begin{aligned}
        \sum_{\boldEll \in \cI_{L}} C_{\boldEll} \sqrt{\widehat V_{\boldEll}/ \widetilde C_\boldEll}
        &\lesssim (L+1)^{\mathfrak l} \sum_{\boldEll \in \cI_{L}} 2^{(1-\beta_{1})\ell_{1}/2 + (1-\beta_{2}) \ell_{2}/2} \\
        & = \mathcal O(2^{\max\left(0, \frac{1-\beta_{1}}{\xi_{1}},
          \frac{1-\beta_{2}}{\xi_{2}}\right)L/2} L^{\mathfrak l + v}) \\
        & = \mathcal O(\varepsilon^{-u} \log_{2}(\varepsilon^{-1})^{\mathfrak l + v + \tilde v}),
      \end{aligned}
    \]
    where, substituting the value of \(\xi_{i}\),
    \begin{equation}\label{eq:rates-calc}
      \begin{aligned}
        u &=
        \frac{1}{2}\max\left( \frac{\xi_{1}}{\alpha_{1}}, \frac{\xi_{2}}{\alpha_{2}} \right) \,
        \max\left(\frac{1-\beta_{1}}{\xi_{1}}, \frac{1-\beta_{2}}{\xi_{2}} \right) \\
        &= \frac{1}{2} \left(\frac{1 +
          \max\left( \frac{1-\beta_{1}}{\alpha_{1}}, \frac{1-\beta_{2}}{\alpha_{2}} \right)/c}
          {(1/c) +\min\left(\frac{\alpha_{1}}{1-\beta_{1}}, \frac{\alpha_{2}}{1-\beta_{2}}
          \right)}\right) \\
        &= {\max\left( \frac{1-\beta_{1}}{2\alpha_{1}}, \frac{1-\beta_{2}}{2\alpha_{2}}
          \right)}.
      \end{aligned}
    \end{equation}
    Moreover, we have the log powers, cf. \eqref{eq:sum-exp-log-powers},
    \[
      v =
      \begin{cases}
        2 & \beta_{1}=\beta_{2}=1, \\
        1 & \beta_{1}=1 < \beta_{2} \textnormal{ or }
            \beta_{2}=1 < \beta_{1}
            \textnormal{ or }
            \frac{1-\beta_{1}}{\alpha_{1}} = \frac{1-\beta_{2}}{\alpha_{2}} > 0,
        \\
        0 & \textnormal{otherwise},
      \end{cases}
    \]
    and
    \[
      \begin{aligned}
        \tilde v %
                 &= \begin{cases}
                   u & \frac{1-\beta_{1}}{\alpha_{1}} = \frac{1-\beta_{2}}{\alpha_{2}},  \\
                   0 & \textnormal{otherwise},
                 \end{cases}
      \end{aligned}
    \]
    corresponding to the additional log-term in \eqref{eq:mimc-L-choice}.
    Hence,
    \[
      \begin{aligned}
        \varepsilon^{-2}
        {\left(\sum_{\boldEll \in \cI_{L}} \sqrt{\widehat V_{\boldEll} \widetilde C_\boldEll}
        \right)
        \left(\sum_{\boldEll \in \cI_{L}} C_{\boldEll}\sqrt{\widehat V_{\boldEll} / \widetilde C_\boldEll}\right)}
        & = \mathcal O(\varepsilon^{-2-2u} \lvert  \log_{2}(\varepsilon^{-1}) \rvert^{\mathfrak l + 2v + 2\tilde v}).
      \end{aligned}
    \]
  \item[The case $\beta_2 > 1$] Then, for any \(0 < \delta < \beta_{2}-1\),
    \[
      \begin{aligned}
        \sum_{\boldEll \in \cI_{L}} C_{\boldEll} \sqrt{\widehat V_{\boldEll}/ \widetilde C_\boldEll}
        &\lesssim \sum_{\boldEll \in \cI_{L}} 2^{(1-\beta_{1})\ell_{1}/2 + (1-\beta_{2})\ell_{2}/2} \ell_{2}^{\mathfrak l} \\
        & \lesssim \sum_{\boldEll \in \cI_{L}} 2^{(1-\beta_{1})\ell_{1}/2 + (1-\beta_{2}+\delta)\ell_{2}/2} \\
        & = \mathcal O(\varepsilon^{-u} \log_{2}(\varepsilon^{-1})^{v + \tilde v}),
      \end{aligned}
    \]
    where, since \(\beta_{2}-1 < \beta_{1}-\delta-1 < 0\), we again have
    \[
      \begin{aligned}
        u &= \max\left(0, \max\left( \frac{\xi_{1}}{\alpha_{1}}, \frac{\xi_{2}}{\alpha_{2}} \right) \,
                \max\left(\frac{1-\beta_{1}}{2\xi_{1}}, \frac{1-\beta_{2}+\delta}{2\xi_{2}} \right)\right) \\
              &= \max\left(0, \frac{1-\beta_{1}}{2 \alpha_{1}}, \frac{1-\beta_{2}}{2 \alpha_{2}} \right),
      \end{aligned}
    \]
    and, in this case, since \(\beta_{2}>1\),
    \[
      v =
      \begin{cases}
        1 & \beta_{1}=1, \\
        0 & \textnormal{otherwise}.
      \end{cases}
    \]
    Hence,
    \[
      \begin{aligned}
        \varepsilon^{-2}
        {\left(\sum_{\boldEll \in \cI_{L}} \sqrt{\widehat V_{\boldEll} \widetilde C_\boldEll}
        \right)
        \left(\sum_{\boldEll \in \cI_{L}} C_{\boldEll}\sqrt{\widehat V_{\boldEll} / \widetilde C_\boldEll}\right)}
        & = \mathcal O(\varepsilon^{-2-2u} \lvert  \log_{2}(\varepsilon^{-1}) \rvert^{2 v + 2\tilde v}).
      \end{aligned}
    \]
  \end{description}
  Finally, we consider the second term in \eqref{eq:mimc-cost-general} and to
  check which of the two terms dominates the other, for sufficiently small
  \(\varepsilon\). Using \eqref{eq:sum-exp} and substituting \eqref{eq:mimc-L-choice},
  yields
  \begin{equation}\label{eq:mimc-min-cost}
    \begin{split}
      \sum_{\boldEll \in \cI_{L}} C_{\boldEll}
      &\lesssim L^{\mathfrak l} \sum_{\boldEll \in \cI_{L}} 2^{\ell_{1}+\ell_{2}} \\
      &\lesssim L^{\mathfrak l + v_{1}}  2^{L / \min(\xi_{1},\xi_{2})} \\
      &\lesssim \varepsilon^{-{\max(\xi_{1}/\alpha_{1},\xi_{2}/\alpha_{2}) / \min(\xi_{1},\xi_{2})}} \log_{2}(\varepsilon^{-1})^{\mathfrak l + v_{1}+\tilde v_{1}} \\
      &\lesssim \varepsilon^{-u_{1}} \log_{2}(\varepsilon^{-1})^{\mathfrak l + v_{1}+\tilde v_{1}},
    \end{split}
  \end{equation}
  for \(u_{1} := {\max(\xi_{1}/\alpha_{1},\xi_{2}/\alpha_{2}) / \min(\xi_{1},\xi_{2})}\),
  \[
    v_{1} =
    \begin{cases}
      1 & \alpha_{1} - \frac{\beta_{1}}{c} = \alpha_{2} - \frac{\beta_{2}}{c} \\
      0 & \textnormal{otherwise},
    \end{cases}
    \]
    and
    \[
    \begin{aligned}
      \tilde v_{1} &=
                     \begin{cases}
                       u_{1} & \frac{1-\beta_{1}}{\alpha_{1}}=
                                       \frac{1-\beta_{2}}{\alpha_{2}} \\
                       0 & \textnormal{otherwise}.
                     \end{cases}
    \end{aligned}
    \]
  For \(c \geq 2\) and recalling that \(\alpha_{i} \geq \beta_{i}/2 \), we have
  \begin{equation}\label{eq:rate-1-bound}
    \begin{aligned}
      u_{1} &:=  \max\left( \frac{\xi_{1}}{\alpha_{1}}, \frac{\xi_{2}}{\alpha_{2}} \right) \Big /
              \min\left(\xi_{1}, \xi_{2} \right)\\%
            &=  \frac{c + \max((1-\beta_{1})/\alpha_{1}, (1-\beta_{2})/\alpha_{2})  }
              { 1 + \min\left(c \, \alpha_{1}-\beta_{1},
              c \, \alpha_{2}-\beta_{2} \right)}\\
            &\leq \frac{c}{1 + (c/2-1)\min\left(\beta_{1}, \beta_{2} \right)} + \max\left(\frac{1-\beta_{1}}{\alpha_{1}}, \frac{1-\beta_{2}}{\alpha_{2}}\right).
      \end{aligned}
    \end{equation}
    When \(\min\left(\beta_{1}, \beta_{2} \right) > 1\), we bound the first term by
    \(2\) since \(c \geq 2\), otherwise we set \(c=2\) to obtain the same bound.
    We now distinguish between four cases:
  \begin{description}
  \item [When \(\min(\beta_{1}, \beta_{2}) > 1\)] Then, \(u_{1} < 2 + 2 u\)
    and hence
    \[
      \sum_{\boldEll \in \cI_{L}} C_{\boldEll} = \cO( \varepsilon^{-2 - 2 u}).
    \]
  \item [When \(\beta_{1}= \beta_{2} = 1\)] In this case, \(u_{1} \leq 2 + 2 u = 2\),
    \(v_{1} + \tilde v_{1} \leq u_{1} + 1 \leq 3 < 4 = 2v + 2\tilde v\).
  \item [When \(1 \in \lbrace \beta_{1}, \beta_{2} \rbrace\) and \(\beta_{1} \neq \beta_{2}\)] In this case,
    \(u_{1} \leq 2 + 2 u\), \(v_{1} + \tilde v_{1} = 1 < 2 = 2v + 2\tilde v \).

  \item [When \(\max(\beta_{1}, \beta_{2}) < 1\)] %
    In this case, either \(v_{1} = 1\) and \(\tilde v_{1} = 0\), or \(v_{1} =
    0\) and \(\tilde v_{1} = u_{1}\) or \(v_{1} = 1\) and \(\tilde v_{1} =
    u_{1} = 1/\alpha_{1}\) (since \(\alpha_{1}=\alpha_{2}\) and \(\beta_{1}=\beta_{2}\) in this last case).
    In all cases, \(v_{1} + \tilde v_{1} \leq 2u + 3\) and \(v + \tilde v \leq 2u + 2\).

  \end{description}
\end{proof}

The following lemma obtains slightly better complexity rates when there is a
stricter bound that improves the rate multiplying \(\ell_{1}\) by \(\vartheta > 0\) in
\eqref{eq:strong-error-coupling-assumption}, but reduces the rate multiplying
\(\ell_{2}\) by \(\vartheta \upsilon > 0\). Looking at Lemma~\ref{lem:general-mimc-cost-error},
we see that the only case a smaller rate multiplying \(\ell_{1}\) could improve
the computational complexity (beyond a logarithmic factor) is when \(\beta_{1}<1\)
and \((1-\beta_{2})/\alpha_{2} < (1-\beta_{1})/\alpha_{1}\).

\begin{lemma}\label{lem:general-mimc-cost-error-min}
  Given the same setup as Lemma~\ref{lem:general-mimc-cost-error}, and
  assuming the stricter bound on the variance
  \begin{equation}\label{eq:mimc-max-rates-var}
    \|\Delta_{\boldEll}\Psi(X)\|_{L^2(\Omega, U)}^2
    \le C\,  2^{-\beta_1 \ell_{1} -\beta_{2} \ell_{2} - \vartheta \max( \ell_{1} -\upsilon \ell_{2}, 0)},
  \end{equation}
  for some $\vartheta, \upsilon, C, \beta_1, \beta_2 > 0$ and all $\boldEll = (\ell_{1}, \ell_{2}) \in
  \N_0^2$. Then there exists \(\alpha_{i} \geq \beta_{i}/2\) for \(i=1,2\), such that
  \begin{equation}\label{eq:mimc-max-rates-bias}
    \|\E[\Delta_{\boldEll}\Psi(X)]\|_{U}
    \le C\,  2^{-\alpha_1 \ell_{1} -\alpha_{2} \ell_{2} - (\vartheta/2) \max( \ell_{1} -\upsilon \ell_{2}, 0)}.
  \end{equation}
  Moreover, there exists an index set $\cI \subset \N_0^2$, depending on \(\varepsilon\) and
  \(\lbrace \alpha_{i}, \beta_{i}\rbrace_{i=1}^{2}\) that fulfills the telescoping
  constraint~\eqref{eq:telescoping-constraint} and a sequence
  $(m_\boldEll)_{\boldEll \in \cI} \subset \N$, also depending on $\varepsilon$, such that an
  MIMC estimator \eqref{eq:mimcei} has a mean-square error (MSE) of
  \[
    \|\MIMCest - \E[\Psi(X)]\|_{L^2(\Omega, U)}^2 = \mathcal O(\varepsilon^2)
  \]
  and, given that \(\cost(\Delta_{\boldEll} \Psi(X)) \simeq \ell_{2}^{\mathfrak{l}} 2^{\ell_{1}+\ell_{2}}\) for
  \({\mathfrak{l}} \in \lbrace0,1\rbrace\), then \[\cost(\MIMCest) = \cO(\varepsilon^{-2 - 2 u}\, |\log_2 \varepsilon|^{4
    +\mathfrak l+ 2 u})\] where
    \[
      2 u = \max\left(0, \frac{1-\beta_{1}-\vartheta}{\alpha_{1}+\vartheta/2}, \frac{1-\beta_{2}}{\alpha_{2}},
        \frac{(1-\beta_{2}) + \upsilon (1-\beta_{1})} {\alpha_{2} + \upsilon \alpha_{1}}\right).
    \]
\end{lemma}
\begin{proof}
  First we state a generic result. For \(\eta_{1}, \eta_{2} \in \R\) and \(\xi_{1},
  \xi_{2}, \vartheta, \upsilon, L > 0\), we have the bound
  \[ %
    \max_{\ell \in \N^{2}, \xi_{1} \ell_{1} + \xi_{2} \ell_{2} < L}({\eta_{1} \ell_{1} +\eta_{2} \ell_{2}
      - \vartheta \max( \ell_{1} -\upsilon \ell_{2}, 0)}) \leq L \max\left(0, \frac{\eta_{1}-\vartheta}{\xi_{1}},
      \frac{\eta_{2}}{\xi_{2}}, \frac{\eta_{2} + \upsilon \eta_{1}}{\xi_{2}+\upsilon\xi_{1}}
    \right). %
  \]
  This can be obtained by noting that outside of the boundaries \(\ell_{1} = 0\),
  \(\ell_{2} = 0\), \(\xi_{1} \ell_{1} + \xi_{2} \ell_{2} = L\) and \(\ell_{1} - \upsilon \ell_{2} =
  0\), one is maximizing a linear function. Hence the maximum is achieved at
  one of the four points at which these lines intersect\footnote{The proof is
    based on an argument by M.B. Giles in a soon-to-published work.}. Hence,
  \begin{equation}\label{eq:sum-bound-with-max}
    \begin{aligned}
      \sum_{\substack{\ell \in \N^{2} \\
      \xi_{1} \ell_{1} + \xi_{2} \ell_{2} \leq L}} \left(2^{\eta_{1} \ell_{1} +\eta_{2} \ell_{2} - \vartheta \max(
      \ell_{1} -\upsilon \ell_{2}, 0)}\right)^{1/2}
      &\lesssim L^{2} \, 2^{
        (L/2) \max\left(0, \frac{\eta_{1}-\vartheta}{\xi_{1}},
        \frac{\eta_{2}}{\xi_{2}}, \frac{\eta_{2} + \upsilon \eta_{1}}{\xi_{2}+\upsilon\xi_{1}}
        \right)}.
    \end{aligned}
  \end{equation}
  Using the same reasoning, we can bound for \(\eta_{1}, \eta_{2} > 0\),
  \begin{equation}\label{eq:sum-bound-inverted-with-max}
    \begin{aligned}
      \sum_{\substack{\ell \in \N^{2} \\
      \xi_{1} \ell_{1} + \xi_{2} \ell_{2} > L}} \left(2^{-\eta_{1} \ell_{1} -\eta_{2} \ell_{2} - \vartheta \max(
      \ell_{1} -\upsilon \ell_{2}, 0)}\right)^{1/2}
      &\lesssim L \, 2^{
        -(L/2) \min\left(0, \frac{\eta_{1}+\vartheta}{\xi_{1}},
        \frac{\eta_{2}}{\xi_{2}}, \frac{\eta_{2} + \upsilon \eta_{1}}{\xi_{2}+\upsilon\xi_{1}}
        \right)}.
    \end{aligned}
  \end{equation}

  We now turn to proving the claims of the lemma. First,
  \eqref{eq:mimc-max-rates-bias} follows from \eqref{eq:mimc-max-rates-var}
  using Jensen's inequality. We then make the same choice of \(\mathcal I \equiv
  \mathcal I_{L}\) in \eqref{eq:alternative-index-set-shape}, and set
  \begin{equation}\label{eq:mimc-L-choice-max-rates}
    L = \max(\lceil
    \theta(\log_2(2 \varepsilon^{-1}) +
    \log_{2}({\theta}\log_2( 2 \varepsilon^{-1}))
    )\rceil , 1),
  \end{equation}
  for
  \[
    \theta = \max\left(\frac{\xi_{1}}{\alpha_{1}+\vartheta/2}, \frac{\xi_{2}}{\alpha_{2}},
      \frac{\xi_{2}+\upsilon\xi_{1}}{\alpha_{2} + \upsilon \alpha_{1}} \right),
  \]
  and \(\lbrace m_{\boldEll} \rbrace_{\boldEll \in \mathcal I_{L}}\) as in
  \eqref{eq:ml-general}, but for \(\widehat V_{\boldEll} = 2^{-\beta_1 \ell_{1}
    -\beta_{2} \ell_{2} - \vartheta \max( \ell_{1} -\upsilon \ell_{2}, 0)}\), so that a similar
  calculation to the proof of Lemma~\ref{lem:general-mimc-cost-error}, using
  \eqref{eq:sum-bound-inverted-with-max}, shows that the MSE is of
  \(\cO(\varepsilon^{2})\). To bound the computational cost, we set \(\eta_{i} :=
  (1-\beta_{i})\), and obtain
  \[
    \sum_{\boldEll \in \cI_{L}} C_{\boldEll} \sqrt{\widehat V_{\boldEll} /
      \widetilde C_\boldEll} \lesssim L^{\mathfrak l} \sum_{\boldEll \in \cI_{L}}
    \left(2^{\eta_{1} \ell_{1} + \eta_{2} \ell_{2} - \vartheta \max( \ell_{1} -\upsilon \ell_{2},
        0)}\right)^{1/2}
  \]
  and a similar bound on \(\sum_{\boldEll \in \cI_{L}} \sqrt{\widehat V_{\boldEll}
    \widetilde C_\boldEll}\). Using \eqref{eq:sum-bound-with-max}, we obtain
  \[
    \left( \sum_{\boldEll \in \cI_{L}} \sqrt{\widehat V_{\boldEll} \widetilde
        C_\boldEll} \right) \left( \sum_{\boldEll \in \cI_{L}} C_{\boldEll}
      \sqrt{\widehat V_{\boldEll} / \widetilde C_\boldEll} \right)
    \lesssim L^{4+\mathfrak l}
    \, 2^{L \max\left(0, \frac{\eta_{1}-\vartheta}{\xi_{1}},
      \frac{\eta_{2}}{\xi_{2}}, \frac{\eta_{2} + \upsilon \eta_{1}}{\xi_{2}+\upsilon\xi_{1}}
        \right)}.
  \]
  For \(L\) as in \eqref{eq:mimc-L-choice-max-rates}, we note that for any
  constant \(a \geq 0\)
  \[
    L^{4+\mathfrak l} \, 2^{a L} \lesssim \lvert \log(\varepsilon) \rvert^{4+\mathfrak l} \:
    \left(\lvert \log(\varepsilon) \rvert \: \varepsilon^{-1}\right)^{\theta a}.
  \]
  We now make the choice \(\xi_{i} = \alpha_{i} + \eta_{i}/2 = \alpha_{i} + (1-\beta_{2})/2\).
  Then, we obtain
  \begin{equation}\label{eq:mimc-max-rates-sampling-cost}
    \left( \sum_{\boldEll \in \cI_{L}} \sqrt{\widehat V_{\boldEll} \widetilde
        C_\boldEll} \right) \left( \sum_{\boldEll \in \cI_{L}} C_{\boldEll}
      \sqrt{\widehat V_{\boldEll} / \widetilde C_\boldEll} \right)
    \lesssim \lvert \log(\varepsilon) \rvert^{4+2u+\mathfrak l} \: \varepsilon^{- 2u},
  \end{equation}
  for, using a similar calculation to \eqref{eq:rates-calc},
  \[
    \begin{aligned}
      2u &:=%
           \max\left(0, \frac{\eta_{1}-\vartheta}{\alpha_{1}+\vartheta/2},
           \frac{\eta_{2}}{\alpha_{2}},
           \frac{\eta_{2}+\upsilon\eta_{1}}{\alpha_{2} + \upsilon \alpha_{1}} \right).
    \end{aligned}
  \]
  By a similar calculation to the proof in
  Lemma~\ref{lem:general-mimc-cost-error}, we bound \eqref{eq:mimc-min-cost}
  by
  \begin{equation}
    \sum_{\ell \in \mathcal I_{L}} C_{\boldEll}
    \lesssim \varepsilon^{-u_{1}} \log_{2}(\varepsilon^{-1})^{\mathfrak l + v_{1}+\tilde v_{1}},
  \end{equation}
  for
  \[
    \begin{aligned}
      u_{1} &:=  {\max(\xi_{1}/(\alpha_{1}+\vartheta/2),\xi_{2}/\alpha_{2}) / \min(\xi_{1},\xi_{2})}\\%
            &=  \frac{2 + \max\left(\frac{1-\beta_{1}-\vartheta}{\alpha_{1}+\vartheta/2}, \frac{1-\beta_{2}}{\alpha_{2}}\right)  }
              { 1 + \min\left(2 \, \alpha_{1}-\beta_{1},
              2 \, \alpha_{2}-\beta_{2} \right)}\\
            &\leq 2 + \max\left(\frac{1-\beta_{1}-\vartheta}{\alpha_{1}+\vartheta/2}, \frac{1-\beta_{2}}{\alpha_{2}}\right)
    \end{aligned}
  \]
  and \(v_{1} + \tilde v_{1} \leq 3 + 2 u\).
  Substituting these bounds in \eqref{eq:mimc-cost-general} yields claimed
  computational complexity.
\end{proof}

Next, we include a proof of Lemma~\ref{lem:strong-error-coupling}
\begin{proof}[Proof of Lemma~\ref{lem:strong-error-coupling}]
  We treat the two different QoI settings~\eqref{eq:QoI_def1} and~\eqref{eq:QoI_def2}
  separately.

  For $\Psi(X) = \psi(X(T))$, a similar expansion as in the proof of
  Theorem~\ref{thm:second-order-diff} yields
  	\begingroup
	\allowdisplaybreaks
	\begin{align*}
          &\Psi(X_K^{\overline{M}})  - \Psi(X_N^{\overline{M}})  - \Psi(X_K^M) + \Psi(X_N^M) \\
          & = \psi(X_K^{\overline{M}}(T))  - \psi(X_N^{\overline{M}}(T))  - \psi(X_K^M(T)) + \psi(X_N^M(T)) \\
          & = \int^1_0 U_T(\lambda)  \cE^{\overline{M},M}_{K,N}(T) \dd \lambda
          + \int^1_0 \int^1_0 \tilde U_T(\lambda, \tilde \lambda)
          \big(X^{\overline{M}}_N(T) - X^M_N(T) \big) \big(X^M_K(T) - X^M_N(T) \big) \dd \lambda \dd \tilde \lambda,
	\end{align*}%
	\endgroup
	where
	\begin{align*}
	  U_t(\lambda) = \psi'\Big( (1-\lambda)
          \big(X^{\overline{M}}_N(t) + X^M_K(t) - X^M_N(t)\big) + \lambda X^{\overline{M}}_K(t) \Big) \quad \text{for} \quad t \in [0,T]
	\end{align*}
	and
	\begin{align*}
		\tilde U_t(\lambda, \tilde \lambda) = \psi''\Big(X^M_N(t) &+ \lambda \big(X^{\overline{M}}_N(t)-X^M_N(t)\big) + \tilde \lambda \big(X^M_K(t)-X^M_N(t)\big) \Big) \quad \text{for} \quad t \in [0,T].
	\end{align*}
        By Assumption~\ref{ass:qoi} and H\"older's inequality,
        \[
        \begin{split}
        \|\Psi(&X_K^{\overline{M}})  - \Psi(X_N^{\overline{M}})  - \Psi(X_K^M) + \Psi(X_N^M)\|_{L^2(\Omega, U)}\\
        &  \le C \Big(\|\cE^{\overline{M},M}_{K,N}(T)\|_{L^2(\Omega, H)} +
        \|X^{\overline{M}}_N(T) - X^M_N(T)\|_{L^4(\Omega, H)} \|X^M_K(T) - X^M_N(T)\|_{L^4(\Omega, H)} \Big),
        \end{split}
        \]
        and inequality~\eqref{eq:strong-error-coupling} follows by Theorem~\ref{thm:second-order-diff} and
        Lemma~\ref{lem:appr}\ref{lem:appr:space-error} and~\ref{lem:appr:time-error}.

        The first inequality of~\eqref{eq:strong-error-coupling2} follows directly from Lemma~\ref{lem:appr}\ref{lem:appr:space-error} along with the Lipschitz regularity of \(\Psi\). For the second inequality, we may argue exactly as in the proof of the bounds of Lemma~\ref{lem:appr}\ref{lem:appr:time-error} to see that
          $\| X_N^M(t) - X_N^{\bar M}(t) \|_{L^{10}(\Omega,H)} \le C \min(\lambda_N^{\max(1-\kappa,0)} M^{-1/2}, M^{-\min(1,\kappa)/2})$ for a constant \(C\) that does not depend on \(N\) or \(M\). The claim therefore follows by setting \(C_N := C \lambda_N^{\max(1-\kappa,0)}\).

        For the setting $\Psi(X) = \int_0^T \psi(X(t)) \dd t$, Jensen's inequality yields
        \[
        \begin{split}
          &\|\Psi(X_K^{\overline{M}})  - \Psi(X_N^{\overline{M}})  - \Psi(X_K^M) + \Psi(X_N^M)\|_{L^2(\Omega, U)} \\     & \le
        \int_0^T \| \psi(X_K^{\overline{M}}(t))  - \psi(X_N^{\overline{M}}(t))  - \psi(X_K^M(t)) + \psi(X_N^M(t)) \|_{L^2(\Omega, U)} \dd t,
        \end{split}
        \]
        and an expansion argument as above implies there
        exists a constant $C< \infty$ such that for all $t \in [0,T]$,
        \begin{equation}\label{eq:doubleDiffPsiProofIneq1}
        \begin{split}
          \qquad\| \psi(X_K^{\overline{M}}(t))  - \psi(X_N^{\overline{M}}(t))  - &\psi(X_K^M(t)) + \psi(X_N^M(t)) \|_{L^2(\Omega, U)} \\
          &\le C  \Big(\|X_K^{\overline{M}}(t))  - X_N^{\overline{M}}(t)  - X_K^M(t) + X_N^M(t) \|_{L^2(\Omega, H)} \\
          & \qquad  +
          \|X^{\overline{M}}_N(t) - X^M_N(t)\|_{L^4(\Omega, H)} \|X^M_K(t) - X^M_N(t)\|_{L^4(\Omega, H)} \Big)\,.
          \end{split}
        \end{equation}
       From the representation
        \begingroup
        \allowdisplaybreaks
	\begin{align*}
          X_K^{\overline{M}}(t))  &- X_N^{\overline{M}}(t)  - X_K^M(t) + X_N^M(t) \\
          &=\int^{t}_0 P_K (I-P_N) S(t-s) \big( F(X^{\overline{M}}_K(\lfloor s \rfloor_{{\overline{M}}^{-1}})) - F(X^M_K(\lfloor s \rfloor_{M^{-1}})) \big) \dd s \\
	  &\quad+\int^{t}_0 P_N S(t-s) \big( F(X^{\overline{M}}_K(\lfloor s \rfloor_{M^{-1}})) - F(X^{\overline{M}}_N(\lfloor s \rfloor_{M^{-1}})) \\
	  &\hspace{10em}- F(X^M_K(\lfloor s \rfloor_{M^{-1}})) + F(X^M_N(\lfloor s \rfloor_{M^{-1}})) \big) \dd s \\
	  &\quad+\int^{t}_0 P_N S(t-s) \big( F(X^{\overline{M}}_K(\lfloor s \rfloor_{{\overline{M}}^{-1}})) - F(X^{\overline{M}}_N(\lfloor s \rfloor_{{\overline{M}}^{-1}})) \\
	  &\hspace{10em}- F(X^{\overline{M}}_K(\lfloor s \rfloor_{M^{-1}})) + F(X^{\overline{M}}_N(\lfloor s \rfloor_{M^{-1}})) \big) \dd s \\
	  &\quad + \int^{t}_0 S(t-s) G \cE^{{\overline{M}},M}_{K,N}(s) \dd W(s) + \int^{t}_0 S(t-s) G\cH^{{\overline{M}},M}_{K,N}(s) \dd W(s) \\
	  &=: \mathrm{I}^1_F(t ) + \mathrm{I}^2_F(t) + \mathrm{I}^3_F(t) + \mathrm{I}^1_W(t) + \mathrm{I}^2_W(t),
        \end{align*}
        \endgroup and the proof of Theorem~\ref{thm:second-order-diff}
        readily extends to the currently considered continuum setting
        of all $t \in [0,T]$, meaning there exists a constant
        $C<\infty$ such that
        \begin{equation}\label{eq:extensionSecOrdDiff}
        \begin{split}
        \max_{t \in [0,T]} \Big\| \sum_{i=1}^3 \mathrm{I}^i_F(t )  + \mathrm{I}^1_W(t) + \mathrm{I}^2_W(t) &\Big\|_{L^2(\Omega, H)}^2\\
        &\le
        \begin{cases}
          C \min(\lambda_{N}^{-2\kappa},  \, M^{-2\kappa} ) & \text{if} \quad \kappa \in (0,1/2),\\
          C \min(\lambda_{N}^{-(2 \kappa -1)} M^{-1}, \, \lambda_N^{-2\kappa}) & \text{if} \quad \kappa \in [1/2,1),\\
            C \min(\lambda_{N}^{-\kappa} M^{-1}, \lambda_N^{-2\kappa}) & \text{if} \quad \kappa \in [1,2)\,.
        \end{cases}
        \\
        &\le C \lambda_{N}^{-\kappa}  M^{-\min(\kappa,1)}.
        \end{split}
        \end{equation}
        And Lemma~\ref{lem:appr}\ref{lem:appr:space-error} and~\ref{lem:appr:time-error}
        imply that
        \[
        \begin{split}
        \max_{t \in [0,T]}
        \|X^{\overline{M}}_N(t) - X^M_N(t)\|_{L^4(\Omega, H)}& \|X^M_K(t) - X^M_N(t)\|_{L^4(\Omega, H)}\\
        &\le
        \begin{cases}
          C \min(M^{-\kappa},\lambda_{N}^{-\kappa/2} M^{-\kappa/2}) & \text{if} \quad \kappa \in (0,1/2),\\
          C \min( \lambda_{N}^{-(2 \kappa -1)/2} M^{-1/2}, \lambda_{N}^{-\kappa/2} M^{-\kappa/2}) & \text{if} \quad \kappa \in [1/2,1),\\
            C \lambda_{N}^{-\kappa/2} M^{-1/2} & \text{if} \quad \kappa \in [1,2)\,.
        \end{cases}
        \end{split}
        \]
        The last two inequalities and~\eqref{eq:doubleDiffPsiProofIneq1} verify~\eqref{eq:strong-error-coupling}.

        Next, for the setting with $\psi \in \cL(H,U)$, it holds for the constant $c_\psi = \|\psi\|_{\cL(H,U)}<\infty$
        and all $t \in [0,T]$ that
        \[
        \begin{split}
          \| \psi(X_K^{\overline{M}}(t))) & - \psi(X_N^{\overline{M}}(t))  - \psi(X_K^M(t)) + \psi(X_N^M(t))\|_{L^2(\Omega, U)}^2 \\
          &\le
        c_\psi      \| X_K^{\overline{M}}(t))  - X_N^{\overline{M}}(t)  - X_K^M(t) + X_N^M(t)\|_{L^2(\Omega, H)}^2\\
        &\le c_\psi \left\| \sum_{i=1}^3 \mathrm{I}^i_F(t )  + \mathrm{I}^1_W(t) + \mathrm{I}^2_W(t)\right\|_{L^2(\Omega, H)} \, .
        \end{split}
        \]
        And the extension of Theorem~\ref{thm:second-order-diff} described in~\eqref{eq:extensionSecOrdDiff}
        verifies~\eqref{eq:strong-error-coupling2}.
\end{proof}
\section{Positive and negative norm estimates for nonlinear composition mappings}
\label{sec:appendix}

In this appendix we derive positive and negative norm estimates for mappings of the type
\begin{equation}
	\label{eq:composition-mapping}
	F(u)(x) = f(u(x)), \quad \text{for a.e. } x \in \cD
\end{equation}
and a sufficiently smooth function $f \colon \R \to \R$, which are relevant for Assumption~\ref{ass:reg} in the main part of the paper. Results like these are commonly used in the numerical SPDE literature, for example in~\cite[Example~3.2]{W16} and \cite[Appendix~A]{KLP20}. However, these results are too restrictive for our needs. In particular, the results applicable when the regularity parameter $\kappa \in (1,2)$ can not, as far as we are aware, be found in the literature.

The derivation is done in the context of Example~\ref{ex:heat}, where $A = - \Delta$ with homogeneous zero Dirichlet boundary conditions, but the arguments could be modified to accommodate more general symmetric elliptic operators. We recall that $H=L^2(\cD)$ for a bounded domain $\cD \subset \R^d$, $d=1,2,3$, which is either convex or has $\cC^2$ boundary $\partial \cD$. It is equipped with the standard $L^2$ inner product $\langle \cdot ,  \cdot \rangle$ and induced norm $\| \cdot \|$.

We recall (see, e.g., \cite[Chapter~1]{AP95}) that when $f$ is once continuously differentiable with a bounded derivative $f'$, the G\^ateaux derivative of $F \in \cG^1(H,H)$ is given by
\begin{equation*}
	(F'(u)v)(x) = f'(u(x))v(x), \quad \text{for }u,v \in H \text{ and a.e. } x \in \cD.
\end{equation*}

In light of this, the following result (corresponding to Assumption~\ref{ass:reg}\ref{ass:derivative-bounded}) is immediate.
\begin{proposition}
	\label{prop:F-Fprime-H-bound}
	Let $f$ be continuously differentiable with a bounded derivative $f'$. Then, there is a constant $C < \infty$ such that $\|F'(u)v\|\le C\|v\|$ for all $u,v \in H$.
\end{proposition}

Next, we need to recall some results on Sobolev spaces. We employ the notation $W^{r,p}:= W^{r,p}(\cD)$, $r \in [0,\infty), p \in (1,\infty),$ for the Sobolev space equipped with the norm
\begin{equation*}
	\|u\|_{W^{r,p}} := \Big(\|u\|_{W^{m,p}}^p + \sum_{|\alpha| = m} \int_{\cD \times \cD} \frac{|D^\alpha u(x) - D^\alpha u(y)|^p}{|x-y|^{d+p\sigma}}\dd x \dd y\Big)^{1/p},
\end{equation*}
for non-integer $r = m + \sigma$, $m \in \N_0, \sigma \in (0,1)$, where $\|\cdot\|_{W^{m,p}}$ is the standard integer order Sobolev norm and $|\cdot|$ the Euclidean norm on $\R^d$, see also~\cite[Definition~1.3.2.1]{G85}.
We will make use of a Sobolev embedding theorem, namely
\begin{equation}
	\label{eq:sobolev-embedding-1}
	W^{s,q} \hookrightarrow W^{r,p}, \quad r \le s, \quad r - d/p \le s - d/q, \quad p,q \in [1,\infty),
\end{equation}
where the arrow notation denotes continuous embedding and either $s \in \N_0$ or $s \neq t$.
We also need a result on multiplication in Sobolev spaces. Specifically, provided that
\begin{enumerate}[label=(\roman*)]
	\item $u \in W^{r_0,p_0}$ and $v \in W^{r_1,p_1}$,
	\item \label{eq:sobolev-multiplication:ii}$s \le \min(r_0,r_1)$,
	\item \label{eq:sobolev-multiplication:iii}$s - d/q \le \min(r_0 - d/p_0,r_1-d/p_1)$,
	\item \label{eq:sobolev-multiplication:iv}$s - d/q < r_0 - d/p_0 + r_1-d/p_1$ and
	\item either $\max(p_0,p_1) \le q$ or the inequalities in \ref{eq:sobolev-multiplication:ii}-\ref{eq:sobolev-multiplication:iii} are strict and instead of \ref{eq:sobolev-multiplication:iv} it holds that $s < r_0 + r_1 - d/\min(p_0,p_1)$,
\end{enumerate}
then $u v \in W^{s,q}$ with
\begin{equation}
	\label{eq:sobolev-multiplication}
	\| u v \|_{W^{s,q}} \le C \| u \|_{W^{r_0,p_0}} \| v\|_{W^{r_1,p_1}}
\end{equation}
for a constant \(C\) that do not depend on \(u\) or \(v\).
For more details on Sobolev spaces and proofs of these results, see \cite{BH21} and the references therein.

The space $W^{r,2}$ is a Hilbert space.  Recalling the notation $\dot{H}^r$ from Section~\ref{sec:notation}, it holds that (see, e.g., \cite[Theorems~16.12-13]{Y10})
\begin{equation}
	\label{eq:sobolev_dot_equivalence}
	\dot{H}^r =
	\begin{cases}
		W^{r,2} \text{ if } r \in [0,1/2), \\
		\left\{u \in W^{r,2} : u = 0 \text{ a.e.\ on } \partial \cD \right\} \text{ if } r \in (1/2,3/2) \cup (3/2,2],
	\end{cases}
\end{equation}
with norm equivalence. In the case that $\cD$ has $\cC^2$ boundary, the latter equivalence holds also for $s=3/2$. In either case, it is true that $\dot{H}^s \hookrightarrow W^{s,2}$ also for $s \in \{1/2,3/2\}$.

We can now consider a regularity transfer property, corresponding to Assumption~\ref{ass:reg}\ref{ass:difference-regularity-transfer}.
\begin{proposition}
	Suppose that $f$ is once continuously differentiable with a bounded derivative. Then, for all $\kappa \in [0,1] \setminus \{1/2\}$, there is a constant $C < \infty$ such that
	\begin{equation*}
		\| F(u)-F(v) \|_{\dot{H}^{\kappa}} \le C \big(\| v \|_{\dot{H}^{\kappa}} +  \| u \|_{\dot{H}^{\kappa}}\big) \text{ for all } u, v \in \dot{H}^{\kappa}.
	\end{equation*}
	Suppose instead that $\kappa \in (\max(1,d/2),2], f \colon \R \to \R$ is twice continuously differentiable with bounded derivatives. Then, there is a constant $C < \infty$ such that
	\begin{equation*}
		\| F(u)-F(v) \|_{\dot{H}^{\kappa}} \le C \big(1 + \| v \|^2_{\dot{H}^{\kappa}} +  \| u \|^2_{\dot{H}^{\kappa}}\big) \text{ for all } u, v \in \dot{H}^{\kappa}.
	\end{equation*}
\end{proposition}
\begin{proof}
	As in several of the proofs in the main part of the paper, we let here (and in the forthcoming results) $C< \infty$ denote a generic constant which may change from line to line but does not depend on \(u\) or \(v\).
	We first note that $F(u)-F(v) = 0$ a.e.\ on $\partial \cD$ as long as both $u$ and $v$ fulfill this property. Therefore, it suffices to show that
	\begin{equation*}
		\| f(u) \|_{W^{\kappa,2}} \le C \| u \|^{\lceil \kappa \rceil}_{W^{\kappa,2}}
	\end{equation*}
	for $\kappa > 0$. When $\kappa < 1$, this follows directly from the Lipschitz property of $f$. For the case that $\kappa = 1$, the chain rule in Sobolev spaces yields that
	\begin{equation*}
		D^{x_i} f(u) = f'(u) D^{x_i} u, \quad i \in \{1,2,3\}.
	\end{equation*}
	Since $f'$ is bounded, this yields $\|f(u)\|_{W^{1,2}} \le C \|u\|_{W^{1,2}}$.

	In light of the definition of the norm in $W^{\kappa,2}$, it only remains to show that
	\begin{equation}
		\label{eq:app:prop:2}
		\| f'(u) D^{x_i} u \|_{W^{\kappa-1,2}} \le C \| u \|^2_{W^{\kappa,2}}.
	\end{equation}
	where $i \in \{1,2,3\}$ and $\kappa \in (\max(d/2,1),2)$. When $d=1$, it follows directly from~\eqref{eq:sobolev-multiplication} that $$\| f'(u) D^{x_i} u \|_{W^{\kappa-1,2}} \le C \| D^{x_i} u \|_{W^{\kappa-1,2}} \| f'(u) \|_{W^{1,2}}.$$ We can therefore assume $d \ge 2$ and note that in this case
	\begin{equation}
		\label{eq:app:prop:3}
		\| D^{x_i} u \|_{W^{\kappa-1,d/(1+d/2-\kappa)}} \le \| u \|_{W^{1,d/(1+d/2-\kappa)}} \le C \| u \|_{W^{\kappa,2}}.
	\end{equation}
	Next, the fact that $f'$ is bounded yields $\| f'(u) D^{x_i} u \| \le C \| u \|_{W^{1,2}} \le C \| u \|_{W^{\kappa,2}}$. For the remaining part of the norm $\| f'(u) D^{x_i} u \|_{W^{\kappa-1,2}}$, we split the integral into two parts via
	\begin{align*}
		&f'(u(x)) (D^{x_i} u)(x) - f'(u(y)) (D^{x_i} u)(y) \\
		&\quad = f'(u(x)) \Big((D^{x_i} u)(x) - (D^{x_i} u)(y)\Big) + (D^{x_i} u)(y) \Big(f'(u(x)) - f'(u(y))\Big).
	\end{align*}
	The first part is bounded by a constant times $\| D^{x_i} u \|_{W^{\kappa-1,2}}$. For the second, we write $\kappa_d := d + 2(\kappa-1)$ and consider $\theta_1,\theta_2 \in \R$ satisfying $\theta_1+\theta_2 = 1$. After using the fact that $f'$ is Lipschitz, H\"older's inequality yields
	\begin{align*}
		&\int_{\cD \times \cD} \frac{|(D^{x_i} u)(y)|^2 |u(x) - u(y)|^2}{|x-y|^{\kappa_d}}\dd x \dd y \\
		&\quad\le \bigg(\int_{\cD \times \cD} \frac{|(D^{x_i} u)(y)|^{\frac{d}{1+d/2-\kappa}}}{|x-y|^{\frac{d\theta_1\kappa_d}{2+d-2\kappa}} }\dd x \dd y\bigg)^{\frac{2 + d - 2\kappa}{d}} \times \bigg(\int_{\cD \times \cD} \frac{|u(x) - u(y)|^{\frac{d}{\kappa-1}}}{|x-y|^{d+(\frac{d\theta_2\kappa_d }{2(\kappa-1)}-d)}}\dd x \dd y\bigg)^{\frac{2(\kappa-1)}{d}}.
	\end{align*}
	We now choose $\theta_1 < (2 + d - 2 \kappa)/\kappa_d$
	which necessitates $\theta_2 = 1 - \theta_1 > 4(\kappa-1)/\kappa_d$. Then, the first integral is bounded by a constant times
	\begin{align*}
		\sup_{y \in \cD} \bigg(\int_{\cD} |x-y|^{-\frac{d\theta_1\kappa_d}{2+d-2\kappa}} \dd x \bigg)^{\frac{2 + d - 2\kappa}{d}} \| D^{x_i} u \|^2_{W^{\kappa-1,d/(1+d/2-\kappa)}} \le C \| u \|^2_{W^{\kappa,2}}
	\end{align*}
	where we used~\eqref{eq:app:prop:3} and the fact that $\theta_1 \kappa_d/(2+d-2\kappa) < 1$. The second integral is bounded by $\| u \|^2_{W^{ \theta_2 \kappa_d /2 - (\kappa-1) ,d/(\kappa-1)}}$. The fact that this quantity is in turn bounded by $\| u \|^2_{W^{\kappa,2}}$ is a consequence of~\eqref{eq:sobolev-embedding-1}. This is applicable when
	\begin{equation*}
		\frac{\kappa-1}{d} \Big( \frac{d \theta_2 \kappa_d}{2(\kappa-1)} - 2 d \Big) \le \kappa - \frac{d}{2} \iff \theta_2 \le \frac{4 (\kappa - 1)}{\kappa_d} + \frac{2 \kappa - d}{\kappa_d},
	\end{equation*}
	and since the right hand side of this inequality is strictly larger than $4(\kappa-1)/\kappa_d$ we can apply~\eqref{eq:sobolev-embedding-1} by choosing $\theta_2$ sufficiently small. This concludes the derivation of \eqref{eq:app:prop:2}.
	\end{proof}
Note that the case that $d = 3$ and $\kappa \in (1,3/2)$ is not included in the statement above. We are not aware of a way to derive such a result. The same restriction applies to a negative norm bound on $F'$, corresponding to Assumption~\ref{ass:reg}\ref{ass:derivative-negnorm}, which we derive next.
\begin{proposition}
	\label{prop:F-Fprime-neg-bound}
	Let $\eta \in (d/2,2]$. If $\kappa \in [0,\min(\eta,1)] \setminus \{1/2\}$, $f \colon \R \to \R$ is twice continuously differentiable with bounded derivatives and $u \in \dot{H}^\kappa$, the operator $F'(u)$ extends to $\cL(\dot{H}^{-\kappa},\dot{H}^{-\eta})$. Moreover, there is a constant $C < \infty$, independent of $u \in \dot{H}^\kappa$, such that
	\begin{equation*}
		\| F'(u) \|_{\cL(\dot{H}^{-\kappa},\dot{H}^{-\eta})} \le C \| u \|_{\dot{H}^{\kappa}}.
	\end{equation*}
	If instead $\kappa \in (\max(1,d/2),\eta] \setminus \{3/2\}, f \colon \R \to \R$ is three times continuously differentiable with bounded derivatives, $d \le 2$ and $u \in \dot{H}^\kappa$, the operator $F'(u)$ extends to $\cL(\dot{H}^{-\kappa},\dot{H}^{-\eta})$. Moreover, there is a constant $C < \infty$, independent of $u \in \dot{H}^\kappa$, such that
	\begin{equation*}
		\| F'(u) \|_{\cL(\dot{H}^{-\kappa},\dot{H}^{-\eta})} \le C (1 + \| u \|_{\dot{H}^{\kappa}}^2).
	\end{equation*}
\end{proposition}
\begin{proof}
	We first note that by the same arguments as in the proof of the previous proposition, $\|f'(u)\|_{W^{\kappa,2}} \le C \|u\|_{W^{\kappa,2}}$ when $\kappa \in [0,1]$ while $\|f'(u)\|_{W^{\kappa,2}} \le C (1 + \|u\|^2_{W^{\kappa,2}})$ when $\kappa \in (\max(1,d/2),2)$.

	Consider now the case that $\kappa \in [0,1]$. By~\eqref{eq:sobolev-multiplication}, $\| F'(u) v \|_{W^{\kappa,2}} \le C \| u \|_{W^{\kappa,2}} \| v \|_{W^{\eta,2}}$. Thus,
	\begin{equation*}
		\| A^{\frac{\kappa}{2}} F'(u) A^{-\frac{\eta}{2}}\|_{\cL(H)} =  \| F'(u)\|_{\cL(\dot{H}^{\eta},\dot{H}^{\kappa})} \le C \| F'(u)\|_{\cL(W^{\eta,2},W^{\kappa,2})} \le C \| u \|_{W^{\kappa,2}} \le C \| u\|_{\dot{H}^\kappa}
	\end{equation*}
	for $u \in \dot{H}^\kappa$, where we made use of \eqref{eq:sobolev_dot_equivalence} and the fact that if $v = 0$ a.e.\ on $\partial \cD$, then so is $F'(u)v$. Since $F'(u)$ is symmetric on $H$, we obtain for $v \in H$ that
	\begin{align*}
		\| A^{-\frac{\eta}{2}} F'(u) v \| &= \sup_{\substack{w \in H \\ \| w \|=1}} |\langle A^{-\frac{\eta}{2}} F'(u) v, w \rangle| = \sup_{\substack{w \in H \\ \| w \|=1}} |\langle  A^{-\frac{\kappa}{2}} v, A^{\frac{\kappa}{2}} F'(u) A^{-\frac{\eta}{2}}w \rangle| \\
		&\le \| v \|_{\dot{H}^{-\kappa}} \| A^{\frac{\kappa}{2}} F'(u) A^{-\frac{\eta}{2}}\|_{\cL(H)}
	\end{align*}
	and since $H$ is dense in $\dot{H}^{-\kappa}$, this completes the proof of the first statement. The proof of the second statement is analogous.\end{proof}

Finally, we mention a negative norm result on the second derivative of $F$. For this we recall that when $f$ is twice continuously differentiable with bounded derivatives, then $F \in \cG^2(H,\dot{H}^{-\eta})$ for $\eta \in (d/2,2]$. In fact, $F''(u)(v_1,v_2)$ is even an element of $L^1(\cD)$, given by
\begin{align*}
	(F''(u)(v_1,v_2))(x) = f''(u(x))v_1(x)v_2(x) , \quad \text{for }u,v_1,v_2 \in H \text{ and a.e. } x \in \cD.
\end{align*}
The proof is essentially the same as in \cite[Example~3.2]{W16} and therefore omitted.

\begin{proposition}
	\label{prop:F-Fdoubleprime}
	Suppose that the mapping $F$ is given by \eqref{eq:composition-mapping} and let $\eta \in (d/2,2]$. If $f \colon \R \to \R$ is twice continuously differentiable with bounded derivatives, there is a constant $C < \infty$ such that
	\begin{equation*}
		\| F''(u) (v_1,v_2) \|_{\dot{H}^{-\eta}} \le C \| v_1 \| \| v_2 \| \text{ for all } u, v_1, v_2 \in H.
	\end{equation*}
\end{proposition}
\end{document}